\documentclass[11pt,reqno]{amsart}
\usepackage{amscd,amsmath,amsopn,amssymb,amsthm,multicol,amsbsy}
\usepackage{tikz,subdepth,anysize,verbatim,ifthen,xargs,colortbl,float}
\usepackage{longtable,mathtools}
\usepackage{todonotes}
\usepackage[english]{babel}
\usepackage[utf8]{inputenc}
\usepackage{enumerate}
\usepackage{enumitem}
\usepackage{caption}
\usepackage{subcaption}
\usepackage{array,color}
\usepackage{enumitem}
\usepackage{tgpagella}
\usepackage{tikz}
\usepackage[utf8]{inputenc}
\usepackage[margin=2.5cm]{geometry}
\usepackage[unicode]{hyperref}
\usepackage{framed}
\usepackage{xcolor}
\usepackage{rotating}
\usepackage[normalem]{ulem}
\allowdisplaybreaks

\hypersetup{
  pdftitle   = {\texorpdfstring{Geometric realization of the Lie superalgebra $\D(2,1;a)$}{}},
  pdfkeywords = {},
  pdfauthor  = {Anna Escofet, Boris Kruglikov, Dennis The},
  pdfcreator = {\LaTeX\ with package \flqq hyperref\frqq}
}

\usetikzlibrary{positioning}
\definecolor{mygray}{gray}{0.7}
\usetikzlibrary{tikzmark}
\usetikzlibrary{shapes,snakes}
\usetikzlibrary{positioning}
\tikzset{blau node/.style={rectangle,rounded corners,fill=blue!20,draw,minimum size=1cm,inner sep=0pt},
            }
            \tikzset{marro node/.style={rectangle,rounded corners,fill=brown!20,draw,minimum size=1cm,inner sep=0pt},
            }
            \tikzset{lila node/.style={rectangle,rounded corners,fill=purple!20,draw,minimum size=1cm,inner sep=0pt},
            }
\tikzset{verd node/.style={rectangle,rounded corners,fill=green!20,draw,minimum size=1cm,inner sep=0pt},
            }
\tikzset{vermell node/.style={rectangle,rounded corners,fill=red!20,draw,minimum size=1cm,inner sep=0pt},
            }
\tikzset{taronja node/.style={rectangle,rounded corners,fill=orange!20,draw,minimum size=1cm,inner sep=0pt},
            }
\tikzset{groc node/.style={rectangle,rounded corners,fill=yellow!20,draw,minimum size=1cm,inner sep=0pt},
            }
\tikzset{magenta node/.style={rectangle,rounded corners,fill=magenta!20,draw,minimum size=1cm,inner sep=0pt},
            }
            \usetikzlibrary{shapes.misc}
\tikzset{cross/.style={cross out, draw=black, minimum size=2*(#1-\pgflinewidth), inner sep=0pt, outer sep=0pt},
cross/.default={3pt}}

\newcommand\pr{\operatorname{pr}}

\newcommand\im{\operatorname{im}}
\newcommand{\D}{\operatorname{D}}
\newcommand\fg{\mathfrak{g}}
\newcommand\ff{\mathfrak{f}}
\newcommand\fp{\mathfrak{p}}
\newcommand\fq{\mathfrak{q}}
\newcommand\fb{\mathfrak{b}}
\newcommand\fh{\mathfrak{h}}
\newcommand\fk{\mathfrak{k}}
\newcommand\fm{\mathfrak{m}}
\newcommand\fsl{\mathfrak{sl}}
\newcommand\CC{\mathbb{C}}
\newcommand\NN{\mathbb{N}}
\newcommand\ZZ{\mathbb{Z}}
\newcommand\PP{\mathbb{P}}
\newcommand\VV{\mathbb{V}}
\newcommand\UU{\mathbb{U}}
\newcommand\fgl{\mathfrak{gl}}

\newcommand\gl{\mathfrak{gl}}
\renewcommand\sl{\mathfrak{sl}}
\newcommand\so{\mathfrak{so}}

\newcommand{\niplus}{\ni\hspace{-1.1em}+}

\newcommand\op[1]{\mathop{\rm #1}\nolimits}
\newcommand\pd{\partial}

\newcommand{\comm}[1]{}
\newcommand\I{{\rm{I}}}

\newcommand\IV{{\rm{IV}}}

\newcommand\cch{{\footnotesize\raisebox{1pt}{$\chi$}}}

\newcommand\C{{\mathbb C}}
\newcommand\g{{\mathfrak g}}

\newcommand\m{{\frak m}}
\newcommand\p{\partial}
\newcommand\cC{\mathcal{C}}
\newcommand\cD{\mathcal{D}}
\newcommand\hcD{\hat{\mathcal{D}}}

\newtheorem{theorem}{Theorem}[section]
\newtheorem{proposition}[theorem]{Proposition}
\newtheorem{lemma}[theorem]{Lemma}
\newtheorem{example}[theorem]{Example}
\newtheorem{definition}[theorem]{Definition}
\newtheorem{remark}[theorem]{Remark}
\newtheorem{corollary}[theorem]{Corollary}

\numberwithin{equation}{section}

\begin{document}

\title[Geometric realizations of the Lie superalgebra $D(2,1;a)$]{Geometric realizations of the Lie superalgebra 
$D(2,1;a)$}

\author{Anna Escofet, Boris Kruglikov, Dennis The}

\address{Department of Mathematics and Statistics, UiT The Arctic University of Norway, Troms\o\  90-37,Norway}

\email{\quad anna.e.pacreu@uit.no; \quad boris.kruglikov@uit.no; \quad dennis.the@uit.no}

\thanks{}

\begin{abstract}
For every parabolic subgroup $P$ of a Lie supergroup $G$, the homogeneous superspace $G/P$ carries a 
$G$-invariant supergeometry. We address the question whether $\fg=\text{Lie}(G)$ is the maximal supersymmetry 
of this supergeometry in the case of the exceptional Lie superalgebra $D(2,1;a)$. 
For each choice of parabolic $\fp\subset\fg$, we consider the corresponding negatively graded Lie subalgebra $\fm\subset\fg$, 
and compute its Tanaka--Weisfeiler prolongations, with reduction of the structure group when required,
thus realizing $D(2,1;a)$ via symmetries of supergeometries. 

This gives 6 inequivalent supergeometries:
one of these is a vector superdistribution, two are given by cone fields of supervarieties, 
and the remaining three are higher order structure reductions (a novel feature). 
We describe those super\-geometries and realize $D(2,1;a)$ supersymmetry explicitly in each case. 
\end{abstract}

\maketitle

\section{Introduction and main results}

In 1888, W.\ Killing classified all complex finite-dimensional simple Lie algebras, finding five exceptional 
simple Lie algebras in addition to the known matrix algebras. While the classical series $A_n$, $B_n$, $C_n$, $D_n$ 
come naturally as symmetries of a tensor on a vector space, the exceptional ones appeared abstractly. 
A concrete geometric realization for the simplest of these, the Lie algebra $G_2$, was achieved by E.\ Cartan \cite{Cartan} 
and F.\ Engel \cite{Engel} in 1893. In modern terms, this corresponds to the symmetry algebra of Klein geometries 
$G_2/P_1$, $G_2/P_{12}$ and $G_2/P_2$.
Realizations of $F_4$, $E_6$, $E_7$, $E_8$ are due to A. Borel, H. Freudental, J. Tits and others in early 1950s, see \cite{Fr,Ti}
and also \cite{The} for a universal approach to geometric realization of all simple Lie algebras.

In 1977,  V.\ Kac classified all complex finite-dimensional simple Lie superalgebras \cite{Kac}. Again, 
the classical series $A(m,n)$, $B(m,n)$, $C(m,n)$, $D(m,n)$, as well as two ``strange'' families $P(n)$ 
and $Q(n)$, come naturally as symmetries of a tensor on a supervector space. 
The so-called Cartan superalgebras $W(n)$ are all vector fields on an odd vector space $\CC^{0|n}$.
Two exceptional Lie superalgebras $G(3)$ and $F(4)$ 
were realized geometrically in \cite{G3super}, \cite{G3F4A} and \cite{F4super}.
The aim of the present paper is to provide a geometric realization for the remaining case:
a one-parameter family of exceptional Lie superalgebras $D(2,1;a)$ that are deformations of $D(2,1)$.

The Lie superalgebra $\fg=D(2,1;a)$ has 4 different simple root systems up to the Weyl group action \cite{Frappat,Serganova}. In other words, there are four different Borel subalgebras. 
The corresponding Dynkin diagrams are labeled with roman numerals. 
The first three are outer equivalent with a change of the parameter $a$, so we will use only DD-I and DD-IV, see Figure \ref{rootsystems}.
For each Dynkin diagram labeled $\Xi$, the table there indicates the corresponding simple root system 
$\Pi_\Xi= \{\alpha_1,\alpha_2,\alpha_3\}$. The parabolic $\fp^\Xi_\chi$ corresponds to crosses on 
nodes $\chi$, and it is generated by the Borel subalgebra and negative root vectors $e_{-\alpha}$ for non-crossed simple 
roots $\alpha$. This gives the grading 
 \begin{equation}\label{grading}
\fg=\underbrace{\fg^{\Xi,\cch}_{-\nu}\oplus\dots\oplus\fg^{\Xi,\cch}_{-1}}_{\fm^{\Xi}_{\cch}}
\oplus\underbrace{\fg^{\Xi,\cch}_0\oplus\fg^{\Xi,\cch}_{1}\dots\oplus \fg^{\Xi,\cch}_{\nu}}_{\fp^{\Xi}_{\cch}}.
 \end{equation}

The Tanaka--Weisfeiler prolongation \cite{Tanaka,Kac,G3super} of $\fm=\fg_-$ is the maximal graded 
Lie superalgebra $\op{pr}(\fm)=\oplus_i\op{pr}_i(\fm)$ with the negative part $\fm$, for which 
the only trivial $\fm$-submodule is $\mathfrak{z}(\fm)=\fg_{-\nu}$. 
For simple Lie superalgebras and any choice of parabolic (or grading as above) $\op{pr}(\fm)\supseteq\fg$.
We are interested in assessing whether $\op{pr}(\fm)=\fg$. In the case $\op{pr}_0(\fm)\neq\fg_0$,
one can define the restricted prolongation $\op{pr}(\fm,\fg_0)$ in a similar manner, and more generally 
one can define the algebraic prolongation $\op{pr}(\fg_{\leq k})$ for any $k\geq0$.
From an algebraic viewpoint, this prolongation generalizes the extension by derivations, and indeed
$\op{pr}_0(\fm)=\mathfrak{der}_0(\fm)$.

 \begin{theorem}\label{T1}
Up to isomorphism, there are 6 parabolic superalgebras of $\fg=D(2,1;a)$, indicated below. 
Figure \ref{6geometries} summarizes the algebraic data on the grading and the Tanaka--Weisfeiler prolongation.
 \end{theorem}

The prolongation result for the depth-1 case $\fp_2^{\rm I}$ (Cartan prolongation) is due to \cite{Poletaeva}.

\begin{figure}[h]
\[
 \begin{array}{|c|c|c|c|c|} \hline
\vphantom{\frac{A^A}{A^A}}
\mbox{Parabolics} & \nu & \dim(\fg_0),\dots,\dim(\fg_{-\nu})& \fg_0 & \mbox{Prolongations}\\[2pt]
\hline\hline
\vphantom{\frac{A^A}{A^A}}
\fp_1^{\rm I},\,\fp_2^{\rm II},\,\fp_3^{\rm III} & 2 & (7|0,0|4,1|0) & \mathfrak{co}(4)  & 
\pr(\fm,\fg_0)\neq\fg,\quad\pr(\fg_{\leq1})=\fg \\[2pt]
\hline
\begin{array}{c}
\vphantom{\frac{A^A}{A^A}}
\fp_2^{\rm I},\,\fp_3^{\rm I},\,\fp_1^{\rm II},\,
\fp_3^{\rm II},\,\fp_1^{\rm III}\\ \fp_2^{\rm III},\,
\fp_1^{\rm IV},\,\fp_2^{\rm IV},\,\fp_3^{\rm IV}
\end{array}
  & 1 & (5|4,2|2) & \gl(2|1)& \pr(\fm)\neq\fg,\quad\pr(\fm,\fg_0)=\fg \\[2pt]
\hline
\begin{array}{c}
\vphantom{\frac{A^A}{A^A}}
\fp_{12}^{\rm I},\,\fp_{13}^{\rm I},\,\fp_{12}^{\rm II}\\
\fp_{23}^{\rm II},\,\fp_{13}^{\rm III},\,\fp_{23}^{\rm III}
\end{array}
 & 3 
& (5|0,1|2,0|2,1|0) & \gl(2)\oplus\CC &  \pr(\fm,\fg_0)\neq\fg,\quad\pr(\fg_{\leq1})=\fg \\[2pt]
\hline
\begin{array}{c}
\vphantom{\frac{A^A}{A^A}}
\fp_{23}^{\rm I},\,\fp_{13}^{\rm II},\,\fp_{12}^{\rm III},\,\\
\fp_{12}^{\rm IV},\,\fp_{13}^{\rm IV},\,\fp_{23}^{\rm IV} 
\end{array}
& 2 & (3|2,2|2,1|1) &  \gl(1,1)\oplus\CC  & \pr(\fm)\neq\fg,\quad\pr(\fm,\fg_0)=\fg \\[2pt]
\hline
\vphantom{\frac{A^A}{A^A}}
\fp_{123}^{\rm I},\,\fp_{123}^{\rm II},\,\fp_{123}^{\rm III} & 4 & (3|0,2|1,0|2,0|1,1|0) & 
\CC^{3|0} & \pr(\fm,\fg_0)\neq\fg,\quad\pr(\fg_{\leq1})=\fg \\[2pt]
\hline
\vphantom{\frac{A^A}{A^A}}
\fp_{123}^{\rm IV} & 3 & (3|0,0|3,3|0,0|1) & \CC^{3|0} & \pr(\fm)=\fg\\
\hline
  \end{array}
 \]
\caption{Parabolic subalgebras of $D(2,1;a)$, depth $\nu$, dimensions and structure algebras,  
prolongations and necessary reductions.}
 \label{6geometries}
\end{figure}

The computation of the above algebraic prolongation is closely related to the computation of the first 
(generalized) Spencer cohomology. This latter is the cohomology of the bi-graded Chevalley–Eilenberg complex
$\fg\otimes\Lambda(\fm^*)$ with the differential $\pd:\fg\otimes\Lambda^n\fm^*\to\fg\otimes\Lambda^{n+1}\fm^*$
preserving the grading associated with a choice of parabolic $\fp$.
Lower degree cohomology groups are presented in Table \ref{SpH}. The meaning of those is the following.

The zeroth cohomology is $H^0(\fm,\fg)=\mathfrak{z}(\fm)=\fg_{-\nu}$. 
The first cohomology controls the prolongation as follows \cite{KST2}: 
if $H^1(\fm,\fg)_{\geq0}=0$ then $\pr(\fm)=\fg$, if $H^1(\fm,\fg)_+=0$ then $\pr(\fm,\fg_0)=\fg$, 
and if $H^1(\fm,\fg)_i=0$ for $i\geq k>0$ then $\pr(\fg_{\geq k})=\fg$. 
The second cohomology $H^2(\fm,\fg)$ is often identified with the space of curvatures 
\cite{LPS,Poletaeva}, but this is not fully justified geometrically, see a discussion in \cite{G3super}.
We provide the corresponding data for completeness and future references. 

\begin{table}[h]\centering
\begin{tabular}{c|ccc}
& \quad $H^0(\fm,\fg)$ \quad & \quad $H^1(\fm,\fg)$ \quad & \quad $H^2(\fm,\fg)$ \quad \\[2pt]
    \hline\\[-8pt]
$\fm_1^\I$ & $\CC_{-2}^{1|0}$ & $\CC_1^{0|4}$ & $\CC_2^{9|0}$ \\[3pt]
$\fm_2^\I$ & $\CC_{-1}^{2|2}$ & $\CC_0^{3|4}$ & $\CC_2^{4|4}$  \\[3pt]
$\fm_{12}^\I$ & $\CC_{-3}^{1|0}$ & $\CC_{-2}^{1|0}\oplus\CC_1^{0|2}$ & $\CC_2^{3|0}\oplus\CC_3^{0|2}$  \\[3pt]
$\fm_{23}^\I$ & $\CC_{-2}^{1|1}$ & $\CC_{-1}^{2|2}\oplus\CC_0^{0|1}$ & $\CC_0^{1|1}\oplus\CC_2^{2|2}$  \\[3pt]
$\fm_{123}^\I$ & $\CC_{-4}^{1|0}$ & $\CC_{-3}^{2|0}\oplus\CC_1^{0|1}$ & $\CC_{-2}^{1|0}\oplus\CC_2^{1|0}\oplus\CC_3^{0|2}$  \\[3pt]
$\fm_{123}^\IV$ & $\CC_{-3}^{0|1}$ & $\CC_{-1}^{0|3}$ & $\CC_0^{1|0}\oplus\CC_2^{3|0}$  \\[3pt]\end{tabular}
\caption{The Spencer cohomology $H^{i,j}=H^j(\fm,\fg)_i$ of order $j\leq2$ 
for different choices of parabolics $\fp$, with indicated grading $i$ and super dimension.}\label{SpH}
\end{table}

From a geometric viewpoint, the Tanaka--Weisfeiler algebraic prolongation $\fg$ of $\fm$ or $(\fm,\fg_0)$ 
restricts the symmetry of the underlying geometric structure $(M,\mathcal{D},q)$ as follows 
\cite[Theorem 1.1]{KST2} (inequality for both even and odd dimensions of the corresponding superalgebras): 
 \begin{equation}\label{symineq}
\dim\mathfrak{sym}(M,\mathcal{D},q) \leq \dim\fg.
 \end{equation}
Here $M$ is a supermanifold with a vector superdistribution $\mathcal{D}$, corresponding to $\fg_{-1}$,
and $q$ is either a geometric structure on $\mathcal{D}$ or a structure reduction, which has the symbol given by 
$\fm$ and possibly $\fg_0$. The inequality in \eqref{symineq} also holds in the case of higher order reduction, 
where $\fg$ is the prolongation of $\fg_{\leq k}$ for some $k>0$.

For $M=G/P$ equipped only with the superdistribution $\mathcal{D}$, 
its full symmetry algebra is given in Table \ref{SymmA}.

 \begin{table}[h]\centering
\begin{tabular}{c|ccc}
& \quad $\pr(\fm)\simeq\mathfrak{sym}(G/P,\mathcal{D})$ \quad \\[2pt]
    \hline\\[-8pt]
$\fm_1^\I,\ \fm_{12}^\I,\ \fm_{123}^\I$ & $\mathfrak{cont}(1|4)$ \\[3pt]
$\fm_2^\I$ & $\mathfrak{vect}(2|2)$  \\[3pt]
$\fm_{23}^\I$ & $\bigl(\mathcal{O}(1|3)\otimes_{\mathcal{O}(1|2)}\mathfrak{cont}(1|2)\bigr)
\ltimes\bigl(\mathcal{O}(1|3)\otimes_{\mathcal{O}(0|1)}\mathfrak{vect}(0|1)\bigr)$  \\[3pt]
$\fm_{123}^\IV$ & $D(2,1;a)$ \\[3pt]\end{tabular}
\caption{Symmetries of $(G/P,\mathcal{D})$ without additional geometric structures.}\label{SymmA}
 \end{table}
It will be shown that in all cases, except for the $|1|$-graded geometry $G/P_2^\I$, the generalized flag variety $M$ 
is related to some contact manifold (so $D(2,1;a)$ will be naturally embedded as a subalgebra
of a contact algebra, as shown in Table \ref{SymmA}, except for $G/P_{23}^\I$ in which case this is true 
only partially) even though only $G/P_1^\I$ is naturally a (odd) contact space.

While \eqref{symineq} restricts the symmetry from above, the flat model $G/P$ carries $\fg$ as the global symmetry,
thus realizing this dimension. In other words, to realize the maximal symmetry dimension we consider the 
corresponding homogeneous superspace $M=G/P$ with the geometric structure according to the above algebraic computation. In this way, with a proper choice of (connected simply-connected) supergroup $G$ with 
$\op{Lie}(G)=\fg=D(2,1;a)$ and the corresponding parabolic subgroup $P$, we obtain:

 \begin{itemize}
\item
$M^{3|4}_\diamond=G/P_{123}^\IV$ with the geometry given by a vector superdistribution $\mathcal{D}$;
\item
$M^{2|2}=G/P_2^\I$ with structure reduction given by a supervariety $\mathcal{V}\subset\mathbb{P}
(\mathcal{T}M)$ and $M^{3|3}=G/P_{23}^\I$ with structure reduction given by a supervariety 
$\mathcal{W}\subset\mathbb{P}(\mathcal{C})$ for a supervector bundle $\mathcal{E}$ naturally associated 
to $\mathcal{D}$, cf.\ \cite{G3super,F4super};
\item
$M^{1|4}=G/P_1^\I$, $M^{2|4}=G/P_{12}^\I$, $M^{3|4}=G/P_{123}^\I$, each with a higher order reduction
of the frame bundle $F_1\subset\mathcal{F}r_1(M)=\op{Pr}_1(M,\mathcal{D})$, see \cite{KST2}.
 \end{itemize}

 \begin{theorem}\label{T2}
The geometries on $M^{1|4}$, $M^{2|4}$, $M^{3|4}$, $M^{2|2}$, $M^{3|3}$, $M^{3|4}_\diamond$
have the symmetry superalgebra equal to $\fg=D(2,1;a)$ locally and globally. 
The corresponding automorphism supergroup $\op{Aut}(M,\mathcal{D},q)$ has dimension $(9|8)$ and 
its connected component is covered by the Lie supergroup $G$.
 \end{theorem}
 
More precisely, the geometric realizations of $D(2,1;a)$ for these cases are presented in Theorems 
\ref{Tpi1}, \ref{Tpi12}, \ref{Tpi123}, \ref{Tpi2}, \ref{Tpi23}, \ref{Tpiv123} respectively (see also 
Theorem \ref{Tpiv123plus}).

Let us note that in the classical case, for a complex simple Lie group $G$ and parabolic subgroup $P$,
the Lie algebra $\fg=\op{Lie}(G)$ is almost always realized by the symmetry of $G/P$ equipped with 
the vector distribution $\mathcal{D}$ corresponding to $\fg_{-1}$ and a possible reduction of the structure group 
$G_0\subset\op{End}(\mathcal{D})$. The only two exceptions, requiring higher order reductions,
are projective geometry $A_n/P_1$ and contact projective geometry $C_n/P_1$, see \cite{Yamaguchi}. 
For $D(2,1;a)$, half of the cases require higher order reductions, including the full flag variety
$G/B$, where $B=P_{123}^\I$. This makes a sharp difference both with the classical cases \cite{Yamaguchi}
and with the other exceptional Lie superalgebras \cite{G3F4A}.

The paper is organized as follows.
While referring to \cite{L,CCF} for the background on supergeometries and to \cite{Kac,Va} for the Lie superalgebras,
we will recall the basics related to the Lie superalgebra $D(2,1;a)$ in Section \ref{S2}.
Then we compute algebraic prolongations and demonstrate the computation of Spencer cohomology via the
Hochshild--Serre spectral sequence in Section \ref{S3}, justifying Theorem \ref{T1}, Figure \ref{6geometries} and Table \ref{SpH}. 
Computations are similar for different cases, so some technical details are delegated to appendices. 
Then we discuss geometric prolongations and describe the corresponding geometries in Section \ref{S4},
where we also present explicit realizations of $D(2,1;a)$ by vector fields.
Finally, we investigate twistor correspondences between them in Section \ref{S5}, finish the justification of 
Table \ref{SymmA} and make concluding remarks. 

Notations:
We will denote by $\VV_k$ and $\UU_k$ the even, resp.\ odd, irreducible $\sl_2$ modules of highest weight $k$.
Similarly, $\VV_{k,l}$ and $\UU_{k,l}$ will denote the highest weight modules for $\so_4=\sl_2\oplus\sl_2$.
The symbols $S^k$ and $\Lambda^k$ will mean super symmetric and super skew-symmetric tensor powers.

\section{Parabolics of $D(2,1;a)$, gradations and prolongations}\label{S2}

\subsection{The construction of $D(2,1;a)$}

This Lie superalgebra is given by $\fg=\fg_{\bar0}\oplus\fg_{\bar1}$ , where 
 $$
\fg_{\bar0}=\fsl_2\oplus\fsl_2\oplus\fsl_2,\qquad \fg_{\bar{1}} = \CC^2\boxtimes\CC^2\boxtimes\CC^2.
 $$
Each $\fsl_2$ has the standard basis $X,H,Y$ with commutation relations $[H,X]=2X$, $[H,Y]=-2Y$, $[X,Y]=H$. 
Each standard representation $\CC^2$ has coordinates $x,y$ (also: a basis in the space of linear functions)
with the action of $\fsl_2$ given via the operators $x\partial_y$, $x\partial_x-y\partial_y$, $y\partial_x$; thus 
the corresponding basis $X_i,H_i,Y_i$ of the $i$-th copy of $\fsl_2\subset\fg_{\bar0}$ 
acts on $i$-th factor $\CC^2$ of $\fg_{\bar1}$. 

Denote by $\eta_i$ the invariant symplectic forms on the respective copies of $\CC^2$,
normalized by $\eta_i(x,y)=1$, and introduce the isomorphisms 
$\varphi_i:S^2\CC^2\to\fsl_2$ (the $i$-th copy) given by 
$x^2\mapsto X_i$, $-2xy\mapsto H_i$, $-y^2\mapsto Y_i$. 
The odd-odd bracket $[\fg_{\bar{1}},\fg_{\bar{1}}]\subset\fg_{\bar{0}}$ is defined by the formula
 \begin{multline*}
[v_1\otimes v_2\otimes v_3,w_1\otimes w_2\otimes w_3]= 
s_1 \varphi_1(v_1w_1)\eta_2(v_2,w_2)\eta_3(v_3,w_3)\\
+ s_2 \eta_1(v_1,w_1)\varphi_2(v_2w_2)\eta_3(v_3,w_3)
+ s_3 \eta_1(v_1,w_1)\eta_2(v_2,w_2)\varphi_3(v_3w_3)
\end{multline*}
for $v_i,w_i\in\CC^2$ and some $ s_1, s_2, s_3\in\CC$ with $s_1+s_2+s_3=0$ by the Jacobi identity.
(In what follows, we write $xyx$ instead of $x\otimes y\otimes x$, etc.)

This Lie superalgebra will be denoted $\Gamma(s_1,s_2,s_3)$.
It is simple if and only if $ s_1 s_2 s_3\neq0$, see e.g.\ \cite{Scheunert}.
In addition, the Lie superalgebras $\Gamma(s_1,s_2,s_3)$ and $\Gamma(s'_1,s'_2,s'_3)$ are isomorphic 
if and only if $s_i'=\lambda s_{\tau(i)}$, $1\leq i\leq3$, for some $\tau\in S_3$ and $\lambda\in\CC\setminus\{0\}$. 
 
 \begin{definition}\label{d21alpha} 
For $a\in\CC\setminus\{0,-1\}$, $D(2,1;a)$ is the simple Lie superalgebra $\Gamma( s_1, s_2, s_3)$ with 
\begin{equation}\label{normalization}
 s_1=-1-a,\qquad s_2=1,\qquad s_3=a.
\end{equation}
 \end{definition}

\begin{remark}
The symmetric group $S_3$ changes the parameter so: $a\mapsto\frac1a$ and $a\mapsto-1-a$.
Therefore we can regard $a$ as an element of $\PP^1\setminus\{0,-1,\infty\}$ modulo the action of $S_3$. 
 \end{remark}

The Lie superalgebra $D(2,1;a)$ is isomorphic to $D(2,1)=\mathfrak{osp}(4|2)$ when $a\in\{1,-2,-\frac12\}$.

\subsection{An analogue to the Killing form}

For $D(2,1;a)$, the Killing form $\op{str}(\op{ad}_x\op{ad}_Y)$ vanishes, however there exists a nondegenerate 
bilinear form $B$ on $\fg$ that is: 
(i) consistent $B(u,v)=0$ if $|u|\neq|v|$, (ii) supersymmetric $B(u,v)=(-1)^{|u||v|}B(v,u)$, 
(iii) invariant $B([u,v],w)=B(v,[u,w])$. It is unique, up to scale, see \cite[Proposition 5.3.]{Kac}.

To get the explicit form of $B$, denote by $K_i$ the Killing form on the $i$-th copy of $\fsl_2$ in $\fg_{\bar{0}}$ ($1\leq i\leq 3$) and let $A=\eta_1\eta_2\eta_3$. Since $S^2\fg_{\bar{0}}^*$ contains three trivial modules for
$\fg_{\bar0}$ while $\Lambda^2(\fg_{\bar{1}}^*)$ only one, this analog of the Killing form 
$B:\fg\otimes\fg\longrightarrow\CC$ should be a linear combination of $K_i$ and $A$, and indeed a computation gives
 \begin{equation}\label{KFB}
B=\frac{1}{4s_1}K_1+\frac{1}{4s_2}K_2+\frac{1}{4s_3}K_3-A.
 \end{equation}

\subsection{Root systems, Cartan matrices and Dynkin diagrams}

Recall that a Cartan subalgebra $\fh\subset\fg$ is a Cartan subalgebra $\fh$ of $\fg_{\bar0}$.
We adopt the root conventions in \cite[\S 2.19]{Frappat}.
In particular, a root $\alpha$ is even if $\fg_\alpha\cap\fg_{\bar{0}}\neq\emptyset$, and odd 
if $\fg_\alpha\cap\fg_{\bar{1}}\neq\emptyset$; the set of roots is split accordingly 
$\Delta=\Delta_{\bar0}\cup\Delta_{\bar1}\subset\fh^*$, and there is also a splitting 
$\Delta=\Delta^-\cup\Delta^+$ into negative and positive roots. 
A simple root system $\Pi\subset\Delta$ is a basis $\{\alpha_i\}_{i\in\mathcal{A}}$ of $\fh^*$ that $\ZZ_+$-generates $\Delta^+$; here $\mathcal{A}=\{1,2,3\}$.

For the Lie superalgebra $\fg=D(2,1;a)$, we fix $\fh=\langle H_1,H_2,H_3\rangle$ and let $\{\varepsilon_1,\varepsilon_2,\varepsilon_3\}$ be the functionals given by $\varepsilon_i(\sum_{i=1}^3h_iH_i)=h_i$.
The $D(2,1;a)$ root system is (see \cite[\S 2.20]{Frappat}):
 $$
\Delta_{\bar0}=\{\pm2\varepsilon_1,\pm2\varepsilon_2,\pm2\varepsilon_3\}\qquad\text{ and }\qquad\Delta_{\bar1}=\{\pm\varepsilon_1\pm\varepsilon_2\pm\varepsilon_3\}.
 $$

For any $\alpha\in\fh^*$, there exists $h_\alpha\in\fh$ such that $B(h,h_\alpha)=\alpha(h)$ for all $h\in\fh$, 
and we can define a scalar product $\langle\cdot,\cdot\rangle=B_{|\fh\otimes\fh}$ with $\langle\alpha,\beta\rangle=B(h_\alpha,h_\beta)=\alpha(h_\beta)=\beta(h_\alpha)$ \cite{Kac}.  Note that $h_{\varepsilon_k}=s_kH_k$ and $B(\varepsilon_i,\varepsilon_j)=s_i\delta_{ij}$, so that all odd roots are isotropic. 

 \begin{definition}
Associate to a simple root system $\Pi$ the symmetric Cartan matrix by $A_{ij}=\langle\alpha_i,\alpha_j\rangle$ where $\alpha_i,\alpha_j\in\Pi$.  A Cartan matrix $C$ is obtained henceforth by the normalizations:
 \begin{itemize}
  \item[(i)] if $\alpha_j\in\Delta_{\bar{0}}$, rescaling the $j$-th row so that $C_{jj}=2$,
  \item[(ii)] if $\alpha_j\in\Delta_{\bar{1}}$ and $\langle\alpha_j,\alpha_j\rangle\neq0$, rescaling the $j$-th row so that $C_{jj}=1$,
  \item[(iii)] if $\alpha_j\in\Delta_{\bar{1}}$ and $\langle\alpha_j,\alpha_j\rangle=0$, rescaling the $j$-th row freely.
 \end{itemize}
\end{definition}
The Dynkin diagram is associated to a simple root system $\Pi$:
 \begin{itemize}
  \item[(a)] Nodes for simple roots: white for even, black for odd non-isotropic ($C_{ii}\neq0$), grey for odd isotropic ($C_{ii}=0$); for $D(2,1;a)$ there are only white and grey nodes.
  \item[(b)] Edges connect nodes with $C_{ij}\neq0$, with multiplicity $\max\{|C_{ij}|,|C_{ji}|\}$ marked on edges
(multiplicity 0: no edge, multiplicity 1: no mark); for $D(2,1;a)$ edges do not direct. 
  \end{itemize}

If $\alpha\in\Delta_{\bar0}$, then $\langle\alpha,\alpha\rangle\neq0$. The even reflection $r_\alpha$ is defined by 
 \begin{equation*}\label{eq:evenreflection}
r_\alpha(\beta) = \beta-\frac{2\langle\beta,\alpha\rangle}{\langle\alpha,\alpha\rangle}\alpha.
 \end{equation*}
The Weyl group $W=\langle r_\alpha:\alpha\in\Delta_{\bar0}\rangle$ preserves $\Delta_{\bar0}$ and $\Delta_{\bar1}$. 

If $\alpha\in\Delta_{\bar1}$ is an odd isotropic simple root ($\langle\alpha,\alpha\rangle=0$), then the odd 
reflection $r_\alpha$ maps the simple root system $\Pi=\{\alpha_i\}_{i\in\mathcal{A}}$ to 
a $W$-inequivalent simple root system $\Pi'=\{\alpha'_i\}_{i\in\mathcal{A}}$ by the rule (see \cite{Serganova})
 \begin{equation*}\label{eq:oddreflection}
r_\alpha(\beta)= \begin{cases}
  \beta + \alpha, & \langle\alpha, \beta \rangle \neq 0;\\
  \beta, & \langle\alpha, \beta \rangle = 0, \, \beta \neq \alpha;\\
  -\alpha, & \beta = \alpha. \end{cases}
 \end{equation*}
This map extends to a permutation of $\Delta$, but it is not a linear map of $\fh^*$ and it does not preserve 
$\Delta_{\bar0}$ and $\Delta_{\bar1}$. Odd reflections successively generate all inequivalent simple root systems \cite{Serganova}. 

Up to $W$-equivalence, there are four simple root systems for $D(2,1;a)$ displayed in Figure \ref{rootsystems}, 
along with their respective Cartan matrices. Red dashed arrows represent the odd reflections between 
different Dynkin diagrams. (We write $a$ on the edges, while some others write $|a|$.)

\begin{figure}[h]
\begin{center}
\scalebox{1}{
\begin{tabular}{|c|c|c|c|}
	\hline
 \vphantom{$\frac{A^a}{A}$}   I   & II  & III & IV \\ \hline
    \raisebox{-0.2in}{
 \begin{tikzpicture}[remember picture]
 
 \node (A1)[draw, circle, scale=1] {};
 \node [draw,circle,fill=mygray] {};
 \node (B1)[draw, circle, scale=1, above right of=A1, xshift=0.5cm] {};
 \node (C1)[draw, circle, scale=1, below right of=A1, xshift=0.5cm] {};
 
 \node [left of=A1, xshift=0.5cm]{$\alpha_1$};
 \node [right of=B1, xshift=-0.5cm]{$\alpha_2$};
 \node [right of=C1, xshift=-0.5cm]{$\alpha_3$};
 
 \draw (A1) -- (B1);
 \draw (A1) --node[anchor=north, xshift=-0.15cm, yshift=0.05cm]{$a$} (C1);
 \end{tikzpicture}}
     & 
     \raisebox{-0.2in}{
 \begin{tikzpicture}[remember picture]
 
 \node (A2)[draw, circle, scale=1] {};
 \node [draw,circle,fill=mygray] {};
 \node (B2)[draw, circle, scale=1, above right of=A2, xshift=0.5cm] {};
 \node (C2)[draw, circle, scale=1, below right of=A2, xshift=0.5cm] {};
 
 \node [left of=A2, xshift=0.5cm]{$\alpha_2$};
 \node [right of=B2, xshift=-0.5cm]{$\alpha_3$};
 \node [right of=C2, xshift=-0.5cm]{$\alpha_1$};
 
 \draw (A2) --node[anchor=south, xshift=-0.35cm, yshift=-0.1cm]{$1+a$} (B2);
 \draw (A2) -- (C2);
 \end{tikzpicture}}
 &
 \raisebox{-0.2in}{
 \begin{tikzpicture}[remember picture]
 
 \node (A3)[draw, circle, scale=1] {};
 \node [draw,circle,fill=mygray] {};
 \node (B3)[draw, circle, scale=1, above right of=A3, xshift=0.5cm] {};
 \node (C3)[draw, circle, scale=1, below right of=A3, xshift=0.5cm] {};
 
 \node [left of=A3, xshift=0.5cm]{$\alpha_3$};
 \node [right of=B3, xshift=-0.5cm]{$\alpha_2$};
 \node [right of=C3, xshift=-0.5cm]{$\alpha_1$};
 
 \draw (A3) --node[anchor=south, xshift=-0.35cm, yshift=-0.1cm]{$1+a$} (B3);
 \draw (A3) --node[anchor=north, xshift=-0.15cm, yshift=0.05cm]{$a$} (C3);
 \end{tikzpicture}}
 &
 \raisebox{-0.2in}{
 \begin{tikzpicture}[remember picture]
 
 \node (A4)[draw, circle, scale=1] {};
 \node [draw,circle,fill=mygray] {};
 \node (B4)[draw, circle, scale=1, above right of=A4, xshift=0.5cm] {};
 \node [draw,circle,fill=mygray, above right of=A4, xshift=0.5cm] {};
 \node (C4)[draw, circle, scale=1, below right of=A4, xshift=0.5cm] {};
 \node [draw,circle,fill=mygray, below right of=A4, xshift=0.5cm] {};
 
 \node [left of=A4, xshift=0.5cm]{$\alpha_3$};
 \node [right of=B4, xshift=-0.5cm]{$\alpha_1$};
 \node [right of=C4, xshift=-0.5cm]{$\alpha_2$};
 
 \draw (A4) --node[anchor=south, xshift=-0.15cm, yshift=-0.05cm]{$a$} (B4);
 \draw (A4) --node[anchor=north, xshift=-0.35cm, yshift=0.1cm]{$1+a$} (C4);
 \draw (1.25,-0.55) -- (1.25,0.55);
 \end{tikzpicture}}
 \\ \hline
  $\begin{array}{l}
 \alpha_1 = \varepsilon_1-\varepsilon_2-\varepsilon_3\\
 \alpha_2 = 2\varepsilon_2\\
 \alpha_3 = 2\varepsilon_3
 \end{array}$ & 
 $\begin{array}{l}
 \alpha_1 = 2\varepsilon_2\\
 \alpha_2 = -\varepsilon_1-\varepsilon_2+\varepsilon_3\\
 \alpha_3 = 2\varepsilon_1
 \end{array}$ &
 $\begin{array}{l}
 \alpha_1 = 2\varepsilon_3\\
 \alpha_2 = 2\varepsilon_1\\
 \alpha_3 = -\varepsilon_1+\varepsilon_2-\varepsilon_3
 \end{array}$ &
 $\begin{array}{l}
 \alpha_1 = -\varepsilon_1+\varepsilon_2+\varepsilon_3\\
  \alpha_2 = \varepsilon_1+\varepsilon_2-\varepsilon_3\\
 \alpha_3 = \varepsilon_1-\varepsilon_2+\varepsilon_3
 \end{array}$  \\  \hline
 $\begin{pmatrix}
    0 & 1 & a \\
    -1 & 2 & 0\\
    -1 & 0 & 2
  \end{pmatrix}$ & 
 $\begin{pmatrix}
    2 & -1 & 0 \\
    -1 & 0 & 1+a\\
    0 &   -1 & 2
  \end{pmatrix}$ &
 $\begin{pmatrix}
    2 & 0 & -1 \\
    0 & 2 & -1\\
    -a & 1+a & 0
  \end{pmatrix}$ &
 $\begin{pmatrix}
    0 & 1 & a \\
    1 & 0 & -1-a\\
    a &    -1- a & 0
  \end{pmatrix}$  \\ \hline
 \end{tabular} 

}
\begin{tikzpicture}[overlay, red, remember picture]
\draw[<->, ultra thick, densely dotted] (A1) to[in=150, out =10] (B4);
\draw[<->, ultra thick, densely dotted] (A2) to[in=-150, out = -1] (C4);
\draw[<->, ultra thick, densely dotted] (A3) to[in=-150, out = 15] (A4);
 \end{tikzpicture}
\end{center}
\caption{Simple roots of $D(2,1;a)$, Cartan matrices and Dynkin diagrams.}
\label{rootsystems}
\end{figure}

\subsection{Parabolic subalgebras and flag supervarieties}

A choice of simple root system $\Pi$ defines the Borel subalgebra
$\fb=\fh\oplus\bigoplus_{\alpha\in\Delta^+}\fg_\alpha$.
A parabolic $\fp$ is a subalgebra of $\fg$ containing $\fb$.

For a subset $\chi\subset\mathcal{A}=\{1,2,3\}$, marked by crosses on the Dynkin diagram, let 
$Z=\sum_{i\in\chi}Z_i$ with the basis $Z_i\in\fh$ dual to $\alpha_i\in\fh^*$, that is $\alpha_i(Z_j)=\delta_{ij}$. 
Then the grading of $\fg$ is $\fg_k=\bigoplus_{Z(\alpha)=k}\fg_{\alpha}$.  
The corresponding parabolic subalgebra is $\fp_\chi=\fg_{\geq0}$.

 \begin{lemma}\label{andreuboris}\cite{Serganova,G3F4A}
Let $\Pi=\{\alpha_i\}_{i\in\mathcal{A}}$ be a simple root system and let $\alpha_j\in\Pi$ be an odd isotropic root. Consider the odd reflection $r_{\alpha_j}:\Pi\longrightarrow\Pi'=\{\alpha'_i\}_{i\in\mathcal{A}}$ that 
permutes $\chi\subset\mathcal{A}$ to $\chi'\subset\mathcal{A}$. Then the parabolic subalgebras 
$\fp_\chi^\Pi$ and  $\fp_{\chi'}^{\Pi'}$ are isomorphic if $\chi\subset\mathcal{A}\setminus\{j\}$.
 \end{lemma}

It is also useful to note that the odd reflection preserves $\Delta^+$ minus the ray of $\alpha$:
 $$
r_\alpha:\Delta^+\xrightarrow{\,\smash{\raisebox{-0.45ex}{\ensuremath{\scriptstyle\sim}}}\,}
(\Delta^+\setminus\{k\alpha:k>0\})\cup(\Delta^-\cap\{k\alpha:k<0\}).
 $$
Using these properties, one concludes that from 28 parabolic subalgebras $\fp^\Xi_\chi$ of $\fg$, 
only 6 are non-equivalent (using also a change of the parameter $a$ in odd reflections).
The equivalence types are noted in Figure \ref{6geometries}. We also display the graph of inclusions in
Figure \ref{pinclusions} (largest at the top).

\begin{small}
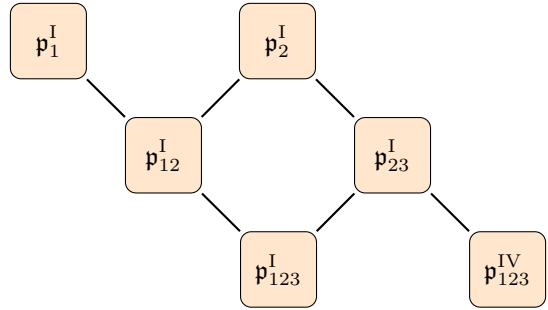
\begin{figure}[h!]\begin{center}\begin{tikzpicture}
     \node[taronja node] (T1) {$\fp_1^\I$};
     \node[taronja node] (T2) [right = 2cm of T1] {$\fp_2^\I$};
     \node[taronja node] (T3) [below right = 0.5cm and 0.5cm of T1] {$\fp_{12}^\I$};
     \node[taronja node] (T4) [below right = 0.5cm and 0.5cm of T2] {$\fp_{23}^\I$};
     \node[taronja node] (T5) [below right = 0.5cm and 0.5cm of T3]  {$\fp_{123}^\I$};
     \node[taronja node] (T6) [below right = 0.5cm and 0.5cm of T4] {$\fp_{123}^\IV$};
      \path[draw,thick]
     (T1) edge node {} (T3)
     (T2) edge node {} (T3)
     (T2) edge node {} (T4)
     (T3) edge node {} (T5)
     (T4) edge node {} (T5)
     (T4) edge node {} (T6);
 \end{tikzpicture}\end{center}
 \caption{Inclusions of parabolics of $D(2,1;a)$.}\label{pinclusions}
 \end{figure}
 \end{small}

The corresponding parabolic subgroups $P^\Xi_\chi$ in the Lie supergroup $G$ of $D(2,1;a)$
give rise to the homogeneous superspaces $G/P^\Xi_\chi$, and these are related by fiber bundles
(we refer to \cite{KST2} for details) with reverse directions to the above inclusions.
These will be discussed below when further enhanced with additional geometric structures.

\subsection{Tanaka--Weisfeiler prolongation}

Let $\fm=\fg_{-\nu}\oplus\cdots\oplus\fg_{-1}$ be a graded Lie superalgebra
such that $\fg_{-1}$ is fundamental, i.e.\ it bracket-generates $\fm$, and $\mathfrak{z}(\fm)=\fg_{-\nu}$. 
 \begin{definition}
The maximal prolongation of $\fm$ is the (unique) maximal graded Lie superalgebra 
 \begin{equation*}
\pr(\mathfrak{m})=\bigoplus_{i=-\nu}^{+\infty}\pr_i(\fm)
 \end{equation*}
\begin{enumerate}
 \item that extends $\mathfrak{m}$, that is $\pr_{<0}(\fm)=\fm$; \label{prolongation1}
 \item is effective, that is $[X,\mathfrak{g}_{-1}]\neq0$ for all $0\neq X\in\mathfrak{g}_{\geq0}$. \label{prolongation2}
\end{enumerate}
Let $\mathfrak{der}_0(\fm)$ be the Lie superalgebra of grade-preserving derivations of $\fm$. Given 
$\fg_0\subset\mathfrak{der}_0(\fm)$, we define $\pr(\fm,\fg_0)$ as the maximal graded Lie superalgebra 
that extends $\mathfrak{m}\oplus\mathfrak{g}_0$ and satisfies effectivity (\ref{prolongation2}).
 \end{definition}
 
There is a necessary condition for rigidity of $\fm$, which means $\dim\op{pr}(\fm)<\infty$, as follows. 
In the classical case it is also sufficient (over $\CC$: \cite{Ot}), but generally this fails in the super case. 

 \begin{proposition}\label{inftype}
Suppose there exists $0\neq v\in\fg_{-1}$ and $V\subset\fm$ of codimension 1 such that
$[v,V]=0$. Then $\dim\pr(\fm)=\infty$.
(In other words, for the rigidity of $\fm$ it is necessary that such $v,V$ cannot exist.)
 \end{proposition}

 \begin{proof}
If $[v,\fm]=0$, i.e.\ $\mathfrak{z}(\fm)\ni v$, the claim is obvious. So let $w\in\fg_{-1}\setminus V$ be such
that $[v,w]\neq0$. Choose $\omega\in V^\perp\cap\fg_{-1}^*$ such that $\omega(w)\neq0$. Then 
$v\otimes\omega\in\fg_0$ and, more generally, $v\otimes\omega^{\otimes k}\in\fg_{k-1}$ for $k>0$.
One can see that this, as an element of $\op{Hom}(\fg_{k-2},\fg_{-1})$, is non-zero.
 \end{proof}
 
A further generalization can be defined as follows. Suppose we already computed $\pr_{\leq k}(\fm)$
or $\pr_{\leq k}(\fm,\fg_0)$. Select its $\fm\oplus\fg_0$-submodule $\fg_{\leq k}=\fg_{-\nu}\oplus\dots\oplus\fg_k$
that contains $\fm$ and is stable with respect to the brackets already defined $[\fg_i,\fg_j]\subset\fg_{i+j}$ for $i+j\leq k$. (Note that, in general, $\fg_{\leq k}$ is not a Lie superalgebra, and 
only partial graded brackets are defined as indicated.)
 \begin{definition}
Under the above conditions, the maximal prolongation of $\fg_{\leq k}$ is the (unique) maximal graded Lie superalgebra 
$\pr(\fg_{\leq k})=\bigoplus_{i=-\nu}^{+\infty}\pr_i(\fg_{\leq k})$
that extends $\fg_{\leq k}$ and is effective as in above (\ref{prolongation2}).
 \end{definition}
Following \cite{G3super}, we call $\pr(\mathfrak{m})$ and $\pr(\mathfrak{g}_{\leq k})$ for $k\geq0$ the 
Tanaka--Weisfeiler prolongations. For $i>k$, they are computed by the following inductive formula \cite{Tanaka}. 
If $\pr_{<j}(\fg_{\leq k})=\oplus_{i=-\nu}^{j-1}\fg_i$ then
 \begin{equation}\label{conditionprolongation}
\pr_j(\fg_{\leq k})=\{A\in\op{Hom}_{+j}\bigl(\fm,\pr_{<j}(\fg_{\leq k})\bigr): 
A[x,y]=\left[Ax,y\right]+(-1)^{|x||A|}\left[x,Ay\right],\,\forall x,y\in\fm\}.
 \end{equation}
The structure is of finite type if $\pr_j(\fg_{\leq k})=0$ for $j>\mu$ for some $\mu$. In this case, 
the obtained prolongation $\fg=\oplus_{i=-\nu}^\infty\fg_i$ is a Lie superalgebra.

Now consider a supermanifold $M$ with a vector superdistribution $\mathcal{D}\subset\mathcal{T}M$
(here, both the tangent bundle and its subbundle $\mathcal{D}$ are considered as sheaves over the
structure sheaf $\mathcal{O}_M$ of the supermanifold). If the superdistribution is bracket-generating,
then there is a sheaf of graded nilpotent Lie superalgebras $\fm_M$ over $M$. In the case this is modelled on
a fixed Lie superalgebra $\fm$, a tower of geometric prolongations $\mathcal{F}r_0\to M$ and further 
$\mathcal{F}r_k\to\mathcal{F}r_{k-1}$ for $k>0$ (bundles with fibers $\fg_k$) was constructed in \cite{KST2}. 

A structure reduction of order 0 is given by a subbundle $\mathfrak{F}_0\subset\mathcal{F}r_0$,
which is often given by a geometric structure on $\mathcal{D}$. 
Higher order reductions $\mathfrak{F}_k\subset\mathcal{F}r_k$, $k>0$, are also discussed in \cite{KST2}.
We will consider those for the generalized flag supervarieties $M=G/P^\Xi_\chi$ in what follows.

\section{Algebraic prolongations for $D(2,1;a)$}\label{S3}

In this section, we compute prolongations of $\fm$ and more generally of $\fg_{\leq k}$ for 6 
gradations of $\fg=D(2,1;a)$. We will use both spectral sequences and direct computations based on 
\eqref{conditionprolongation}. Tedious and repeated computations will be shortened or delegated to
the appendix. The result has an independent computer verification, which is based on integer
arithmetic and hence has full mathematical rigor. 

\subsection{Parabolic $\fp_{1}^{\rm I}$}\label{pi1}

Consider $D(2,1;a)$ with the contact grading $\fg=\fg_{-2}\oplus\fg_{-1}\oplus\fg_0\oplus\fg_1\oplus\fg_2$.  

\begin{figure}[h]
 \centering
  \begin{minipage}{0.3\textwidth}
  \centering
 \begin{small}\begin{align*}
   \begin{array}{|c|c|c|} \hline
    k & \Delta_{\bar{0}}(k)  &  \Delta_{\bar{1}}(k) \\ \hline\hline
    0 & \pm\alpha_2,\ \pm\alpha_3 & \\ \hline
    1 &  &  \begin{array}{l} \alpha_1,\ \alpha_1+\alpha_2+\alpha_3, \\  
    \alpha_1+\alpha_2,\ \alpha_1+\alpha_3 \end{array}\\ \hline
    2 & 2\alpha_1+\alpha_2+\alpha_3 & \\ \hline
   \end{array}
  \end{align*}\end{small}
 \end{minipage}
 \begin{minipage}{0.25\textwidth}
  \centering
  \raisebox{-0.2in}{
  \!\!\begin{tikzpicture}[remember picture]
   \node (A1)[draw, circle, scale=1] {};
   \node [draw,circle,fill=mygray] {};
   \node[below] at (0,-0.15) {$\times$};
   \node (B1)[draw, circle, scale=1, above right of=A1, xshift=0.5cm] {};
   \node (C1)[draw, circle, scale=1, below right of=A1, xshift=0.5cm] {};
   \node [left of=A1, xshift=0.5cm]{$\alpha_1$};
   \node [right of=B1, xshift=-0.5cm]{$\alpha_2$};
   \node [right of=C1, xshift=-0.5cm]{$\alpha_3$};
   \draw (A1) -- (B1);
   \draw (A1) --node[anchor=north]{
   } (C1);
  \end{tikzpicture}}
 \end{minipage}
 \begin{minipage}{0.25\textwidth}
  \centering
  \begin{small}\begin{align*}
   \begin{array}{|c|c|c|} \hline
\vphantom{\frac{A}A}   k & (\fg_k)_{\bar0}/\fh  &  (\fg_k)_{\bar1} \\ \hline\hline
    0 & \begin{array}{l} X_2,\ Y_2 , \\  X_3,\ Y_3 \end{array} & \\ \hline
    1 &  &  \begin{array}{l} xyy,\ xxy, \\  xyx,\ xxx \end{array}\\ \hline
    2 & X_1 & \\ \hline
   \end{array}
  \end{align*}\end{small}
 \end{minipage}
 \caption{The grading $\mathfrak{p}_1^{\rm I}$.}
\end{figure}

 \begin{proposition}
For the grading of $\fg=D(2,1;a)$ associated with $\fp_1^{\rm I}$: $\pr(\fm)=\pr(\fm,\fg_0)\neq\fg$.
In fact, $\dim\pr(\fm,\fg_0)=\infty$.
 \end{proposition} 

 \begin{proof}
The prolongation of the Heisenberg superalgebra $\fm$ is the contact algebra 
$\pr(\fm)=\mathfrak{cont}(\fm)$. The distribution $\fg_{-1}$ is odd, so its 
structure algebra is the conformal symplectic algebra
$\fg_0=\mathfrak{co}(4)$. Thus $\pr(\fm)=\pr(\fm,\fg_0)$ and the claim follows.
 \end{proof}

 \begin{proposition}\label{propp1I}
For the grading of $\fg=D(2,1;a)$ associated with $\fp_1^{\rm I}$: $\pr(\fg_{\leq1})=\fg$.
 \end{proposition}

 \begin{proof}
As $\fg_0^{ss}=\sl_2\oplus\sl_2$ modules, $\fg_{-1}=\UU_{1,1}=\fg_1$. 
We first claim that $\pr_2(\fg_{\leq1})=\fg_2$. The space
$\fg_{-1}$ has basis $\{e_1=yxx,e_2=yyx,e_3=yxy,e_4=yyy\}$ with the dual coframe 
$\{\omega^1,\omega^2,\omega^3,\omega^4\}$ and $\fg_1$ has basis $\{f_1=xxx,f_2=xyx,f_3=xxy,f_4=xyy\}$.

We have $\pr_2(\fg_{\leq1}) \hookrightarrow\fg_1\otimes\fg_{-1}^\ast=
\mathbb{V}_{2,2}\oplus\mathbb{V}_{2,0}\oplus\mathbb{V}_{0,2}\oplus\mathbb{V}_{0,0}$, 
and the summands have the following highest weight vectors:
 \begin{align*}
\VV_{2,2}:\, f_1\otimes\omega^4,\qquad
\VV_{0,2}:\, f_3\otimes\omega^4+f_1\otimes\omega^2,\qquad
\VV_{2,0}:\, f_2\otimes\omega^4+f_1\otimes\omega^3,\\
\VV_{0,0}:\, f_1\otimes\omega^1+f_2\otimes\omega^2+f_3\otimes\omega^3+f_4\otimes\omega^4.
 \end{align*}
By Schur's lemma, to determine $\pr_2(\fg_{\leq1})$, it suffices to test \eqref{conditionprolongation} 
on the highest weight vectors:
 \begin{itemize}
  \item If $A=f_1\otimes\omega^4\in\VV_{2,2}$ is in $\pr_2(\fg_{\leq1})$, then 
$0=A[e_2,e_4]=[A e_2,e_4]+[e_2,A e_4]=[e_2,f_1]=s_3X_3\neq0$, which is a contradiction. Therefore $\VV_{2,2}\not\subset\pr_2(\fg_{\leq1})$.
  \item If $B=f_3\otimes\omega^4+f_1\otimes\omega^2$ is in $\pr_2(\fg_{\leq1})$, then $0=B[e_2,e_4]=[Be_2,e_4]+[e_2,Be_4]=[f_1,e_4]+[e_2,f_3]=-s_3H_3\neq0$. 
Therefore $\VV_{0,2}\not\subset\pr_2(\fg_{\leq1})$.
  \item Similarly we can see that $\VV_{2,0}\not\subset\pr_2(\fg_{\leq1})$.
  \item Since $\fg_2\subset\pr_2(\fg_{\leq1})$, then $\pr_2(\fg_{\leq1})=\fg_2=\VV_{0,0}$.
 \end{itemize}
Finally, we have $\pr_3(\fg_{\leq1})\hookrightarrow\fg_2\otimes\fg_{-1}^\ast=\UU_{1,1}$ as a $\fg_0$-module. 
This has highest weight vector $C=X_1\otimes\omega^4\in\UU_{1,1}$. If $C\in\pr_3(\fg_{\leq1})$, then 
$0=C[e_2,e_4]=-[e_2,X_1]=f_2\neq 0$. Therefore $\pr_3(\fg_{\leq1})=0$. This proves that $\pr(\fg_{\leq1})=\fg$.
\end{proof}

\subsection{Parabolic $\fp_{2}^{\rm I}$}\label{pi2}

Consider $D(2,1;a)$ with the grading $\fg=\fg_{-1}\oplus\fg_0\oplus\fg_1$ associated with $\fp_{2}^{\rm I}$.
   
\begin{figure}[h]
    \centering
\begin{minipage}{0.3\textwidth}\centering
\begin{small}\begin{align*}
 \begin{array}{|c|c|c|} \hline
 k & \Delta_{\bar{0}}(k) & \Delta_{\bar{1}}(k) \\ \hline\hline
 0 & \pm\alpha_3 & \pm\alpha_1,\ \pm(\alpha_1+\alpha_3)\\ \hline
 1 & \alpha_2,\ 2\alpha_1+\alpha_2+\alpha_3 & \alpha_1+\alpha_2,\ \alpha_1+\alpha_2+\alpha_3\\ \hline
 \end{array}
 \end{align*}\end{small}
\end{minipage} 
    \begin{minipage}{0.25\textwidth}\centering
\!\!\!\!\raisebox{-0.2in}{\begin{tikzpicture}[remember picture]
\node (A1)[draw, circle, scale=1] {};
\node [draw,circle,fill=mygray] {};
\node (B1)[draw, circle, scale=1, above right of=A1, xshift=0.5cm] {};
\node[below] at (1.2,0.55) {$\times$};
\node (C1)[draw, circle, scale=1, below right of=A1, xshift=0.5cm] {};
\node [left of=A1, xshift=0.5cm]{$\alpha_1$};
\node [right of=B1, xshift=-0.5cm]{$\alpha_2$};
\node [right of=C1, xshift=-0.5cm]{$\alpha_3$};
\draw (A1) -- (B1);
\draw (A1) --node[anchor=north]{} (C1);
 \end{tikzpicture}}\!\!
\end{minipage}
\begin{minipage}{0.25\textwidth}\centering
\begin{small}\begin{align*}
 \begin{array}{|c|c|c|} \hline
 k & \!(\fg_k)_{\bar{0}}/\fh\! & (\fg_k)_{\bar{1}} \\ \hline\hline
 0 &  X_3,\ Y_3 & \!\!\begin{array}{l} xyy,\ xyx,\\ yxx,\ yxy \end{array}\!\! \\ \hline
 1 & X_2,\ X_1 & \!\!xxy,\ xxx\!\! \\ \hline
 \end{array}
 \end{align*}\end{small}
\end{minipage} 
        \caption{The grading $\mathfrak{p}_2^{\rm I}$.}
\end{figure}

 \begin{proposition}
For the $\fp_{2}^{\rm I}$ associated grading of $\fg=D(2,1;a)$: $\pr(\fm)\neq\fg$. 
In fact, $\dim\pr(\fm)=\infty$.
 \end{proposition}

Indeed, in the $|1|$-graded case the prolongation is $\pr(\fm)=\fm\otimes S(\fm^*)$, which geometrically 
corresponds to the Lie superalgebra of all supervector fields $\mathfrak{vect}(\fm)$.

\begin{proposition}\cite{Poletaeva}
For the grading of $\fg=D(2,1;a)$ associated with $\fp_{2}^{\rm I}$:  $\pr(\fm,\fg_0)=\fg$. 
 \end{proposition}

We have confirmed this result by an independent symbolic computation.

\subsection{Parabolic $\fp_{12}^{\rm I}$}\label{pi12}
  
Consider $D(2,1;a)$ with the grading  $\fg=\fg_{-3}\oplus\cdots\oplus\fg_3$ associated with $\fp_{12}^{\rm I}$.
 
 \begin{figure}[h]
    \centering
\begin{minipage}{0.3\textwidth}\centering
\begin{align*}\begin{small}
 \begin{array}{|c|c|c|} \hline
 k & \Delta_{\bar{0}}(k)  & \Delta_{\bar{1}}(k) \\ \hline\hline
 0 & \pm\alpha_3 & \\ \hline
 1 &  \alpha_2 & \alpha_1,\ \alpha_1+\alpha_3\\ \hline
 2 & & \alpha_1+\alpha_2,\ \alpha_1+\alpha_2+\alpha_3\\ \hline
 3 &  2\alpha_1+\alpha_2+\alpha_3 & \\ \hline
 \end{array}\end{small}
 \end{align*}
 \end{minipage}
\begin{minipage}{0.25\textwidth}\centering
\!\!\raisebox{-0.2in}{
 \begin{tikzpicture}[remember picture]
  \node (A1)[draw, circle, scale=1] {};
 \node [draw,circle,fill=mygray] {};
 \node (B1)[draw, circle, scale=1, above right of=A1, xshift=0.5cm] {};
\node[below] at (0,-0.15) {$\times$};
\node[below] at (1.2,0.55) {$\times$};
  \node (C1)[draw, circle, scale=1, below right of=A1, xshift=0.5cm] {};
  \node [left of=A1, xshift=0.5cm]{$\alpha_1$};
 \node [right of=B1, xshift=-0.5cm]{$\alpha_2$};
 \node [right of=C1, xshift=-0.5cm]{$\alpha_3$};
  \draw (A1) -- (B1);
 \draw (A1) --node[anchor=north]{} (C1);
 \end{tikzpicture}}
\end{minipage}
\begin{minipage}{0.25\textwidth}\centering
\begin{align*}\begin{small}
 \begin{array}{|c|c|c|} \hline
 k & (\fg_k)_{\bar{0}}/\fh & (\fg_k)_{\bar{1}} \\ \hline\hline
 0 & X_3,\ Y_3 & \\ \hline
 1 & X_2 &  xyy,\ xyx \\ \hline
 2& & xxy,\ xxx \\ \hline
 3& X_1 & \\ \hline
 \end{array}\end{small}
 \end{align*}
 \end{minipage}
        \caption{The grading $\mathfrak{p}_{12}^{\rm I}$.}
\end{figure}
   
 \begin{proposition}
For the grading of $\fg=D(2,1;a)$ associated with $\fp_{12}^{\rm I}$: $\pr(\fm)=\pr(\fm,\fg_0)\neq\fg$. 
In fact, $\dim\pr(\fm,\fg_0)=\infty$.
 \end{proposition}
 
 \begin{proof}
It will be shown in Section \ref{TwSec} by geometric arguments that the Lie superalgebra $\pr(\fm)$ for $\fp_{12}^\I$ is
isomorphic to its counterpart for $\fp_1^\I$, namely the Lie superalgebra $\mathfrak{cont}(1|4)$,
though with a different grading. In this case $\fg_0=\gl_2\oplus\CC$ and $\fg_1=\CC^{1|0}\oplus\CC^{0|2}$,
while $\pr_1(\fm,\fg_0)=\CC^{0|2}\oplus\CC^{1|0}\oplus\CC^{0|2}$. Thus also
$\pr(\fm,\fg_0)=\mathfrak{cont}(1|4)$.
 \end{proof}

Let us note that the claim on infinite-dimensionality of $\pr(\fm)$ does not follow from Proposition \ref{inftype},
as the required vector $v\in\fg_{-1}$ does not exist. Consequently, this is an example where the
rigidity criterion for the Tanaka--Weisfeiler prolongation is not sufficient.
 
 \begin{proposition}
For the grading of $\fg=D(2,1;a)$ associated with $\fp_{12}^{\rm I}$: $\pr(\fg_{\leq1})=\fg$.
 \end{proposition}

 \begin{proof}
Let $\{e_1=Y_2, e_2=yxx, e_3=yxy\}$ be a basis of $\fg_{-1}$ and $\omega^1,\omega^2,\omega^3$
be the dual basis; let $\{e_4=yyx, e_5=yyy\}$ be a basis of $\g_{-2}$ and $\{e_6=Y_1\}$ that for $\g_{-3}$.

Next, let $\{f_1=X_2, f_2=xyx, f_3=xyy\}$ be a basis of $\fg_1=\VV_0\oplus\UU_1$. 
We have $\pr_2(\fg_{\leq1})\hookrightarrow \fg_1\otimes\fg_{-1}^\ast
=\VV_0\oplus\UU_1\oplus\UU_1\oplus\VV_2\oplus\VV_0$
with highest weight vectors of the isotypic components:
 \begin{eqnarray*}
  \VV_0\oplus\VV_0:\,f_1\otimes\omega^1,\,  f_3\otimes\omega^3+f_2\otimes\omega^2;\qquad
  \UU_1\oplus\UU_1:\, f_1\otimes\omega^3,\, f_2\otimes\omega^1;\qquad
  \VV_2:\, f_2\otimes\omega^3.
 \end{eqnarray*}
Let us test the highest weight vectors in $\fg_1\otimes\fg_{-1}^\ast$:
\begin{itemize}
  \item If $A=f_2\otimes\omega^3\in\VV_2$ is in $\pr_2(\fg_{\leq1})$, then $0=A[e_2,e_3]=[A e_2,e_3]+[e_2,Ae_3]=[e_2,f_2]=-s_3 X_3\neq0$, which is a contradiction. Therefore $\VV_2\not\subset\pr_2(\fg_{\leq1})$.
 \item If $B=af_1\otimes\omega^1+b(f_3\otimes\omega^3+ f_2\otimes\omega^2)$ is in $\pr_2(\fg_{\leq1})$, 
then $B e_4=B[e_1,e_2]=[af_1,e_2]+[e_1,bf_2]=0$. Thus $0=B[e_1,e_4]=[af_1,e_4]=ae_2$, so $a=0$.
Similarly, $B[e_2,e_3]=[bf_2,e_3]+[e_2,bf_3]=bs_3H_3$, so $b=0$. 
 \item If $C=cf_1\otimes\omega^3+d\,f_2\otimes\omega^1$ is in $\pr_2(\fg_{\leq1})$, then 
\begin{equation*}
 Ce_5=C[e_1,e_3]=[df_2,e_3]+[e_1,cf_1]=d\left(\frac{s_1}{2}H_1-\frac{s_2}{2}H_2+\frac{s_3}{2}H_3\right)-cH_2, 
\end{equation*}
and $0=C[e_1,e_5]=[df_2,e_5]+[e_1,d\left(\frac{s_1}{2}H_1-\frac{s_2}{2}H_2+\frac{s_3}{2}H_3\right)-cH_2]=(-2ds_2-2c) Y_2$, therefore $c=-ds_2$, and $\UU_1\oplus\UU_1\not\subset\pr_2(\fg_{\leq1})$. 
Since $\fg_2\subset\pr_2(\fg_{\leq1})$, so $\pr_2(\fg_{\leq1})=\fg_2=\UU_1$.
 \end{itemize}

Let $\{f_4=xxy,f_5=xxx\}$ be a basis of $\fg_2=\UU_1$.
We have $\pr_3(\fg_{\leq1})\hookrightarrow\fg_2\otimes\fg_{-1}^\ast=\UU_1\oplus\VV_2\oplus\VV_0$
with highest weight vectors of the irreducible components:
 \begin{align*}
  \UU_1:\, f_5\otimes\omega^1;\qquad
  \VV_2:\, f_5\otimes\omega^3;\qquad
  \VV_0:\, f_5\otimes\omega^2+f_4\otimes\omega^3.
\end{align*}
Let us test the highest weight vectors in $\fg_2\otimes\fg_{-1}^\ast$.
\begin{itemize}
  \item If $D=f_5\otimes\omega^1\in\UU_1$ is in $\pr_3(\fg_{\leq1})$, then 
  $D e_4=D[e_1,e_2]=[f_5,e_2]=0
  $ and $0=D[e_1,e_4]=[f_5,e_4]=s_3 X_3\neq 0$. Therefore $\UU_1\not\subset\pr_3(\fg_{\leq1})$.
\item If $E=f_5\otimes\omega^3\in\VV_2$ is in $\pr_3(\fg_{\leq1})$, then $Ee_5=E[e_1,e_3]=[e_1,Ee_3]=[e_1,f_5]=f_2$ and $0=E[e_3,e_5]=[f_5,e_5]+[e_3,f_2]=-s_2 H_2\neq0$. Therefore $\VV_2\not\subset\pr_3(\fg_{\leq1})$.
\item Since $\fg_3\subset\pr_3(\fg_{\leq1})$, we conclude $\pr_3(\fg_{\leq1})=\fg_3=\VV_0$.
 \end{itemize}

Finally, let $\{f_6=X_1\}$ be a basis of $\fg_3=\VV_0$.
We have $\pr_4(\fg_{\leq1})\hookrightarrow\fg_3\otimes\fg_{-1}^\ast=\VV_0\oplus\UU_1$,
and these isotypic components have highest weight vectors:
 \begin{align*}
  \VV_0:\, f_6\otimes\omega^1\qquad
  \UU_1:\, f_6\otimes\omega^3
\end{align*}
Let us test the highest weight vectors in $\fg_3\otimes\fg_{-1}^\ast$.
\begin{itemize}
  \item If $F=f_6\otimes\omega^1\in\VV_0$ is in $\pr_4(\fg_{\leq1})$, then
  $F e_4=F[e_1,e_2]=[f_6,e_2]=f_5
  $, and $0=F[e_1,e_4]=[f_6,e_4]+[e_1,f_5]=2f_2\neq 0$. Therefore $\VV_0\not\subset\pr_4(\fg_{\leq1})$.
\item If $G=f_6\otimes\omega^3\in\UU_1$ is in $\pr_4(\fg_{\leq1})$, then 
$0=G[e_2,e_3]=[e_2,f_6]=-f_5\neq0$. Therefore $\UU_1\not\subset\pr_4(\fg_{\leq1})$, and 
we conclude $\pr_4(\fg_{\leq1})=\fg_4=0$.
 \end{itemize} 
This proves that $\pr(\fg_{\leq1})=\fg$.
 \end{proof}

\subsection{Parabolic $\fp_{23}^{\rm I}$}\label{pi23}

Consider $D(2,1;a)$ with the grading $\fg=\fg_{-2}\oplus\cdots\oplus\fg_2$ associated with $\fp_{23}^{\rm I}$.

\begin{figure}[h]
    \centering
\begin{minipage}{0.3\textwidth}\centering
 \begin{align*}\begin{small}
 \begin{array}{|c|c|c|} \hline
 k & \Delta_{\bar{0}}(k) & \Delta_{\bar{1}}(k) \\ \hline\hline
 0 & & \pm\alpha_1 \\ \hline
 1 &  \alpha_2,\ \alpha_3 & \alpha_1+\alpha_2,\ \alpha_1+\alpha_3\\ \hline
 2 & 2\alpha_1+\alpha_2+\alpha_3 & \alpha_1+\alpha_2+\alpha_3 \\ \hline
 \end{array}\end{small}
 \end{align*}
 \end{minipage}
\begin{minipage}{0.25\textwidth}\centering
\!\!\raisebox{-0.2in}{
 \begin{tikzpicture}[remember picture]
  \node (A1)[draw, circle, scale=1] {};
 \node [draw,circle,fill=mygray] {};
 \node (B1)[draw, circle, scale=1, above right of=A1, xshift=0.5cm] {};
\node[below] at (1.2,0.55) {$\times$};
\node[below] at (1.2,-0.85) {$\times$};
  \node (C1)[draw, circle, scale=1, below right of=A1, xshift=0.5cm] {};
  \node [left of=A1, xshift=0.5cm]{$\alpha_1$};
 \node [right of=B1, xshift=-0.5cm]{$\alpha_2$};
 \node [right of=C1, xshift=-0.5cm]{$\alpha_3$};
  \draw (A1) -- (B1);
 \draw (A1) --node[anchor=north]{} (C1);
 \end{tikzpicture}}
\end{minipage}
\begin{minipage}{0.25\textwidth}\centering
 \begin{align*}\begin{small}
 \begin{array}{|c|c|c|} \hline
 k & (\fg_k)_{\bar{0}}/\fh & (\fg_k)_{\bar{1}} \\ \hline\hline
 0 & & xyy,\ yxx\\ \hline
 1 &  X_2,\ X_3 & xyx,\ xxy \\ \hline
 2 & X_1 & xxx \\ \hline
 \end{array}\end{small}
 \end{align*}
 \end{minipage}
        \caption{The grading $\mathfrak{p}_{23}^{\rm I}$.}
\end{figure}
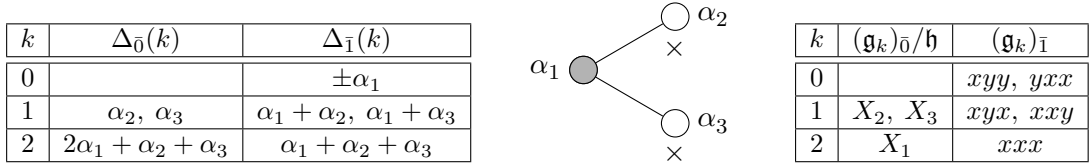

 \begin{proposition}
For the $\fp_{23}^{\rm I}$ associated grading of $\fg=D(2,1;a)$: $\pr(\fm)\neq\fg$. 
In fact, $\dim\pr(\fm)=\infty$.
 \end{proposition}

 \begin{proof}
Let us demonstrate the last claim, which implies the first one. The precise description of the prolongation will
be given in Section \ref{S3322} in geometric terms. Note that the vector $v=X_2$ and the subspace 
$V=\langle X_1,X_2,X_3,xxy,xxx\rangle$ satisfy the conditions of Proposition \ref{inftype}, as well
as the vector $v=X_3$ and the subspace $V=\langle X_1,X_2,X_3,xyx,xxx\rangle$.
Therefore $\dim\pr(\fm)=\infty$.
 \end{proof}
 
 \begin{proposition}\label{propPi23prol}
For the grading of $\fg=D(2,1;a)$ associated with $\fp_{23}^{\rm I}$: $\pr(\fm,\fg_0)=\fg$.
 \end{proposition}

The proof is by a long computation, which we delegate to Appendix \ref{B1}.

\subsection{Parabolic $\fp_{123}^{\rm I}$}\label{pi123}
 
Consider $D(2,1;a)$ with the grading $\fg=\fg_{-4}\oplus\cdots\oplus\fg_4$ associated with $\fp_{123}^{\rm I}$.
Here $\fg_0=\fh$.
 
 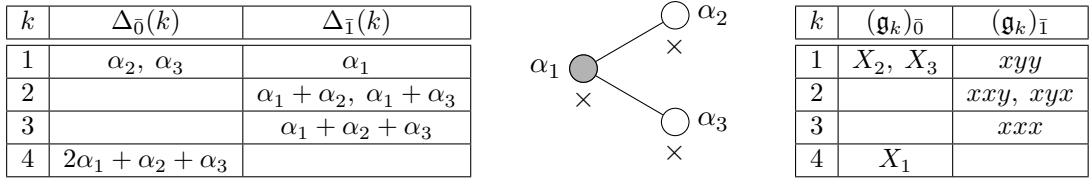
\begin{figure}[h]
    \centering
\begin{minipage}{0.3\textwidth}\centering
\begin{align*}\begin{small}
 \begin{array}{|c|c|c|} \hline
 k & \Delta_{\bar{0}}(k) & \Delta_{\bar{1}}(k) \\ \hline\hline
 1 &  \alpha_2,\ \alpha_3 & \alpha_1\\ \hline
 2 &   & \alpha_1+\alpha_2,\ \alpha_1+\alpha_3 \\ \hline
 3 &  & \alpha_1+\alpha_2+\alpha_3 \\ \hline
 4 &   2\alpha_1+\alpha_2+\alpha_3 & \\ \hline
 \end{array}\end{small}
 \end{align*}
 \end{minipage}
\begin{minipage}{0.25\textwidth}\centering
\!\!\raisebox{-0.2in}{
 \begin{tikzpicture}[remember picture]
  \node (A1)[draw, circle, scale=1] {};
 \node [draw,circle,fill=mygray] {};
 \node (B1)[draw, circle, scale=1, above right of=A1, xshift=0.5cm] {};
\node[below] at (0,-0.15) {$\times$};
\node[below] at (1.2,0.55) {$\times$};
\node[below] at (1.2,-0.85) {$\times$};
  \node (C1)[draw, circle, scale=1, below right of=A1, xshift=0.5cm] {};
  \node [left of=A1, xshift=0.5cm]{$\alpha_1$};
 \node [right of=B1, xshift=-0.5cm]{$\alpha_2$};
 \node [right of=C1, xshift=-0.5cm]{$\alpha_3$};
  \draw (A1) -- (B1);
 \draw (A1) --node[anchor=north]{} (C1);
 \end{tikzpicture}}
\end{minipage}
\begin{minipage}{0.25\textwidth}\centering
\begin{align*}\begin{small}
 \begin{array}{|c|c|c|} \hline
 k & (\fg_k)_{\bar{0}} & (\fg_k)_{\bar{1}} \\ \hline\hline
 1 & X_2,\ X_3 & xyy \\ \hline
 2 &  &   xxy,\ xyx \\ \hline
 3 &  & xxx \\ \hline
 4 & X_1 &  \\ \hline
 \end{array}\end{small}
 \end{align*}
 \end{minipage}
        \caption{The grading $\mathfrak{p}_{123}^{\rm I}$.}
\end{figure}
   
 \begin{proposition}
For the grading of $\fg=D(2,1;a)$ associated with $\fp_{123}^{\rm I}$: $\pr(\fm)=\pr(\fm,\fg_0)\neq\fg$. 
In fact, $\dim\pr(\fm,\fg_0)=\infty$.
 \end{proposition}
 
 \begin{proof}
It will be shown in Section \ref{TwSec} by geometric arguments that also the Lie superalgebra $\pr(\fm)$ for 
$\fp_{123}^\I$ is isomorphic to its counterpart for $\fp_1^\I$, namely the Lie superalgebra $\mathfrak{cont}(1|4)$,
though with yet another different grading. Thus again $\pr(\fm,\fg_0)=\mathfrak{cont}(1|4)$.
 \end{proof}

The claim on infinite-dimensionality of $\pr(\fm)$ again does not follow from Proposition \ref{inftype},
as the required vector $v\in\fg_{-1}$ does not exist. Thus, this is yet another example where the
rigidity criterion for the Tanaka--Weisfeiler prolongation is not sufficient.
 
 \begin{proposition}\label{propPi123prol}
For the grading of $\fg=D(2,1;a)$ associated with $\fp_{123}^{\rm I}$: $\pr(\fg_{\leq1})=\fg$.
 \end{proposition}
 
This statement follows from geometric considerations in Section \ref{TwSec} that demonstrate equivalence of 
symmetries of the corresponding flat geometry and that for $\fp_1^\I$. The explicit symmetries also show
the place of reduction. A direct algebraic proof is delegated to Appendix \ref{B2}.

\subsection{Parabolic $\fp_{123}^{\rm IV}$}\label{piv123}

Consider $D(2,1;a)$ with the grading $\fg=\fg_{-3}\oplus\cdots\oplus\fg_3$ associated with $\fp_{123}^{\rm IV}$.
Here $\fg_0=\fh$.
 
\begin{figure}[h]
    \centering
\begin{minipage}{0.3\textwidth}\centering
\begin{align*}\begin{small}
 \begin{array}{|c|c|c|} \hline
 k & \Delta_{\bar{0}}(k) & \Delta_{\bar{1}}(k) \\ \hline\hline
 1 & & \alpha_1,\ \alpha_2,\ \alpha_3\\ \hline
 2 & \begin{array}{l} \alpha_1+\alpha_2,\ \alpha_1+\alpha_3,\\ \alpha_2+\alpha_3 \end{array}
 & \\ \hline
 3 &  & \alpha_1+\alpha_2+\alpha_3 \\ \hline
 \end{array}\end{small}
 \end{align*}
 \end{minipage}
\begin{minipage}{0.25\textwidth}\centering
\!\!\raisebox{-1.1in}{
 \begin{tikzpicture}[remember picture]
  \node (A1)[draw, circle, scale=1] {};
 \node [draw,circle,fill=mygray] {};
 \node (B1)[draw, circle, scale=1, above right of=A1, xshift=0.5cm, fill=mygray] {};
\node[below] at (0,-0.15) {$\times$};
\node[above] at (1.2,0.85) {$\times$};
\node[below] at (1.2,-0.85) {$\times$};
  \node (C1)[draw, circle, scale=1, below right of=A1, xshift=0.5cm, fill=mygray] {};
  \node [left of=A1, xshift=0.5cm, yshift=-0.05cm]{$\alpha_1$};
 \node [right of=B1, xshift=-0.5cm, yshift=-0.05cm]{$\alpha_2$};
 \node [right of=C1, xshift=-0.5cm, yshift=-0.05cm]{$\alpha_3$};
 \draw (A1) -- (B1); \draw (A1) -- (C1); \draw (B1) -- (C1);
 \end{tikzpicture}}
\end{minipage}
\begin{minipage}{0.25\textwidth}\centering
\begin{align*}\begin{small}
 \begin{array}{|c|c|c|} \hline
 k & (\fg_k)_{\bar{0}} & (\fg_k)_{\bar{1}} \\ \hline\hline
 1 &  & yxx,\ xxy,\ xyx \\ \hline
 2 & X_1,\ X_2,\ X_3 &  \\ \hline
 3 &  & xxx \\ \hline
 \end{array}\end{small}
 \end{align*}
 \end{minipage}
        \caption{The grading $\mathfrak{p}_{123}^{\rm IV}$.}
\end{figure}
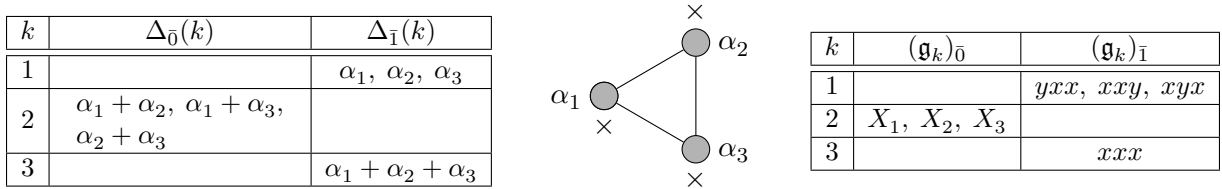

 \begin{proposition}\label{propPiv123prol}
For the grading of $\fg=D(2,1;a)$ associated with $\fp_{123}^\IV$: $\pr(\fm)=\fg$.
 \end{proposition}
 
Again, the proof follows the same computational scheme, and this long calculation is delegated to Appendix \ref{B3}.

\section{Geometric realizations as symmetries}\label{S4}

Let $G$ be an exceptional Lie supergroup $G$ of type $D(2,1;a)$ (for instance, $\exp\fg$) and $P$ its
parabolic subgroup marked according to the choice of parabolic subalgebra $\fp\subset\g$.
For each such choice, we will construct a geometry on the homogeneous superspace $G/P$, see
\cite{CCF,Va} for details on super quotients and \cite{V1,G3F4A} for examples of flag supermanifolds.

\subsection{Superspace $M^{1|4}=G/P_1^\I$}\label{Gpi1}

Consider the space $M^{1|4}=G/P_1^I$ with local coordinates $(x|\xi_i)$, $1\leq i\leq4$, and its
invariant odd contact structure $\mathcal{C}$.
This space can be locally identified with the jet space $J^1(\C^{0|2},\C^{1|0})$ with local coordinates
$(\theta_1,\theta_2,u,u_{\theta_1},u_{\theta_2})$; however, we will not rely on this identification
because of the following:

 \begin{remark}\label{rmk1}
We have $J^2(\C^{0|2},\C^{1|0})\simeq
\C^{2|4}(\theta_1,\theta_2,u,u_{\theta_1},u_{\theta_2},u_{\theta_1\,\theta_2})$ and furthermore
$J^2\simeq J^3\simeq\dots\simeq J^\infty$.
The Lie superalgebra $D(2,1;a)$, realized by contact vector fields, acts transitively on this superspace, 
hence a geometric realization in the form of invariant differential equation on $u=u(\theta_1,\theta_2)$
is impossible.
 \end{remark}
 
Instead, we will use the contact structure in Darboux form ($x$ even, $\xi_i$ odd)
 \begin{equation}\label{Darboux1}
\sigma=dx+\sum_{i=1}^4\xi_id\xi_i.
 \end{equation}
The contact distribution, annihilating $\sigma$, is
 $$
\mathcal{C}=\langle D_{\xi_i}=\p_{\xi_i}+\xi_i\p_x\rangle. 
 $$
Contact vector fields correspond to functions $f\in\mathcal{O}_M$ and have the form 
 \begin{equation}\label{Xf}
X_f=f\p_x +(-1)^{|f|}\tfrac12\sum_{j=1}^4(D_{\xi_j}f)D_{\xi_j}.
 \end{equation}
The Lagrange--Jacobi bracket, given by $[X_f,X_g]=X_{\{f,g\}}$, has the formula
 \begin{equation}\label{JLB}
\{f,g\}=f\p_xg-(-1)^{|f|\,|g|}g\p_xf+(-1)^{|f|}\tfrac12\sum_{j=1}^4(D_{\xi_j}f)(D_{\xi_j}g).
 \end{equation}
This gives an isomorphism between the Lie superalgebra of contact vector fields $(\mathfrak{cont}(M),[,]$)
and the Jacobi superalgebra $(\mathcal{O}_M,\{,\})$. If the elements (vector fields, resp.\ functions) are 
restricted to be polynomial in $x$, then both superalgebras are weighted with the shifted weight $w(X_f)=w(f)-2$. 
The weights of the Jacobi superalgebra are multiplicative with $w(1)=0$, $w(\xi_i)=1$, $w(x)=2$, and similarly $w(\p_{\xi_i})=-1$, $w(\p_x)=-2$.
 
The following table shows the gradation of a basis of these two Lie superalgebras that we denote 
$[\mathfrak{c}=\oplus_{k\ge-2}\mathfrak{c}_k]\simeq[\mathfrak{j}=\oplus_{k\ge0}\mathfrak{j}_k]$.
Both gradations are parity-consistent, however note that $w([X_f,X_g])=w(X_f)+w(X_g)$ 
while $w(\{f,g\})=w(f)+w(g)-2$. We represent elements $X_f\in\mathfrak{c}_k$ via their 
generating functions $f\in\mathfrak{j}_{k+2}$, but use the grading $k$ of $X_f$.
 \begin{center}
 \begin{tabular}{c|c|c}
$k$ &  $\mathfrak{c}_{k}$ & $\dim\mathfrak{c}_{k}$\\
\hline
$-2$ & 1 & 1\\
$-1$ & $\xi_i$ & 4\\
$0$ & $x$, $\xi_i\xi_j$ & $1+6=7$\\
$1$ & $x\xi_i$, $\xi_i\xi_j\xi_k$ & $4+4=8$\\
$2$ & $x^2$, $x\xi_i\xi_j$, $\xi_i\xi_j\xi_k\xi_l$ & $1+6+1=8$
 \end{tabular}
 \end{center}

Note that $\mathfrak{c}_0=\mathfrak{co}(4)=\mathfrak{sl}(2)\oplus\mathfrak{sl}(2)\oplus\C$
and, as $\mathfrak{c}_0$ module, $\mathfrak{c}_1=\C^4\oplus\C^4$, where each irreducible
summand can be also represented as $\C^2\boxtimes\C^2$.

That is where we do a reduction: choosing one standard $\C^4$ submodule in $\mathfrak{c}_1=\C^4\otimes\C^2$
(the latter factor is trivial as a module) is equivalent to choosing a point 
$\varepsilon\in\mathbb{P}(\C^2)=\mathbb{P}^1$: this factor is generated by $x\xi_i+\varepsilon\xi_i^\vee$, 
where $\xi_i^\vee=\p_{\xi_i}\nu$ for $\nu=\xi_1\xi_2\xi_3\xi_4$.

Calling the reduced Lie superalgebra $\g=\oplus_i\g_i$ with $\g_i=\mathfrak{c}_i$ for $-2\leq i\leq0$, 
$\g_1=\C^4_\varepsilon\subset\mathfrak{c}_1$ and $\g_i=\op{pr}_i(\g_{\leq1})$, we get
$\g_2=\langle\tfrac12x^2-\varepsilon\nu\rangle$ and $\g_i=0$ for $i>2$.

 \begin{proposition}
The graded Lie superalgebra in $\mathfrak{c}$ given by 
 \begin{equation}\label{jg}
\fg=\langle 1,\ \xi_i,\ x,\ \xi_i\xi_j,\ x\xi_i+\varepsilon\xi_i^\vee,\ \tfrac12x^2-\varepsilon\nu\rangle. 
 \end{equation}
is isomorphic to $D(2,1;a)$ for $a=\frac{1-\varepsilon}{1+\varepsilon}$, where $\varepsilon\neq\pm1,\infty$. 
 \end{proposition}

The verification is by identifying the bases, allowing to express $s_1,s_2,s_3$ via $\varepsilon$ (see also Remark \ref{Rkp1p12}).
Note that the standard Lie superalgebra $D(2,1)=\mathfrak{osp}(4|2)$ corresponds to $\varepsilon\in\{0,\pm3\}$
(the $S_3$ action has stabilizer $\ZZ_2$ on those).
The other special values $a=\frac{-1\pm i\sqrt{3}}{2}$ (the $S_3$ action has stabilizer $\ZZ_3$ on those), 
correspond to $\varepsilon=\pm i\sqrt{3}$.
 
 \begin{proposition}\label{normg}
The subalgebra $\g\subset\mathfrak{c}$ is maximal. 
 \end{proposition}
 
 \begin{proof}
Suppose there exists a larger proper subalgebra $\hat{\g}\supset\g$. Note that $\hat{\g}_0=\g_0=\mathfrak{co}(4)$, 
so we can split all $\hat{\g}_i$, $i>0$, into $\g_0$-modules. If $\g_1\subsetneq\hat{\g}_1$ then
$\hat{\g}_1=\mathfrak{c}_1$ and this by brackets generates the entire $\mathfrak{c}_+$, which together with
$\g$ gives $\mathfrak{c}$. Thus $\hat{\g}_1=\g_1$.

The $\g_0$-module $\mathfrak{c}_2$ consists of two trivial and one standard module. If the standard module was
present in $\hat{\g}_2$, then its brackets with $\g_{-1}$ by \eqref{JLB} would generate $\mathfrak{c}_1$. 
Similarly if both trivial modules were present in $\hat{\g}_2$, this would again imply $\hat{\g}_1=\mathfrak{c}_1$.
Thus $\hat{\g}_2=\g_2$.

One can proceed similarly with $\mathfrak{c}_3$ generated by $x^2\xi_i$, $x\xi_i\xi_j\xi_k$, and so on,
but in fact the equalities $\hat{\g}_s=0$ for $s>2$ (as well as $\hat{\g}_2=\g_2$) are equivalent to 
the Tanaka--Weisfeiler prolongation computation $\op{pr}(\g_{\le1})=\g$ established in Proposition \ref{propp1I}. 
 \end{proof}

 \begin{remark}
The proof actually implies that $D(2,1;a)$-module 
$\mathfrak{c}/\g=(\mathfrak{c}_1/\g_1)\oplus(\mathfrak{c}_2/\g_2)\oplus\g_3\oplus\g_4\oplus\dots$
is infinite-dimensional irreducible.
 \end{remark}

The space of generating functions $f$ for $\g$ is described by the system $\mathcal{E}$ of differential equations 
 \begin{equation}\label{Eq}
(\varepsilon\,\p_x^2+\p_{\xi_1}\p_{\xi_2}\p_{\xi_3}\p_{\xi_4})(f)=0,\quad
\xi_i(\varepsilon\,\p_x\p_{\xi_i}+\p_{\xi_i^\vee})(f)=0,\quad 
\p_x\p_{\xi_i}\p_{\xi_j}(f)=0,
 \end{equation}
we do not have summation by $i$ in the middle, and $\p_{\xi_i^\vee}=\p_{\xi_j}\p_{\xi_k}\p_{\xi_l}$
if $\xi_i^\vee=\xi_j\xi_k\xi_l$. Note that these equations imply $\p_x^3(f)=0$ as well as $\p_x^2\p_{\xi_i}(f)=0$.

Recall \cite{KL} that a differential equation $\mathcal{E}$ can be considered geometrically as a submanifold 
in jets and its infinitesimal symmetry is a vector field on the base whose natural prolongation to the space of jets (acting via the Lie derivative) preserves the ideal of the system.  Then a symmetry maps solutions to solutions, and this is a linear map for linear differential equations.

We will consider only point symmetries, i.e.\ vector fields on the space $J^0M=M\times\C$ of dependent 
and independent variables $(x,\xi_i,h)$ with even $h$.  Moreover, we restrict to contact symmetries of the contact base $M$ lifted to $J^0M$ by the rule 
 \begin{equation}\label{lift}
\mathfrak{cont}(M)\ni X_f\mapsto X_f^{(0)}:=X_f+f_xh\p_h\in\mathfrak{vect}(J^0M). 
 \end{equation}
One can verify that this is a homomorphism of Lie algebras, in fact a symplectization:

 \begin{proposition}
Map \eqref{lift} has its image in the Lie superalgebra of Hamiltonian vector fields $\mathfrak{symp}(J^0M)$
with respect to the symplectic structure $\omega=d(-h^{-1}\sigma)$ for $\sigma$ from \eqref{Darboux1}.
 \end{proposition} 
 
 \begin{proof}
We use the standard formulae $L_{X_f }\sigma=f_x\sigma$, $\iota_{X_f}\sigma=f$ and hence 
$\iota_{X_f}d\sigma=f_x\sigma-df$ for the contact fields (with even contact form). 
Since $\omega=h^{-2}dh\wedge\sigma-dh^{-1}\sigma$ is an even symplectic form on $\tilde{M}:=J^0M$, 
we get  $\tilde{X}_{\tilde{f}}:=X^{(0)}_f=\omega^{-1}(d\tilde{f})$ for $\tilde{f}:=h^{-1}f$.
Thus $\tilde{X}_{\tilde f}$ is the Hamiltonian vector field wrt $\omega$, which determines the Poisson bracket 
by $[\tilde{X}_{\tilde{f}},\tilde{X}_{\tilde{g}}]=\tilde{X}_{\{\tilde{f},\tilde{g}\}_{PB}}$. We conclude
 \begin{equation}\label{PBJB}
\{\tilde{f},\tilde{g}\}_{PB}=\widetilde{\{f,g\}_{JB}}, 
 \end{equation}
where in the rhs we specified the notation for the contact Lagrange-Jacobi brackets. 
 \end{proof}

Given the lift \eqref{lift}, the standard formulae for the prolongations of vector fields in the space of jets apply, 
see \cite{KL,G3super}, so we get prolonged vector fields $X_f^{(k)}$ on the jet space $J^kM$. 
The equation $\mathcal{E}$ also prolongs to a submanifold $\mathcal{E}^{(k)}\subset J^kM$ for all $k$ up to infinity.

 \begin{theorem}\label{Tpi1}
The Lie superalgebra $D(2,1;a)$ with $a=\frac{1-\varepsilon}{1+\varepsilon}$ is the symmetry subalgebra 
of the differential system $\mathcal{E}^\infty$ defined by equations \eqref{Eq} within the algebra 
of lifted vector fields $\mathfrak{cont}^{(\infty)}(M)\subset\mathfrak{vect}(J^\infty M)$.
  \end{theorem}
  
The requirement on the symmetry is to preserve not only the PDE system, but also the contact ideal on $M$,
or equivalently the symplectic form $\omega$ on $J^0M$ and the homogeneity in $h$.
 
 \begin{proof}
Symmetries of $\mathcal{E}$ map the space of solutions of $\mathcal{E}$ to itself. 
Since $\mathcal{E}$ is linear, so is its solution space $\mathcal{S}$. It has the basis 
of generating functions given by \eqref{jg}. Due to \eqref{Xf} and \eqref{JLB} we have:
 $$
X_f(g)=\{f,g\}+f_x g.
 $$
The last term does not, in general, belong to the subspace in $\mathfrak{c}$ given by \eqref{jg} for all 
$f,g\in\g$, and that's why we introduced lift \eqref{lift}, $\tilde{\fg}\subset\mathfrak{symp}(J^0M)$,
which yields using \eqref{PBJB} with the notations from the previous proof:
 $$
\tilde{X}_{\tilde{f}}(\tilde{g})=\widetilde{\{f,g\}}.
 $$
Thus $X^{(0)}_f$ preserves the class of solutions $h=f(x,\xi)$ of $\mathcal{E}$, or equivalently 
the space of functions $\{\widetilde{g(x,\xi)}\}\simeq\mathcal{S}$ iff $\{f,g\}\in\g$ 
$\forall g\in\g$. In other words, $f$ should belong to the normalizer of $\g$ within $\mathfrak{cont}(M)$,
which by Proposition \ref{normg} is equal to $\g$.
 \end{proof}

 \begin{remark}
Reduction of $\mathfrak{c}_1=\C^4\oplus\C^4$ to $\g_1$ is equivalent to the restriction to one 
$\mathfrak{c}_0$ submodule $\C^4\simeq\g_{-1}^*$, that is given geometrically through a subbundle 
of the higher frame bundle \cite{KST2}.
This resembles a higher order reduction in contact projective structures \cite{CS}. 
However, there the submodule $\g_1\subset\g_0\otimes\g_{-1}^*$ is uniquely determined,
while in our case it is given by a parameter $\varepsilon$. In addition, since $\g_1$ is odd, this suggests
a class of odd connections, but they do not exist in non-neutral parities \cite{BG}.
This is manifested by the non-existence of invariant nonlinear super PDEs in the lowest dimension 
as in Remark \ref{rmk1}.
 \end{remark}

\subsection{Superspace $M^{2|2}=G/P_2^\I$}\label{Gpi2}

This is a Hermitian supersymmetric space, since $\fg=\fg_{-1}\oplus\fg_0\oplus\fg_1$ is the grading of 
$\fg=D(2,1;a)$, according to Section \ref{pi2}. Here $\fg_0=\fgl(2|1)=\CC\oplus\fsl(2|1)$ has a basis
 \begin{align*}
Z=\frac{1}{2}H_1+\frac{1}{2}H_2,\quad X=X_3,\quad H=H_3,\quad Y=Y_3,\quad I=\frac{s_1}{s_3}H_1+\left(\frac{s_1}{s_3}+1\right)H_2,\\ F^+=-yxx,\quad F^-=yxy,\quad \overline{F}^+=\frac{1}{s_3}xyx,\quad \overline{F}^-=\frac{1}{s_3}xyy,
 \end{align*}
where $Z$ is the grading element and the rest generates $\ff=\fg_{0}^{ss}=\fsl(2|1)$.

This Lie superalgebra is represented in the linear spaces $\fg_{-1}$ and $\fg_{-1}^*$, with 
the bases and weights given in Table \ref{tab:g-1_weights}.
We refer to Appendix \ref{appendixsl21} for the basics of representation theory of $\ff$.

\begin{table}[h]\centering
    \begin{tabular}{|c|c||c|c|}\hline
Vector & Weight & Vector & Weight \\ \hline
$e_1=Y_2$ & $(b-1,0)$ & $\omega^1$ & $(-b+1,0)$ \\ \hline
$e_2=Y_1$ & $(b+1,0)$ & $\omega^2$ & $(-b-1,0)$\\ \hline
$e_3=yyx$ & $(b,1)$ & $\omega^3$ & $(-b,-1)$\\ \hline
$e_4=yyy$ & $(b,-1)$ & $\omega^4$ & $(-b,1)$\\ \hline
    \end{tabular}
\caption{Bases of $\fg_{-1}$ and $\fg_{-1}^\ast$ with $(I,H)$-weights, $b=-\frac{s_1}{s_3}+\frac{s_2}{s_3}$.}
    \label{tab:g-1_weights}
\end{table}

 \begin{lemma}\label{Le2}
There are precisely two Lie subalgebras $\fk$ with $\fg_0\subset\fk\subset\fgl(\fg_{-1})$, namely
 \begin{eqnarray*}
\fk_1&=&\fg_0+\langle \varphi_1=e_1\otimes\omega^2, \varphi_2=-s_1e_1\otimes\omega^4
   +e_3\otimes\omega^2, \varphi_3=-s_1e_1\otimes\omega^3-e_4\otimes\omega^2\rangle,\\
\fk_2&=&\fg_0+\langle\varphi_4=e_2\otimes\omega^1, \varphi_5=\frac{s_2}{s_3}e_2\otimes\omega^4
   -\frac{1}{s_3}e_3\otimes\omega^1, \varphi_6=-\frac{s_2}{s_3}e_2\otimes\omega^3
   -\frac{1}{s_3}e_4\otimes\omega^1\rangle.
 \end{eqnarray*}
 \end{lemma}

 \begin{proof}
Let us first describe how $\fg_0=\fgl(2|1)$ is embedded in $\fgl(2|2)=\fgl(\fg_{-1})$.

In terms of the representation theory of $\ff$ (Appendix \ref{appendixsl21}), we have $\fg_{-1}=\rho(b,1)$, 
i.e.\ it is an irreducible representation of dimension 4, with highest $H$-weight $1$ and corresponding $I$-weight 
$b=-2\frac{s_1}{s_3}-1=1+\frac2a$. Also, $\fg_{-1}^*=\rho(-b,1)$.
The tensor product $\fg_{-1}\otimes \fg_{-1}^*$ by \cite{Frappat} is equal to:
 \begin{equation*}
\rho(b,1)\otimes\rho(-b,1)=\rho(0,2)\oplus\rho(0;-1,1;0),
 \end{equation*}
where $\rho(0;-1,1;0)$ is a non-completely reducible $\fsl(2|1)$-representation that decomposes as 
semi-direct sums of $\fsl(2|1)$ irreducible representations:
\begin{equation*}
 \rho(0;-1,1;0)=\rho(0,0)\begin{array}{l}
                             \niplus \rho_+(1)\\
                             \niplus \rho_-(1)
                             \end{array}
                             \niplus\rho(0,0).
\end{equation*}
The semi-direct sum symbol $\niplus$ means that the space to the left is an $\ff$-submodule 
(and the space to the right is a quotient), while
 \begin{equation*}
\rho_+(1)=\pi_1(1)\oplus\pi_0(2),\qquad\rho_-(1)=\pi_1(-1)\oplus\pi_0(-2).
 \end{equation*}

We want to find all Lie superalgebras $\fk$ with $\fg_0\subset\fk\subset\fgl(\fg_{-1})$, and those define
$\fg_0$-submodules $V$ of $\fgl(\fg_{-1})/\fg_0$. Thus we proceed by identifying those submodules $V$.
 
Since the grading element $Z$ acts trivially on $\fgl(\fg_{-1})$, $Z$ generates the invariant $\rho(0,0)$ 
in $\rho(0;-1,1;0)$, and
  \begin{equation*}
\frac{\fgl(\fg_{-1})}{\fg_0}\cong\frac{\rho(0,2)\oplus\left(\rho(0,0)
\begin{array}{l}\niplus \rho_+(1)\\ \niplus \rho_-(1)\end{array}\niplus\rho(0,0)\right)}
{\rho(0,0)\oplus\rho(0,2)}\cong
\begin{array}{l}\rho_+(1)\\ \rho_-(1)\end{array}\niplus\rho(0,0)
  \end{equation*}
has precisely two non-trivial $\fg_0$-invariant subspaces. 
Up to scalar multiple, the only vector in $\fgl(\fg_{-1})$ with weight $(-2,0)$ 
is $\varphi_1=e_1\otimes\omega^2\in\pi_0(-2)\subset\rho_-(1)$, $\varphi_1\not\in\fg_0$. 
Recalling that the dual action is given via the negative of the super-transpose (cf.\ \cite{Frappat}), we compute
 \begin{equation*}
\varphi_2=F^+\cdot\varphi_1=-s_1e_1\otimes\omega^4+e_3\otimes\omega^2,\quad
\varphi_3=F^-\cdot\varphi_1=-s_1e_1\otimes\omega^3-e_4\otimes\omega^2, 
 \end{equation*}
and also $\overline{F}^+\cdot\varphi_1=\overline{F}^-\cdot\varphi_1=0$. The three vectors 
$\varphi_1,\varphi_2,\varphi_3$ generate a subspace $V_1$ with $[V_1,V_1]=0$, $[\fg_0,V_1]\subset V_1$. 
Therefore 
$\mathfrak{k}_1=\fg_0+V_1$ is a Lie superalgebra and $\fg_0\subset \mathfrak{k}_1\subset\fgl(\fg_{-1})$. 

Similarly, $\varphi_4=e_2\otimes\omega^1$, $\varphi_5=\overline{F}^+\cdot\varphi_4=\frac{s_2}{s_3}e_2\otimes\omega^4-\frac{1}{s_3}e_3\otimes\omega^1$ and
$\varphi_6=\overline{F}^-\cdot\varphi_4=-\frac{s_2}{s_3}e_2\otimes\omega^3-\frac{1}{s_3}e_4\otimes\omega^1$ generate an Abelian superalgebra and $\fg_0$-module $V_2$. Therefore
we obtain a Lie superalgebra $\mathfrak{k}_2=\fg_0+V_2$ and $\fg_0\subset\fk_2\subset\fgl(\fg_{-1})$.
 \end{proof}

 \begin{remark}
Geometrically, the reduction of the structure group to $G_0=COSp(2|2)\subset GL(2|2)$ is given 
by a $G_0=\exp\fg_0$ principal bundle $\mathfrak{F}_0\to M$, see \cite{KST2}.
The embedding of the corresponding Lie superalgebras $\fg_0\subset\fgl(\fg_{-1})$, or equivalently,
$\mathfrak{gl}(1|2)=\mathfrak{cosp}(2|2)\subset\mathfrak{gl}(2|2)$
is encoded by the following matrix:
 \begin{equation}\label{sl13=so22}
\begin{pmatrix}
  a_1 & 0 & -s_2b_4 & s_2b_3\\
  0 & a_2 & -s_1b_2 & s_1b_1\\
  b_1 & b_3 & a_1+a_3 & a_4\\
  b_2 & b_4 & a_5 & a_2-a_3
\end{pmatrix}.
 \end{equation}
where $a_1,\dots,a_5$ are even coordinates, $b_1,\dots,b_4$ are odd coordinates
(these describe the embedding through the functor of points as is customary in supergeometry \cite{Va,CCF},
see also \cite[\S2.2]{KST2}). 
For the parameters $s_1,s_2,s_3$ ($s_1+s_2+s_3=0$), only their combination $a=s_3/s_2\in\mathbb{P}^1/S_3$ 
is essential ($a\neq0,-1,\infty$).
 \end{remark}
 
 Figure \ref{F:g0e1e2} gives a diagram for the action of $\fg_0$ in $\fg_{-1}$.
 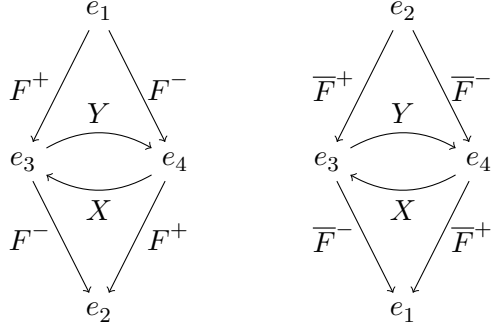
\begin{figure}[h]\centering
\begin{tikzpicture}
  \node (e1l) at (0, 2) {$e_1$};
  \node (e3l) at (-1, 0) {$e_3$};
  \node (e4l) at (1, 0) {$e_4$};
  \node (e2l) at (0, -2) {$e_2$};
  \draw[->] (e1l) -- node[left] {$F^+$} (e3l);
  \draw[->] (e1l) -- node[right] {$F^-$} (e4l);
  \draw[->] (e3l) -- node[left] {$F^-$} (e2l);
  \draw[->] (e4l) -- node[right] {$F^+$} (e2l);
  \draw[->] (e3l) to[bend left] node[above] {$Y$} (e4l);
  \draw[->] (e4l) to[bend left] node[below] {$X$} (e3l);
  \node (e2r) at (4, 2) {$e_2$};
  \node (e3r) at (3, 0) {$e_3$};
  \node (e4r) at (5, 0) {$e_4$};
  \node (e1r) at (4, -2) {$e_1$};
  \draw[->] (e2r) -- node[right] {$\overline{F}^-$} (e4r);
  \draw[->] (e2r) -- node[left] {$\overline{F}^+$} (e3r);
  \draw[->] (e3r) -- node[left] {$\overline{F}^-$} (e1r);
  \draw[->] (e4r) -- node[right] {$\overline{F}^+$} (e1r);
  \draw[->] (e3r) to[bend left] node[above] {$Y$} (e4r);
  \draw[->] (e4r) to[bend left] node[below] {$X$} (e3r);
\end{tikzpicture}
  \caption{The $\fg_0$-orbits through $e_1$ and $e_2$.}
  \label{F:g0e1e2}
 \end{figure}

 \begin{proposition}\label{supervariety} \quad
 \begin{enumerate}
 \item The $G_0$-orbit $\mathcal{V}_1$ through the line $[e_1]$ is preserved by $\fk_1$, but not $\fk_2$.
 \item The $G_0$-orbit $\mathcal{V}_2$ through the line $[e_2]$ is preserved by $\fk_2$, but not $\fk_1$.
 \end{enumerate}
 \end{proposition}

In what follows, we generate a supervariety by the action of the exponential on a supervector. 
Following \cite[(2.18)-(2.19)]{G3super}, we use the convention of letting column vectors denote {\em right}-coordinates with respect a given basis (and similarly the columns of a matrix describe right-coordinates for the images of corresponding basis vectors).  This is so that matrix multiplication on the left works properly and commutes with right scalar multiplication.  With this convention, composition of linear operators corresponds to matrix multiplication as usual.  Differentiation acts from the left, and commutes with right scalar multiplication.  We will use these conventions throughout.

 \begin{proof}
Let $V=\fg_{-1}=\langle e_1=Y_2,e_2=Y_1\,|\,e_3=yyx,e_4=yyy\rangle$ and $\ff$ be as above.

(1)
Consider the point $[e_1]\in\mathbb{P}(V_{\bar0})$, $\fq\subset\ff$ its stabilizer parabolic subalgebra and 
the corresponding grading $\ff=\ff_{-1}\oplus\ff_0\oplus\ff_1$. We have $\ff_{-1}=\{F^+,F^-\}$, where
 \begin{equation*}
   F^+|_{\fg_{-1}}=\begin{pmatrix}
               0 & 0 & 0 & 0\\
               0 & 0 & 0 & s_1\\
               1 & 0 & 0 & 0\\
               0 & 0 & 0 & 0\\
              \end{pmatrix},\quad
              F^-|_{\fg_{-1}}=\begin{pmatrix}
               0 & 0 & 0 & 0\\
               0 & 0 & s_1 & 0\\
               0 & 0 & 0 & 0\\
               -1 & 0 & 0 & 0\\
              \end{pmatrix},
  \end{equation*}
 and since columns describe {\em right}-coordinates with respect to the given basis, we have
 \begin{equation*}
 e^{\theta F^+}|_{\fg_{-1}}=\begin{pmatrix}
               1 & 0 & 0 & 0\\
               0 & 1 & 0 & s_1\theta\\
               -\theta & 0 & 1 & 0\\
               0 & 0 & 0 & 1\\
              \end{pmatrix},\quad
 e^{\varphi F^-}|_{\fg_{-1}}=\begin{pmatrix}
               1 & 0 & 0 & 0\\
               0 & 1 & s_1\varphi & 0\\
               0 & 0 & 1 & 0\\
               \varphi & 0 & 0 & 1\\
              \end{pmatrix},
  \end{equation*}
where $\theta$ and $\varphi$ are odd parameters, and
 \begin{equation*}
\begin{pmatrix}
1 \\ 0\\ 0 \\ 0
\end{pmatrix}\overset{e^{\theta F^+}}{\longmapsto}
\begin{pmatrix}
1 \\ 0 \\ -\theta \\ 0
\end{pmatrix}\overset{e^{\varphi F^-}}{\longmapsto}
\begin{pmatrix}
1 \\ s_1\theta\varphi \\ -\theta \\ \varphi
\end{pmatrix}=l(\theta,\varphi).
 \end{equation*}
We have the affine tangent space
 \begin{equation*}
\widehat{T}^{(1)}_{[e_1]}\mathcal{V}_1=\Big\langle l(\theta,\varphi), \frac{\p}{\p\theta}l(\theta,\varphi),
\frac{\p}{\p\varphi}l(\theta,\varphi)\Big\rangle\bigg|_{\substack{\theta=0\\ \varphi=0}}=
\left\langle\begin{pmatrix}1 \\ 0 \\ 0 \\ 0\end{pmatrix},
\begin{pmatrix}0 \\ 0 \\ 1 \\ 0\end{pmatrix},
\begin{pmatrix}0 \\ 0 \\ 0 \\ 1\end{pmatrix}\right\rangle
 \end{equation*}
and
 \begin{equation*}
\widehat{T}^{(2)}_{[e_1]}\mathcal{V}_1=\Big\langle l(\theta,\varphi), \frac{\p}{\p\theta}l(\theta,\varphi), 
\frac{\p}{\p\varphi}l(\theta,\varphi), \frac{\p^2}{\p\theta\p\varphi}l(\theta,\varphi)\Big\rangle
\bigg|_{\substack{\theta=0\\ \varphi=0}}=V.
 \end{equation*}
We obtain the $\fq$-invariant filtration:
 \begin{equation*}
[e_1]= \widehat{T}^{(0)}_{[e_1]}\mathcal{V}_1\subset \widehat{T}^{(1)}_{[e_1]}\mathcal{V}_1\subset  \widehat{T}^{(2)}_{[e_1]}\mathcal{V}_1=V.
 \end{equation*}
For all $m\in\mathcal{V}_1$ and all $r\in\fk_1$, we have $r\cdot m\in\widehat{T}^{(1)}_{[e_1]}\mathcal{V}_1$, 
but $\varphi_4\cdot[e_1]\notin \widehat{T}^{(1)}_{[e_1]}\mathcal{V}_1$ (notation $\varphi_4$ is from the proof of
Lemma \ref{Le2}). Therefore, $\mathcal{V}_1$ is preserved by $\fk_1$, but not $\fk_2$.

(2)
Consider the point $[e_2]\in\mathbb{P}(V_{\bar0})$, $\fq\subset\ff$ its stabilizer parabolic subalgebra 
and the corresponding grading $\ff=\ff_{-1}\oplus\ff_0\oplus\ff_1$. 
We have $\ff_{-1}=\{\overline{F}^+,\overline{F}^-\}$, where
 \begin{equation*}
\overline{F}^+|_{\fg_{-1}}=\begin{pmatrix}
               0 & 0 & 0 & -s_2/s_3\\
               0 & 0 & 0 & 0\\
               0 & -1/s_3 & 0 & 0\\
               0 & 0 & 0 & 0\\
              \end{pmatrix},\quad
              \overline{F}^-|_{\fg_{-1}}=\begin{pmatrix}
               0 & 0 & s_2/s_3 & 0\\
               0 & 0 & 0 & 0\\
               0 & 0 & 0 & 0\\
               0 & -1/s_3 & 0 & 0\\
              \end{pmatrix}
  \end{equation*}
and
  \begin{equation*}
\begin{pmatrix}
0 \\ 1\\ 0 \\ 0
\end{pmatrix}\overset{e^{\tau \overline{F}^+}}{\longmapsto}
\begin{pmatrix}
0 \\ 1 \\ \tau/s_3 \\ 0
\end{pmatrix}\overset{e^{\nu \overline{F}^-}}{\longmapsto}
\begin{pmatrix}
-s_2\tau\nu/s_3^2 \\ 1 \\ \tau/s_3 \\ \nu/s_3
\end{pmatrix}
 \end{equation*}
where $\tau$ and $\nu$ are odd parameters. Then we obtain similarly to case (1):
 \begin{equation*}
\widehat{T}_{[e_2]}^{(1)} \mathcal{V}_2 =\left\langle\begin{pmatrix}
0 \\ 1 \\ 0 \\ 0
\end{pmatrix},\begin{pmatrix}
0 \\ 0 \\ 1 \\ 0
\end{pmatrix},
\begin{pmatrix}
0 \\ 0 \\ 0 \\ 1
\end{pmatrix}\right\rangle\qquad\text{ and }\qquad  \widehat{T}^{(2)}_{[e_2]} \mathcal{V}_2 =V.
 \end{equation*}  
  
For all $m\in\mathcal{V}_2$ and all $r\in\fk_2$, we have $r\cdot m\in\widehat{T}^{(1)}_{[e_2]}\mathcal{V}_2$, 
but $\varphi_1\cdot[e_2]\notin\widehat{T}^{(1)}_{[e_2]}\mathcal{V}_2$ (notation $\varphi_1$ is from the proof 
of Lemma \ref{Le2}). Therefore, $\mathcal{V}_2$ is preserved by $\fk_2$, but not $\fk_1$.
 \end{proof}

Now we can encode the $G_0$-reduction of the structure group through the reduced supervariety 
$\mathcal{V}=\mathcal{V}_1\cup\mathcal{V}_2$ in $\mathbb{P}^{1|2}=\mathbb{P}(\C^{2|2})$;
each $\mathcal{V}_i$ is an odd quadratic variety, which allows to call $\mathcal{V}$ an odd bi-quadric. 
By Proposition \ref{supervariety}, the description of components of $\mathcal{V}$ is the following:
 \begin{equation*}
\mathcal{V}_1=\{[x_1:x_2:\xi_1:\xi_2]=[1:(1+a)\theta\varphi:\theta:\varphi]\},\quad
\mathcal{V}_2=\{[x_1:x_2:\xi_1:\xi_2]=[-\tau\nu:a^2:a\tau:a\nu]\},
 \end{equation*}
where $\theta,\varphi,\tau,\nu$ are odd parameters, i.e.\ the supervarieties are ``images'' of some 
homogeneous morphisms $\C^{1|2}\to\C^{2|2}$ (then the expressions of $x_1,...,\xi_2$ 
are pullback functions), that is, the map $\mathbb{P}^{0|2}\to\mathbb{P}^{1|2}$. 
Note that they can be described by the following homogeneous relations:
 \begin{equation}\begin{aligned}\label{Vsupervar}
\mathcal{V}_1: &&\quad x_1x_2-(1+a)\xi_1\xi_2=0,\ x_2\xi_1=0,\ x_2\xi_2=0,\ x_2^2=0;\\ 
\mathcal{V}_2: &&\quad x_1x_2+\xi_1\xi_2=0,\ x_1\xi_1=0,\ x_1\xi_2=0,\ x_1^2=0.
 \end{aligned}\end{equation}
 
 \begin{remark}
The bi-quadric $\mathcal{V}$ satisfies the relation $(x_1x_2-(1+a)\xi_1\xi_2)\cdot(x_1x_2+\xi_1\xi_2)=0$,
however it is not characterized by it: the ring $\mathcal{O}_M$ has zero-divisors, and the other defining equations are crucial.
 \end{remark}

Define the cone structure 
$\mathcal{V}_M:=M^{2|2}\times\mathcal{V}\subset\mathbb{P}\mathcal{T}(M^{2|2})$ 
as the subbundle of the (projectivized) tangent bundle of $M^{2|2}$ obtained by the reduction 
$\mathcal{V}=\mathcal{V}_1\cup\mathcal{V}_2$ translated via the flat connection given by 
commuting vector fields $\p_{x_i},\p_{\xi_j}$.

 \begin{theorem}\label{Tpi2}
The symmetry of the cone structure $\mathcal{V}_M$ is the Lie superalgebra $D(2,1;a)$. 
 \end{theorem} 
 
 \begin{proof}
This straightforwardly follows from the above results and the computation of Tanaka--Weisfeiler prolongation
from Section \ref{pi2}. Indeed, by construction, the symmetries of $\mathcal{V}_M$ contain supertranslations, 
which form $\fg_{-1}$, and linear symmetries of the cone, which form $\fg_0$. Next computation of
symmetries contains the prolongation $\fg_1$ and, in fact, coincides with it as, by \cite{KST2}, the
Tanaka--Weisfeiler prolongation majorizes the symmetry algebra.
 \end{proof}

\subsection{Superspace $M^{2|4}=G/P_{12}^\I$}\label{Gpi12}

This superspace $M^{2|4}$ is a $\mathbb{P}^1$-bundle over $M^{1|4}$.
The fiber consists of self-dual 2-planes $\Pi^2\subset\mathcal{C}$ in the odd contact distribution. Indeed,
since the natural conformal symplectic structure on $\mathcal{C}$ is, in classical sense, a conformal metric,
this identification is a twistor correspondence, usually denoted $M^{2|4}=Z(M^{1|4})$,
however in non-holonomic sense, precisely as was considered in \cite{KrMa}. 
(With this it becomes an instance of Legendrian contact structure, i.e.\ parabolic geometry of type $(A_3,P_{13})$,
with the parity of $\fg_{-1}$ inverted.)

To construct $M^{2|4}$ from $M^{1|4}$, select one of the one-parametric families of 
$\alpha$- or $\beta$-planes, i.e.\ self-dual or anti-self dual null planes of half dimension with respect to
the conformal metric on $\mathcal{C}$, which swap upon the change of orientation of $\mathcal{C}$. 
Equivalently, select one of the two $\mathfrak{sl}(2)$ factors in the semisimple part $\mathfrak{so}(4)$ 
of $\g_0$, or choose a cross on one of two white nodes of the Dynkin diagram of $D(2,1;a)$; 
those interchange by an outer automorphism $\fp_{12}^I\simeq\fp_{13}^I$.

In an affine chart, $M^{2|4}$ can be identified with the jet-space $J^2(\C^{0|2},\C^{1|0})$ that fibers over 
$J^1(\C^{0|2},\C^{1|0})$; indeed the entire supermanifold can be identified with the ``compact version'' 
of the jet-space (here compact refers only to even coordinates, while in odd coordinates the superspace 
is ``cylindrical'').

As in the classical Lie-B\"acklund theorem (cf.\ \cite{KL}), the canonical transformations of $J^2$ are 
prolongations of contact vector fields on $J^1$, that was investigated in Section \ref{pi1}.
To get explicit formulae for the generating functions of the vector fields, we have to change the coordinates
 $$
(x,\xi_1,\xi_2,\xi_3,\xi_4)\mapsto(y,\theta_1,\theta_2,\nu_1,\nu_2):=
\bigl(x-i(\xi_1\xi_2+\xi_3\xi_4),\xi_1+i\xi_2,\xi_3+i\xi_4,\xi_1-i\xi_2,\xi_3-i\xi_4\bigr),
 $$
where $i=\sqrt{-1}$. In new coordinates the Darboux form \eqref{Darboux1} becomes
 \begin{equation}\label{Darboux2}
\sigma=dy+\nu_1\,d\theta_1 +\nu_2\,d\theta_2=dy-d\theta_1\cdot\nu_1-d\theta_2\cdot\nu_2.
 \end{equation}

Denoting $D_{\theta_j}=\p_{\theta_j}+\nu_j\p_y$, then similarly to \eqref{Xf},
the contact vector field corresponding to $f\in\mathcal{O}_M$ has the form 
 \begin{equation}\label{Xf2}
X_f=f\p_y +(-1)^{|f|}\sum_{j=1}^2D_{\theta_j}(f)\p_{\nu_j}+\p_{\nu_j}(f)D_{\theta_j}.
 \end{equation}
The Lagrange--Jacobi bracket, analogous to \eqref{JLB}, becomes
 \begin{equation}\label{JLB2}
\{f,g\}=f\p_yg-(-1)^{|f|\,|g|}g\p_yf+(-1)^{|f|}\sum_{j=1}^2
(\p_{\nu_j}f)(D_{\theta_j}g)+(D_{\theta_j}f)(\p_{\nu_j}g).
 \end{equation}

The contact algebra $\mathfrak{c}\simeq\mathfrak{j}$ (with a shift of gradation) on $J^1=M^{1|4}$ naturally 
prolongs to a transitive action on $J^2=M^{2|4}$, but now we choose the following weights on generating functions: 
$w(1)=0$, $w(\theta_j)=1$, $w(\nu_j)=2$, $w(y)=3$, and let $w(X_f)=w(f)-3$. 
With these specifications we get a $|3|$-grading on the vector fields representing $\g=D(2,1;a)$. 

The following shows the initial gradation of $\mathfrak{c}\supset\g$
(and $\op{s}\!.\!\dim$ specifies both even and odd dimensions):
 \begin{center}
 \begin{tabular}{c|c|c}
$k$ &  $\mathfrak{c}_{k}$ & $\op{s}\!.\!\dim\mathfrak{c}_{k}$\\
\hline
$-3$ & 1 & $(1|0)$ \\
$-2$ & $\theta_1,\theta_2$ & $(0|2)$ \\
$-1$ & $\theta_1\theta_2$, $\nu_1$, $\nu_2$ & $(1|2)$ \\
$0$ & $y$, $\theta_1\nu_1,\theta_1\nu_2,\theta_2\nu_1,\theta_2\nu_2$ & $(5|0)$ \\
$1$ & $y\theta_1,y\theta_2$, $\nu_1\nu_2$, $\theta_1\theta_2\nu_1,\theta_1\theta_2\nu_2$ & $(1|4)$ \\
$2$ & $y\theta_1\theta_2$, $y\nu_1,y\nu_2$, $\theta_1\nu_1\nu_2,\theta_2\nu_1\nu_2$ & $(1|4)$ \\
$3$ & $y^2$, $y\theta_1\nu_1$, $y\theta_1\nu_2$, $y\theta_2\nu_1,y\theta_2\nu_2$, $\theta_1\theta_2\nu_1\nu_2$ & $(6|0)$
 \end{tabular}
 \end{center}

With respect to $\g_0=\C\oplus\mathfrak{gl}(2)$ the space $\mathfrak{c}_1$ decomposes into $\C^{1|0}$ 
and two standard representations of changed parity $\C^{0|2}\oplus\C^{0|2}=\C^{0|2}\otimes\C^2$.
Our higher order structure reduction effects in reducing the latter to one 
$\C^{0|2}_a=\langle y\theta_1+a\theta_1\theta_2\nu_2,y\theta_2-a\theta_1\theta_2\nu_1\rangle$,
so the reduction is given by $\g_1=\C^{1|0}(\nu_1\nu_2)\oplus\C^{0|2}_a\subset\mathfrak{c}_1$.

This prolongs to $\g_2=\langle y\nu_1+(a+1)\theta_2\nu_1\nu_2,y\nu_2-(a+1)\theta_1\nu_1\nu_2\rangle$,
$\g_3=\langle y^2-(\theta_1\nu_1+\theta_2\nu_2)y-(a+1)\theta_1\theta_2\nu_1\nu_2\rangle$
and $\g_i=0$ for $i>3$.
This Lie superalgebra is isomorphic to $D(2,1;a)$. 
In fact, identifying the bases we get the expressions $s_1=-1,s_2=1+a,s_3=-a$, whence
the parameter is $s_3/s_2=-\frac{a}{a+1}$ but this is $S_3$-equivalent to $a$.

The generating functions of $\fg$ are determined by the following super PDE system:
 \begin{gather}
\bigl((a+1)\p_y^2+2\p_{\theta_1}\p_{\theta_2}\p_{\nu_1}\p_{\nu_2}\bigr)(f)=0,\quad
\p_y\p_{\theta_1}\p_{\nu_2}(f)=0,\quad \p_y\p_{\theta_2}\p_{\nu_1}(f)=0,\notag\\ 
\p_y\p_{\theta_1}\p_{\theta_2}(f)=0,\qquad \p_y\p_{\nu_1}\p_{\nu_2}(f)=0,\qquad
\p_y\p_{\theta_1}\p_{\nu_1}(f)=\p_y\p_{\theta_2}\p_{\nu_2}(f),\label{PDEsys2}\\
\bigl((a+1)\p_y\p_{\nu_k}\!-(-1)^k\p_{\theta_k}\p_{\nu_1}\p_{\nu_2}\bigr)(f)=0,\quad
\nu_k\bigl(a\p_y\p_{\theta_k}\!-(-1)^k\p_{\theta_1}\p_{\theta_2}\p_{\nu_k}\bigr)(f)=0.\notag
 \end{gather}

 \begin{theorem}\label{Tpi12}
The Lie superalgebra $D(2,1;a)$ is the symmetry subalgebra of the differential system $\mathcal{E}^\infty$ 
defined by equations \eqref{PDEsys2} within the algebra of (lifted, prolonged) 
vector fields $\mathfrak{cont}^{(\infty)}(M)\subset\mathfrak{vect}(J^\infty M)$.
  \end{theorem}

The proof is similar to that of Theorem \ref{Tpi1} and hence omitted. 

 \begin{remark}\label{Rkp1p12}
We can rewrite the generators of $\g$ from Section \ref{Gpi1} in the new contact coordinates of this section.
In particular, $\g_{-2}=\langle1\rangle$, $\g_{-1}=\langle\theta_1,\theta_2,\nu_1,\nu_2\rangle$,
$\g_0=\langle y,\theta_1\theta_2,\theta_1\nu_1,\theta_1\nu_2,\theta_2\nu_1,\theta_2\nu_2,\nu_1\nu_2\rangle$,
and then we get the reduction
 $$
\g_1= \langle y\theta_1+\tfrac{\varepsilon-1}2\theta_1\theta_2\nu_2,y\theta_2+\tfrac{\varepsilon-1}2\theta_1\theta_2\nu_1,
y\nu_1+\tfrac{\varepsilon+1}2\theta_2\nu_1\nu_2,y\nu_2-\tfrac{\varepsilon+1}2\theta_1\nu_1\nu_2\rangle.
 $$
Since $\{\g_{-1},\g_2\}=\g_1$, this gives a relation between the reductions of $\fp_1^\I$ and $\fp_{12}^\I$ through regrading 
of $\g$ using $\theta_1\theta_2$: the grading element changes from 
$Z_1=2y-\theta_1\nu_1-\theta_2\nu_2$ for $\fp_1^\I$ grading of $\g$ to 
$Z_{12}=y+Z_1$ for $\fp_1^\I$ grading of $\g$, and $\theta_1\theta_2$ is the eigenvector of minimal eigenvalue ($+1$) 
for the operator $\op{ad}_y=\{y,\cdot\}$ in $\fg_0$ of $\fp_1^\I$.
It also relates with the parameter $a$ of $D(2,1;a)$ because its $S_3$ orbit is
 $$
\Bigl\{\frac{\varepsilon-1}2,\frac2{\varepsilon-1},\frac{\varepsilon+1}{-2},\frac{-2}{\varepsilon+1},
\frac{1-\varepsilon}{1+\varepsilon},\frac{1+\varepsilon}{1-\varepsilon}\Bigr\}.
 $$
 \end{remark}


\subsection{Superspace $M^{3|3}=G/P_{23}^\I$}\label{Gpi23}

This is yet another $G_0$ reduction, however here $\fg_0$ is not reductive and to encode the
geometry we need some algebraic preliminaries. Let
 \[
Z=H_1+\frac{1}{2}H_2+\frac{1}{2}H_3,\quad E=-\frac{s_1}{2}H_1+\frac{s_2}{2}H_2+\frac{s_3}{2}H_3,\quad N=\frac{H_2+H_3}{2},\quad \psi^+=yxx,\quad \psi^-=xyy.
 \]
Then $\fg_0=\CC\oplus\fgl(1|1)$, with $Z\in\CC$ the grading element and 
$\fgl(1|1)=\langle E,N\,|\,\psi^+,\psi^-\rangle$ with structure equations:
 \[
[N,\psi^\pm]=\pm\psi^\pm,\quad[\psi^+,\psi^-]=E.
 \]
The basis of $\g_{-1}$ and its dual of $\g_{-1}^*$ together with their weights are given in 
Table \ref{tab:g-1_weightsp23i}:
 \begin{table}[h]\centering
    \begin{tabular}{|c|c||c|c|}\hline
        Vector & Weight & Vector & Weight \\ \hline
$e_1=Y_2$ & $(-s_2,-1)$ & $\omega^1$ & $(s_2,1)$ \\ \hline
$e_2=Y_3$ & $(-s_3,-1)$ & $\omega^2$ & $(s_3,1)$\\ \hline
$e_3=yxy$ & $(-s_3,0)$ & $\omega^3$ & $(s_3,0)$\\ \hline
$e_4=yyx$ & $(-s_2,0)$ & $\omega^4$ & $(s_2,0)$\\ \hline
    \end{tabular}
\caption{Bases of $\fg_{-1}$ and $\fg_{-1}^\ast$ with $(E,N)$-weights.}
\label{tab:g-1_weightsp23i}
 \end{table}

 \begin{remark}
The reduction $\fg_0\subset\fgl(\fg_{-1})$ is encoded by the following matrix:
 \begin{equation}
 \begin{pmatrix}
a_1 & 0 & 0 & s_2b_1\\
0 & a_2 & s_3b_1 & 0\\
0 & -b_2 & a_1+a_3 & 0\\
-b_2 & 0 & 0 & a_2-a_3
 \end{pmatrix}.
 \end{equation}
where $a_1,a_2,a_3$ parametrize the even and $b_1,b_2$ parametrize the odd parts of $\g_0$. 
The action of $\psi_-$ on $\g_{-1}$ corresponds to $b_1=1$ (rest zero) and that of $\psi_+$ to $b_2=1$ (rest zero).
Observe that $\g_{-1}$ is decomposable.
\end{remark}

\begin{figure}[h]\centering
    \begin{tikzpicture}
\node (e1) at (0, 0) {$e_1$};
\node (e2) at (2, 0) {$e_2$};
\node (e3) at (4, 0) {$e_3$};
\node (e4) at (6, 0) {$e_4$};
\draw[->] (e1) to[bend left] node[above] {$\psi^+$} (e4);
\draw[->] (e2) to[bend left] node[above] {$\psi^+$} (e3);
\draw[->] (e4) to[bend left] node[below] {$\psi^-$} (e1);
\draw[->] (e3) to[bend left] node[below] {$\psi^-$} (e2);
    \end{tikzpicture}
\caption{Action of $\fg_0$ on $\fg_{-1}$.}
\end{figure}

 \begin{definition}
The typical Kac module $\langle c,n\rangle$ is the $(1|1)$-dimensional highest weight irreducible 
$\fgl(1|1)$-representation characterized by its $E$-eigenvalue $c\in\mathbb{C}_\times$ and
$N$-eigenvalue $n\in\CC$. More precisely, in the basis denoted in \cite{GQS2007} as
$\langle n\rangle$, $\langle n-1\rangle$ it is given by 
 \begin{equation*}
E=\begin{pmatrix}c & 0\\ 0 & c\end{pmatrix}, \quad
N=\begin{pmatrix}n & 0\\ 0 & n-1\end{pmatrix}, \quad
\psi^+=\begin{pmatrix}0 & c\\ 0 & 0\end{pmatrix}, \quad
\psi^-=\begin{pmatrix}0 & 0\\ 1 & 0\end{pmatrix}.
 \end{equation*}
 \end{definition}

 \begin{definition}
The projective cover $\mathcal{P}_h(n)$, $n\in\CC$, is the four-dimensional indecomposable 
$\fgl(1|1)$-representation that admits no extension and is given by 
 \begin{equation*}
\hspace{-2pt}
N=\begin{pmatrix}
n&0&0&0\\0&n+1&0&0\\0&0&n-1&0\\0&0&0&n
 \end{pmatrix}\!\!,\,
\psi^+=\begin{pmatrix}
0&0&0&0\\1&0&0&0\\0&0&0&0\\0&0&1&0
 \end{pmatrix}\!\!,\,
\psi^-=\begin{pmatrix}
0&0&0&0\\0&0&0&0\\1&0&0&0\\0&-1&0&0
 \end{pmatrix}\!\!,\,
E=\begin{pmatrix}
0&0&0&0\\0&0&0&0\\0&0&0&0\\0&0&0&0
 \end{pmatrix}\!\!.
 \end{equation*}
The above basis $\langle n\rangle$, $\langle n+1\rangle$, $\langle n-1\rangle$, $\langle n\rangle$ in the
notations of \cite{GQS2007} has parities even-odd-odd-even.
 
Note that $\mathcal{P}_h(0)$ is the adjoint representation of $\fgl(1|1)$. Also,
 \begin{equation*}
  \mathcal{P}_h(n)=\langle n\rangle
\begin{array}{l}\niplus \langle n+1\rangle\\
\niplus \langle n-1\rangle\end{array}                             
\niplus\langle n\rangle.
 \end{equation*}
 \end{definition}

The tensor product of two typical Kac modules is obtained from the following:
 \begin{proposition}\label{product}(\cite[Eq.(13)]{GQS2007}) If $c_1\neq0,c_2\neq0$ then
\begin{equation*}
  \langle c_1,n_1\rangle\otimes\langle c_2,n_2\rangle\ = \ 
\begin{cases} 
\ \mathcal{P}_h(n_1+n_2-1)  &\text{ for }c_1+c_2=0\,,\\
\ \bigoplus_{p=0}^{1} \ \langle c_1+c_2,n_1+n_2-p\rangle &\text{ for }c_1+c_2\neq0\,.
       \end{cases}
\end{equation*}
 \end{proposition}

Similar to Section \ref{Gpi2}, we aim to identify subalgebras of $\fgl(\g_{-1})$ containing $\g_0$,
to search for a geometric reduction based on $\g_{-1}$. In what follows we assume $s_i\neq s_j$ for $i\neq j$.

 \begin{lemma}\label{quotientp23i} The following is a decomposition of $\g_0$-modules
 \begin{eqnarray*}
 \frac{\fgl(\fg_{-1})}{\fg_0}&\cong&\left(\begin{array}{l}
\langle 1\rangle\\ \langle-1\rangle\end{array}
                             \niplus\langle 0\rangle\right)
\oplus\langle s_2-s_3,0\rangle\oplus\langle s_2-s_3,1\rangle\oplus\langle s_3-s_2,0\rangle\oplus\langle s_3-s_2,1\rangle.
\end{eqnarray*}
with indecomposable parts
 \begin{eqnarray*}
\left(\begin{array}{l}\langle 1\rangle\\ \langle-1\rangle\end{array}
\niplus\langle 0\rangle\right)&=&\langle e_1\otimes\omega^1,e_4\otimes\omega^1, e_1\otimes\omega^4\rangle\\
\langle s_3-s_2,1\rangle&=&\langle e_4\otimes\omega^2, s_2e_1\otimes\omega^2 +s_3e_4\otimes\omega^3\rangle\\
\langle s_3-s_2,0\rangle&=&\langle e_1\otimes\omega^3, -e_1\otimes\omega^2-e_4\otimes\omega^3\rangle,\\
\langle s_2-s_3,1\rangle&=&\langle e_3\otimes\omega^1, s_3e_2\otimes\omega^1+s_2e_3\otimes\omega^4\rangle,\\
\langle s_2-s_3,0\rangle&=&\langle e_2\otimes\omega^4, -e_2\otimes\omega^1-e_3\otimes\omega^4\rangle.
 \end{eqnarray*}
\end{lemma}

 \begin{proof}
We have $\fg_0$-invariant split $\fg_{-1}=\langle e_1,e_4\rangle\oplus\langle e_2,e_3\rangle$, 
$\fg_{-1}^\ast=\langle \omega^1,\omega^4\rangle\oplus\langle \omega^2,\omega^3\rangle$, where
 \begin{equation*}
\langle e_1,e_4\rangle=\langle-s_2,0\rangle,\quad\langle e_2,e_3\rangle=\langle-s_3,0\rangle,\quad \langle \omega^1,\omega^4\rangle=\langle s_2,1\rangle,\quad\langle \omega^2,\omega^3\rangle=\langle s_3,1\rangle.
 \end{equation*} 
Thus, using Proposition \ref{product}, we conclude
 \begin{eqnarray*}
 \frac{\fgl(\fg_{-1})}{\fg_0}&\cong&\frac{\mathcal{P}_h(0)\oplus\mathcal{P}_h(0)\oplus \langle s_2-s_3,0\rangle\oplus\langle s_2-s_3,1\rangle\oplus\langle s_3-s_2,0\rangle\oplus\langle s_3-s_2,1\rangle}{\CC\langle Z\rangle\oplus\mathcal{P}_h(0)}.
\end{eqnarray*}

The only vectors in $\fg_{-1}\otimes\fg_{-1}^\ast$ with weight $(s_3-s_2,1)$ are constant multiples of 
$e_4\otimes\omega^2$. Let $\varphi_1=e_4\otimes\omega^2\in\langle s_3-s_2,1\rangle$, 
$\varphi_2=\psi^-\cdot\varphi_1=s_2e_1\otimes\omega^2+s_3e_4\otimes\omega^3$, then 
$\mathbb{B}_1=\langle s_3-s_2,1\rangle=\langle\varphi_1,\varphi_2\rangle$. 
We can argue similarly to get bases of $\mathbb{B}_2$, $\mathbb{B}_3$ and $\mathbb{B}_4$:
 \begin{align*}
\mathbb{B}_1=  \langle s_3-s_2,1\rangle &=\langle \varphi_1=e_4\otimes\omega^2,
\psi^-\cdot\varphi_1=\varphi_2=s_2e_1\otimes\omega^2+s_3e_4\otimes\omega^3\rangle, &
\psi^-\cdot\varphi_1=0,\\
\mathbb{B}_2= \langle s_3-s_2,0\rangle &=\langle \varphi_3=e_1\otimes\omega^3,
\psi^+\cdot\varphi_3=\varphi_4=-e_1\otimes\omega^2-e_4\otimes\omega^3\rangle, &
\psi^+\cdot\varphi_3=0,\\
\mathbb{B}_3= \langle s_2-s_3,1\rangle &=\langle\varphi_5=e_3\otimes\omega^1,
\psi^-\cdot\varphi_5=\varphi_6=s_3e_2\otimes\omega^1+s_2e_3\otimes\omega^4\rangle, &
\psi^-\cdot\varphi_5=0,\\
\mathbb{B}_4= \langle s_2-s_3,0\rangle &=\langle\varphi_7=e_2\otimes\omega^4,
\psi^+\cdot\varphi_7=\varphi_8=-e_2\otimes\omega^1-e_3\otimes\omega^4\rangle, &
\psi^+\cdot\varphi_7=0.
 \end{align*}
Notice that any vector $v\in\mathbb{B}_i$, $i=1,\dots,4$, has non-zero $E$-eigenvalue, therefore $v\notin\fg_0$. Moreover, since the grading element $Z$ acts trivially on $\fgl(\fg_{-1})$, $Z$ generates the invariant $\langle0\rangle$ in $\mathcal{P}_h(0)$ and
 \begin{equation*}
  \frac{\fgl(\fg_{-1})}{\fg_0}\cong\frac{\left(\langle 0\rangle
\begin{array}{l}\niplus \langle1\rangle\\ \niplus \langle-1\rangle\end{array} \niplus\langle0\rangle\right)\oplus\mathbb{B}_1\oplus\mathbb{B}_2\oplus\mathbb{B}_3\oplus\mathbb{B}_4}{\CC\langle Z\rangle}\cong\mathbb{A}\oplus\mathbb{B}_1\oplus\mathbb{B}_2\oplus\mathbb{B}_3\oplus\mathbb{B}_4,
 \end{equation*}
where
 \begin{equation}\label{descriptionArep}
     \mathbb{A}=\begin{array}{l}
\langle1\rangle\\ \langle-1\rangle
    \end{array}\niplus\langle0\rangle.
 \end{equation}
Next, the action of $\fg_0$ on $\xi=e_1\otimes\omega^1\in\fgl(\fg_{-1})/\fg_0$ is given in 
Figure \ref{fig:Arep}).
 \begin{figure}[ht]\centering
    \begin{tikzpicture}
\node (top) at (0, 2) {$e_1\otimes\omega^1 \equiv e_2\otimes\omega^2\equiv -e_3\otimes\omega^3\equiv -e_4\otimes\omega^4$ mod $\fg_0$};
\node (left) at (-3, 0) {$e_4\otimes\omega^1\equiv-e_3\otimes\omega^2$ mod $\fg_0$};
\node (right) at (3, 0) {$e_1\otimes\omega^4\equiv- e_2\otimes\omega^3$ mod $\fg_0$};
\draw[->] (top) -- (left) node[midway, above left] {$\psi^-$};
\draw[->] (top) -- (right) node[midway, above right] {$\psi^+$};
    \end{tikzpicture}
\caption{Infinitesimal action of $\fg_0$ on $\xi\in\fgl(\fg_{-1})/\fg_0$.}\label{fig:Arep}
 \end{figure}

Thus, the terms of $\mathbb{A}$ in \eqref{descriptionArep} are 
$\langle 1\rangle=\langle \xi_+=e_1\otimes\omega^4\rangle$,
$\langle-1\rangle=\langle\xi_-=e_4\otimes\omega^1\rangle$ and
$\langle0\rangle=\langle\xi=e_1\otimes\omega^1\rangle$.
 \end{proof}

As a $\g_0$-module, $\fg_{-1}$ splits into two $(1|1)$-dimensional submodules $V_1\oplus V_2$, where $V_1 = \langle e_1, e_4 \rangle$ and $V_2 = \langle e_2, e_3 \rangle$, so
we get a reduction $\fg_0^\dag=\fgl(1|1)\oplus\fgl(1|1)\subset\fgl(\fg_{-1})$, corresponding
geometrically to the direct sum $\mathcal{D}=\mathcal{D}_-\oplus\mathcal{D}_+$ into $(1|1)$-dimensional
isotropic subdistributions (isotropic with respect to the tensorial bracket 
$\Lambda^2\mathcal{D}\to\mathcal{T}M/\mathcal{D}$). 

Next, we consider a further reduction to $\fg_0\subset\fg_0^\dag$.

 \begin{corollary}
There exist precisely two proper subalgebras of $\fg_0^\dag$ which contain $\fg_0$. Namely
 \begin{equation}
\mathfrak{k}_+=\fg_0+\langle\xi_+=e_1\otimes\omega^4\rangle\quad\text{and}\quad 
\mathfrak{k}_-=\fg_0+\langle\xi_-=e_4\otimes\omega^1\rangle.
 \end{equation}
 \end{corollary}

 \begin{proof}
From the proof of \ref{quotientp23i} we know:
    \begin{equation*}
        \frac{\fgl(1|1)\oplus\fgl(1|1)}{\fg_0}=\frac{\mathcal{P}_h(0)\oplus\mathcal{P}_h(0)}{\CC\langle Z\rangle\oplus\mathcal{P}_h(0)}=\mathbb{A},
    \end{equation*}
where $\mathbb{A}$ is the $(1|2)$-dimensional indecomposable $\fg_0$-representation from \eqref{descriptionArep}. 
Therefore $\mathbb{A}_\pm=\langle\xi_\pm\rangle$ define the only two proper $\fg_0$-submodules 
in $\fg_0^\dag/\fg_0$. Since $[\xi_\pm,\xi]=\pm\xi_\pm$, we conclude that $\fk_\pm=\fg_0+\langle\xi_\pm\rangle$ are the only Lie superalgebras satisfying $\fg_0\subset\fk_\pm\subset\fg_0^\dag$.
 \end{proof}
 
As the $G_0$-orbits in the projectivization of neither $\g_{-1}=V_1\oplus V_2$ nor of its components 
reveal the parameter $a$ of $D(2,1;a)$, we consider the action in the endomorphism module
 \[
\fg_{-1}\otimes\fg_{-1}^\ast=
(V_1\otimes V_1^\ast)\oplus(V_1\otimes V_2^\ast)\oplus(V_2\otimes V_1^\ast)\oplus(V_2\otimes V_2^\ast).
 \]
Since $V_1=\langle-s_2,0\rangle$, $V_2=\langle-s_3,0\rangle$ as $\g_0$-modules, we get
 \[
V_1\otimes V_1^\ast=\mathcal{P}_h(0)=V_2\otimes V_2^\ast.
 \]
The decompositions of $V_1\otimes V_2^\ast$ and $V_2\otimes V_1^\ast$ 
(for generic values of parameters) contain only typical two-dimensional modules, and hence these are
not relevant for our goals (in fact, the corresponding supervarieties carry no moduli).
We thus consider the following module
 $$
W=V_1\otimes V_1^\ast\oplus V_2\otimes V_2^\ast.
 $$
Let us call a superpoint in $W$ generic if its evaluation does not belong to $(V_1 \otimes V_1^\ast)_{\bar0}\cup(V_2 \otimes V_2^\ast)_{\bar0}$. 
 
 \begin{proposition}
The $G_0$-orbit $\mathcal{W}\subset\mathbb{P}(W)$ through a generic superpoint $o$ 
is not preserved by $\fk_1$ or $\fk_2$. Hence its symmetry superalgebra is equal to $\fg_0$.
 \end{proposition}
 
 \begin{proof} 
Consider the following basis of the $(4|4)$-dimensional space $W$
 \begin{equation}\label{basis44}
e_1\otimes\omega^1,e_4\otimes\omega^4,e_2\otimes\omega^2,e_3\otimes\omega^3 \quad|\quad e_1\otimes\omega^4,e_4\otimes\omega^1,e_2\otimes\omega^3,e_3\otimes\omega^2.
 \end{equation}
One can show that the orbit of $G_0^\dag=\exp\fg_0^\dag$ through a generic superpoint contains a diagonal element 
$a_1\cdot e_1\otimes\omega^1+a_2\cdot e_2\otimes\omega^2+a_3\cdot e_3\otimes\omega^3+a_4\cdot e_4\otimes\omega^4$, and that $[a_1:a_2:a_3:a_4]$ is an invariant of this Lie supergroup action.
(Note that, by genericity, the evaluations of both pairs $(a_1,a_4)$ and $(a_2,a_3)$ are nonzero.) 
 
We want to find the $G_0$-orbit $\mathcal{W}$ through a generic superpoint $o\in\mathbb{P}(W)$.
For simplicity, we choose
$o=[e_1\otimes\omega^1+e_2\otimes\omega^2]\in\mathbb{P}(V_1\otimes V_1^\ast\oplus V_2\otimes V_2^\ast)$
(the computations for other superpoints are similar). We note that the isotropy at that point is spanned by $Z,E,N$
and so we compute: 
 \[
\hspace{-2pt} 
e^{\theta\psi^+}_{|W}=
\begin{pmatrix}
1 & 0 & 0 & 0 & -\theta & 0 & 0 & 0 \\
0 & 1 & 0 & 0 & -\theta & 0 & 0 & 0 \\
0 & 0 & 1 & 0 & 0 & 0 & -\theta & 0 \\
0 & 0 & 0 & 1 & 0 & 0 & -\theta & 0 \\
0 & 0 & 0 & 0 & 1 & 0 & 0 & 0 \\
\theta & -\theta & 0 & 0 & 0 & 1 & 0 & 0 \\
0 & 0 & 0 & 0 & 0 & 0 & 1 & 0 \\
0 & 0 & \theta & -\theta & 0 & 0 & 0 & 1
\end{pmatrix}\!,\ \
e^{\tau\psi^-}_{|W}=\begin{pmatrix}
1 & 0 & 0 & 0 & 0 & s_2 \tau & 0 & 0 \\
0 & 1 & 0 & 0 & 0 & s_2 \tau & 0 & 0 \\
0 & 0 & 1 & 0 & 0 & 0 & 0 & s_3 \tau \\
0 & 0 & 0 & 1 & 0 & 0 & 0 & s_3 \tau \\
s_2 \tau & -s_2 \tau & 0 & 0 & 1 & 0 & 0 & 0 \\
0 & 0 & 0 & 0 & 0 & 1 & 0 & 0 \\
0 & 0 & s_3 \tau & -s_3 \tau & 0 & 0 & 1 & 0 \\
0 & 0 & 0 & 0 & 0 & 0 & 0 & 1
\end{pmatrix},
 \]
whence (we denote $\hat{o}$ an affine lift of $o$)
 \[
    \hat{o}:=\begin{pmatrix}
1\\ 0\\ 1\\ 0\\ 0\\ 0\\ 0\\ 0
    \end{pmatrix}
\overset{e^{\theta\psi^+}}{\longmapsto}
    \begin{pmatrix}
1\\ 0\\ 1\\ 0\\ 0\\ \theta\\ 0\\ \theta\\
     \end{pmatrix}
\overset{e^{\tau\psi^-}}{\longmapsto}
    \begin{pmatrix}
1+ s_2 \tau \theta  \\ s_2 \tau \theta \\ 1+s_3 \tau \theta \\
s_3 \tau \theta \\ s_2 \tau \\ \theta \\ s_3 \tau \\ \theta
    \end{pmatrix}\!,\quad
\widehat{T}_{o}\mathcal{W}=
   \left\langle\begin{pmatrix}
1 \\ 0\\ 1\\ 0\\ 0\\ 0\\ 0\\ 0
     \end{pmatrix},\begin{pmatrix}
0 \\ 0\\ 0\\ 0\\ 0\\ 1\\ 0\\ 1
      \end{pmatrix},\begin{pmatrix}
0 \\ 0\\ 0\\ 0\\ s_2\\ 0\\ s_3\\ 0
      \end{pmatrix}\right\rangle\!.
 \]
Since $\xi_+\cdot \hat{o}=-e_1\otimes\omega^4\notin\widehat{T}_o\mathcal{W}$ and
$\xi_-\cdot \hat{o}=e_4\otimes\omega^1\notin\widehat{T}_o\mathcal{W}$, 
then the supervariety $\mathcal{W}$ is preserved by $\fg_0$ but not by $\fk_1$ or $\fk_2$.
 \end{proof}
 
Note that the projections of the $G_0$ orbit through $o$ to 
$\mathbb{P}(V_1\otimes V_1^\ast)$ and $\mathbb{P}(V_2\otimes V_2^\ast)$ can have ranks 1 or 2,
according to the rank of evaluations of the corresponding endomorphisms. 
Assuming the minimal rank 1 in each case (and moreover that the corresponding nilpotent coefficients $a_k$ vanish), 
we get four possibilities, of which we restrict to the following combination of coefficients: $a_1=a_2=1$
at the even-even and $a_3=a_4=0$ at the odd-odd parts of the corresponding diagonal element. 
In other words, we choose $o = [e_1\otimes\omega^1+e_2\otimes\omega^2]$ as out generic superpoint.
 
The supervariety $\mathcal{W}=G_0\cdot o$ is described in the coordinates $(x_1,\dots,x_4|\xi_1,\dots,\xi_4)$, 
corresponding to the basis \eqref{basis44}, is given by
 $$
x_3-x_4=x_1-x_2,\ x_4=ax_2,\ \xi_3=a\xi_1,\ \xi_4=\xi_2,\ x_1x_2=\xi_1\xi_2.
 $$
In other words, it is the image of $\mathbb{P}^{0|2}$ with affine coordinates $(z|\nu,\theta)$:
 $$
\mathcal{W}= \{[z^2:-\nu\theta:z^2+(1-a)\nu\theta:-a\nu\theta:z\nu:z\theta:az\nu:z\theta]\}\subset\mathbb{P}(W).
 $$
Define the cone structure 
$\mathcal{W}_M=M^{3|3}\times\mathcal{W}\subset\mathbb{P}\op{End}\mathcal{T}(M^{3|3})$
similar to Section \ref{Gpi2}, but now use ``supertranslations'' from $M=\exp(\fm)$, corresponding to its
basis $e_1,e_2,e_3,e_4,[e_1,e_3],[e_3,e_4]$. 

Note that these vector fields are non-commutative. The corresponding nilpotent Lie supergroup 
acts simply transitively by automorphisms of $\mathcal{W}_M$. Here we use the affine version of $M^{3|3}$ 
instead of the ``compact version'' $G/P_{23}^\I$, but they are locally equivalent.

 \begin{theorem}\label{Tpi23}
The symmetry of the cone structure $\mathcal{W}_M$ is the Lie superalgebra $D(2,1;a)$. 
 \end{theorem} 
 
 \begin{proof}
This is again straightforward from the above results and the computation of the Tanaka--Weisfeiler prolongation
from Section \ref{pi23}. The proof follows the steps of Theorem \ref{Tpi2}.
 \end{proof}

\subsection{Superspace $M^{3|4}=G/P_{123}^\I$}\label{Gpi123}

This is yet another twistor construction: the superspace $M^{3|4}$ is a $\mathbb{P}^1$ bundle over
$M^{2|4}$ whose fiber is the odd projectivization of the distribution of the latter: recall that the fiber of 
$M^{2|4}$ over $M^{1|4}$ consists of isotropic odd 2-planes in $\mathcal{C}$, now we take odd lines in those
2-planes. In other words, $M^{3|4}$ is the flag variety $\op{Fl}_{1,2}(\mathcal{C})$ with the fiber over
$J^1(\C^{0|2},\C^{1|0})$ consisting of the pairs $(\ell,\Pi)$, where $\Pi$ is a self-dual 2-plane in $\mathcal{C}$
and $\ell\subset\Pi$ is a line. (This can also be considered as a certain jet-space.)

The contact algebra $\mathfrak{c}\simeq\mathfrak{j}$ on $M^{1|4}$ naturally prolongs to a transitive action 
on $M^{3|4}$. Its grading can be described in terms of coordinates on contact $M^{1|4}$ from 
Section \ref{Gpi12} through the following weights on generating functions: 
$w(1)=0$, $w(\theta_1)=1$, $w(\theta_2)=w(\nu_2)=2$, $w(\nu_1)=3$, $w(y)=4$, and $w(X_f)=w(f)-4$. 
This gives a $|4|$-grading on the vector fields representing $D(2,1;a)$. 

The following shows the initial gradation of $\mathfrak{c}\supset\g$:
 \begin{center}
 \begin{tabular}{c|c|c}
$k$ &  $\mathfrak{c}_{k}$ & $\op{s}\!.\!\dim\mathfrak{c}_{k}$\\
\hline
$-4$ & 1 & $(1|0)$ \\
$-3$ & $\theta_1$ & $(0|1)$ \\
$-2$ & $\theta_2,\nu_2$ & $(0|2)$ \\
$-1$ & $\theta_1\theta_2$, $\theta_1\nu_2$, $\nu_1$ & $(2|1)$ \\
$0$ & $y$, $\theta_1\nu_1,\theta_2\nu_2$ & $(3|0)$ \\
$1$ & $y\theta_1$, $\theta_2\nu_1,\nu_1\nu_2$, $\theta_1\theta_2\nu_2$ & $(2|2)$ \\
$2$ & $y\theta_2,y\nu_2$, $\theta_1\theta_2\nu_1,\theta_1\nu_1\nu_2$ & $(0|4)$ \\
$3$ & $y\nu_1$, $y\theta_1\theta_2$, $y\theta_1\nu_2$, $\theta_2\nu_1\nu_2$ & $(2|2)$ \\
$4$ & $y^2$, $y\theta_1\nu_1$, $y\theta_2\nu_2$, $\theta_1\theta_2\nu_1\nu_2$ & $(4|0)$
 \end{tabular}
 \end{center}

Now the space $\mathfrak{c}_1$ decomposes into $\C^{2|0}$ and $\C^{0|2}=\C^{0|1}\otimes\C^2$.
Our higher order structure reduction effects in reducing the latter to one 
$\C^{0|1}=\langle y\theta_1+a\theta_1\theta_2\nu_2\rangle$. 
The following reductions $\g_k\subset\mathfrak{c}_k$ for $k\ge1$ are given by the Tanaka--Weisfeiler prolongation, 
in particular they stop at the level $k=4$.
One may interpret the reduction via the same differential equation as in case 1.

We can also realize this algebra by vector fields locally in the chart 
$\C^{3|4}(x_1,x_2,x_3|\xi_1,\xi_2,\xi_3,\xi_4)\subset M^{3|4}$. The vector fields, corresponding to the
generators of $\m$, are the following ($u$ are even fields, $v$ are odd fields, minus superscript indicates negative
gradation; $\kappa$ is an inessential even parameter to be specified below; one may take $\kappa=1$ for simplicity).
 \begin{gather*}
u^-_1 = \p_{x_1} + \xi_1\p_{\xi_2} + \kappa\xi_1\xi_3\p_{x_3} + x_2\xi_1\p_{\xi_4},\
u^-_2 = \p_{x_2} + \xi_1\p_{\xi_3} + \xi_1\xi_2\p_{x_3} + x_1\xi_1\p_{\xi_4},\
u^-_3 = (\kappa+1)\p_{x_3}, \\
v^-_1 = \p_{\xi_1},\
v^-_2 = \p_{\xi_2} + \kappa\xi_3\p_{x_3} + x_2\p_{\xi_4},\
v^-_3 = \p_{\xi_3} + \xi_2\p_{x_3} + x_1\p_{\xi_4},\
v^-_4 = \p_{\xi_4} + (\kappa + 1)\xi_1\p_{x_3}.
 \end{gather*}
The structure equations are 
$[u^-_1,v^-_1]=v^-_2$, $[u^-_2,v^-_1]=v^-_3$, $[u^-_2,v^-_2]=v^-_4$, 
$[u^-_1,v^-_3]=v^-_4$, $[v^-_1,v^-_4]=-u^-_3$, $[v^-_2,v^-_3]=u^-_3$.
Next, $0^\text{th}$ Tanaka--Weisfeiler prolongation computed for these vector fields gives
 \begin{gather*}
u^0_1 = x_1\p_{x_1} + \xi_2\p_{\xi_2} + \xi_4\p_{\xi_4} + x_3\p_{x_3},\quad
u^0_2 = x_2\p_{x_2} + \xi_3\p_{\xi_3} + \xi_4\p_{\xi_4} + x_3\p_{x_3},\\
u^0_3 = \xi_1\p_{\xi_1} + \xi_2\p_{\xi_2} + \xi_3\p_{\xi_3} + \xi_4\p_{\xi_4} + 2x_3\p_{x_3},
 \end{gather*}
and $Z=u^0_1+u^0_2+u^0_3$ is the grading element.

The next prolongation yields the two-parameter even space
 \begin{gather*}
u^+_1 = x_1^2\p_{x_1} - \xi_2\p_{\xi_1} - (x_2\xi_2+x_1\xi_3-\xi_4)\p_{\xi_3} - x_1(x_2\xi_2-\xi_4)\p_{\xi_4} - \kappa\xi_2(x_1\xi_3-\xi_4)\p_{x_3},\\
u^+_2 = x_2^2\p_{x_2} - \xi_3\p_{\xi_1} - (x_2\xi_2+x_1\xi_3-\xi_4)\p_{\xi_2} - x_2(x_1\xi_3-\xi_4)\p_{\xi_4} - \xi_3(x_2\xi_2-\xi_4)\p_{x_3},
 \end{gather*}
and the two-parameter odd space, which we write in a linear combination with an even parameter:
 \begin{align*}
v^+_1 =&\ (a+1)(x_1\xi_1-\xi_2)\p_{x_1} + (a\kappa-1)(x_2\xi_1-\xi_3)\p_{x_2} + (a+1)\xi_1\xi_2\p_{\xi_2}\\
 &+ (a\kappa-1)\xi_1\xi_3\p_{\xi_3} + (ax_3+a(\kappa+1)\xi_1\xi_4-\xi_2\xi_3)\p_{\xi_4} + a(\kappa+1)x_3\xi_1\p_{x_3}.
 \end{align*}
Fixing this parameter we compute the further prolongations:
 \begin{align*}
v^+_2 =&\ -(a+1)x_1(x_1\xi_1-\xi_2)\p_{x_1} + (a\kappa-1)(x_1\xi_3-\xi_4)\p_{x_2} + (a+1)\xi_1\xi_2\p_{\xi_1}\\
 &+ ((a+1)\xi_1(x_2\xi_2+x_1\xi_3-\xi_4) - a(x_3+\xi_2\xi_3))\p_{\xi_3} 
 + x_1((a+1)\xi_1(x_2\xi_2-\xi_4)+\xi_2\xi_3-ax_3)\p_{\xi_4} \\
 &+ ((a+1)\kappa\xi_1(x_1\xi_2\xi_3-\xi_2\xi_4)-ax_3\xi_2)\p_{x_3},\\
v^+_3 =&\ (a+1)(x_2\xi_2-\xi_4)\p_{x_1} - (a\kappa-1)x_2(x_2\xi_1-\xi_3)\p_{x_2} + (a\kappa-1)\xi_1\xi_3\p_{\xi_1}\\
 &+ ((a\kappa-1)\xi_1(x_2\xi_2+x_1\xi_3-\xi_4) + a(\kappa\xi_2\xi_3-x_3))\p_{\xi_2} + x_2((a\kappa-1)\xi_1(x_1\xi_3-\xi_4)\\
 &+\xi_2\xi_3 -ax_3)\p_{\xi_4} - ((a\kappa-1)\xi_3(x_2\xi_1\xi_2-\xi_1\xi_4) + a\kappa x_3\xi_3)\p_{x_3},\\
v^+_4 =&\ -(a+1)x_1(x_2\xi_2-\xi_4)\p_{x_1} - (a\kappa-1)x_2(x_1\xi_3-\xi_4)\p_{x_2} + (\xi_2\xi_3-ax_3)\p_{\xi_1}
 - (a\kappa-1)\xi_2(x_1\xi_3-\xi_4)\p_{\xi_2} \\
 &- (a+1)\xi_3(x_2\xi_2-\xi_4)\p_{\xi_3} - (\kappa+1)\xi_2\xi_3\xi_4\p_{x_3},\\
u^+_3 =&\ a(a+1)(\kappa+1)(x_1x_2\xi_1\xi_2-x_1\xi_1\xi_4+\xi_2\xi_4)\p_{x_1} + a(a\kappa-1)(\kappa+1)(x_1x_2\xi_1\xi_3 - x_2\xi_1\xi_4+\xi_3\xi_4)\p_{x_2} \\
 &+ a(\kappa+1)\xi_1(ax_3-\xi_2\xi_3)\p_{\xi_1}
 + a(\kappa+1)((a\kappa-1)\xi_1\xi_2(x_1\xi_3-\xi_4)+ax_3\xi_2)\p_{\xi_2} \\
 &- a(\kappa+1)((a+1)\xi_1(x_2\xi_2\xi_3+\xi_3\xi_4) - ax_3\xi_3)\p_{\xi_3} \\
 &+ a(\kappa+1)\xi_4(ax_3-\xi_2\xi_3)\p_{\xi_4} + a(\kappa+1)((\kappa+1)\xi_1\xi_2\xi_3\xi_4+ax_3^2)\p_{x_3}.
 \end{align*}
 
 \begin{theorem}\label{Tpi123}
Vector fields $(u_i^-,u_j^0,u_k^+|v_i^-,v_j^0,v_k^+)$ realize the algebra $D\bigr(2,1;a(\kappa)\bigl)$
with the parameter $a(\kappa)=\frac{a\kappa-1}{a+1}$ in the space 
$M^{3|4}=\C^{3|4}(x_1,x_2,x_3|y_1,y_2,y_3,y_4)$ equipped with the distribution
 $$
\mathcal{D}=\langle\p_{x_1}+y_3\p_{y_4},\p_{x_2}+y_2\p_{y_4},
\p_{y_1}+x_1\p_{y_2}+x_2\p_{y_3}+x_1x_2\p_{y_3}+(\kappa x_1y_3+x_2y_2-(\kappa+1)y_4)\p_{y_4})\rangle
 $$
naturally split into the sum of two even and one odd lines as indicated by the generators.
 \end{theorem}

 \begin{proof}
The explicit form of the vector fields realizing $\fm$ follows from direct verification of the structure relations
(in other words, $M^{1|4}\simeq\exp\fm$).
Then the fields realizing $\fp$ are obtained by mimicking relations for the Tanaka--Weisfeiler prolongation,
including the higher order reduction. Comparing the brackets of $\fg_{\bar1}$ with those in $D(2,1;a)$
gives the parameters $s_l$. Finally, generators of the distribution $\mathcal{D}$ are obtained from the
condition of commutation with the generators of $\fm$: if the latter are right-invariant vector fields, the former
become left-invariant; we indicate only the basis of $\fg_{-1}$, and this basis is defined up to scale since $\fp$ is Borel.
 \end{proof}

\medskip

\subsection{Superspace $M^{3|4}_\diamond=G/P_{123}^\IV$}\label{Gpiv123}

This is the only case where the geometry is determined by the distribution. 
In fact, the distribution $\mathcal{D}^{0|3}$ on the manifold $M_\diamond^{3|4}$ is naturally split 
into the direct sum of three $(0|1)$-dimensional subdistributions, which are generated by odd vector fields.
(This is because $\fg_0=\fh$ for the Borel parabolic, but also follows from the brackets of \eqref{v1v2v3} below.)
 
Choosing coordinates $(x_{12},x_{31},x_{23}|\xi_1,\xi_2,\xi_3,\theta)$ on $M$ we define
the generators of $\mathcal{D}$: 
 \begin{gather}\label{v1v2v3}
 \begin{array}{rl}
& v_1=\p_{\xi_1}+\xi_2\p_{x_{12}}+s_1\xi_2\xi_3\p_{\theta},\\ 
& v_2=\p_{\xi_2}+\xi_3\p_{x_{23}}+s_2\xi_3\xi_1\p_{\theta},\\ 
& v_3=\p_{\xi_3}+\xi_1\p_{x_{31}}+s_3\xi_1\xi_2\p_{\theta}.
 \end{array}
 \end{gather}
Let us note that the change of variable $\theta\mapsto\theta+k\xi_1\xi_2\xi_3$ results in the change of parameters
$s_i\mapsto s_i-k$. Thus we can fix these parameters (and the coordinate freedom)
by $s_1+s_2+s_3=0$, and this will be assumed in what follows.
 
The only nontrivial brackets in \eqref{v1v2v3} are the following:
 \begin{gather}
v_{12}=[v_1,v_2] = \p_{x_{12}}+(s_1-s_2)\xi_3\p_\theta,\notag\\
v_{23}=[v_2,v_3] = \p_{x_{23}}+(s_2-s_3)\xi_1\p_\theta,\label{v123a}\\
v_{13}=[v_3,v_1] = \p_{x_{13}}+(s_3-s_1)\xi_2\p_\theta,\notag\\
[v_1,v_{23}]=(s_2-s_3)\p_{\theta},\quad [v_2,v_{31}]=(s_3-s_1)\p_{\theta},\quad
[v_3,v_{12}]=(s_1-s_2)\p_{\theta},\label{v123b}
 \end{gather}
and the vector fields $v_1,v_2,v_3, v_{12},v_{13},v_{23}, \p_\theta$ generate 
$\mathcal{T}M_\diamond^{3|4}$.

 \begin{theorem}\label{Tpiv123}
The symmetry algebra of the distribution  $\mathcal{D}=\langle v_1,v_2,v_3\rangle$ is $D(2,1;a)$.
 \end{theorem}

 \begin{proof}
The bracket relations between $v_i$ from \eqref{v1v2v3} are the same as among the generators of $\g_{-1}$ 
in the $\mathbb{Z}$-grading associated with $\fp_{123}^\IV$. 
The upper bound for symmetry follows from the computation of the Tanaka--Weisfeiler prolongation. 
The lower bound follows from the explicit realization of symmetry in Table \ref{F:sym}, 
where we omit the generators $R_2,R_3$, etc, due to cyclic symmetry in indices $1,2,3$, 
and also shorten cyclic-repeated terms in the generator of $\g_3$. 
 \end{proof}

 \begin{table}[h] 
 \[
 \begin{array}{|c|c|} \hline
 \mbox{Level} & \mbox{Symmetries}\\ \hline\hline
-3 & \p_{\theta} \\ \hline
-2 &  \p_{x_{12}}, \quad \p_{x_{23}}, \quad \p_{x_{31}} \\ \hline
-1 & S_1 = \p_{\xi_1} - \xi_3 \p_{x_{31}} + (s_3 \xi_2 \xi_3 + (s_2 - s_3) x_{23}) \p_{\theta}
\quad \text{cycl: } [S_2\dots S_3]\\ \hline
0 & Z_1 = \xi_1 \p_{\xi_1} + x_{12} \p_{x_{12}} + x_{31} \p_{x_{31}} + \theta \p_{\theta}\qquad 
\qquad\ \text{cycl: } [Z_2\dots Z_3] \\ \hline
1 & \begin{array}{ll}
R_1 =&\!\!\! (s_1 - s_2)x_{12}S_2 + (s_3 - s_1) x_{31} S_3 - \xi_1\xi_2(s_1 - s_2)\p_{\xi_2} \\
 &\!\!\!+ (\theta - s_1 \xi_1 \xi_2 \xi_3) \p_{x_{23}} - (s_1-s_2)(s_3-s_1)x_{12}x_{31}\p_{\theta}\quad
 [R_2\dots R_3]
 \end{array} \\ \hline
2 & 
 \begin{array}{ll}
R_{12} = &\!\!\! (s_1 - s_2) x_{12} (\xi_1\p_{\xi_1}+\xi_2\p_{\xi_2}-\xi_3\p_{\xi_3}+ x_{12}\p_{x_{12}} 
+\theta \p_{\theta})\\ 
 &\!\!\! +(\theta - s_2\xi_1\xi_2\xi_3) \p_{\xi_3}+ \xi_2\theta\p_{x_{23}} + s_2\xi_1\xi_2\theta\p_{\theta}\quad
\qquad [R_{23}\dots R_{31}]
 \end{array} \\ \hline
3 &
 \begin{array}{l}
\bigl( (s_1-s_2)(s_3-s_1)x_{12} (\xi_3\xi_1-x_{31}x_{12})  - (s_2-s_3)\xi_1\theta\bigr) \p_{\xi_1}\ \
(+\,\text{cycl})\\
+\bigl( (s_1 - s_2) x_{12} (s_3 \xi_1 \xi_2 \xi_3 + x_{23} (s_2 - s_3) \xi_1 - \theta)\bigr) \p_{x_{12}}\ \ 
(+\,\text{cycl})\\
+\bigl( s_2(s_2-s_3)(s_1-s_3)x_{23}x_{31}\xi_1\xi_2 + s_3 (s_3-s_1)(s_2-s_1) x_{31}x_{12}\xi_2\xi_3 \\
+s_1(s_1-s_2)(s_3-s_2)x_{12}x_{23}\xi_3\xi_1 - (s_1-s_2) (s_2-s_3)(s_3-s_1)x_{12}x_{23}x_{31}\bigr) \p_{\theta}
 \end{array} \\ \hline
 \end{array}
 \]
 \caption{Symmetries of $\mathcal{D}$; one representative for each slot $\g_{-1},\dots,\g_2$}\label{F:sym}
 \end{table}

Next, under the assumption that $s_1,s_2,s_3$ are pairwise distinct, we observe that:
 \begin{enumerate}
\item the derived distribution $\mathcal{C}:= [\mathcal{D},\mathcal{D}]=\langle v_{ij}|v_k\rangle$ is a 
mixed contact distribution;
\item $\mathcal{D}\subset\mathcal{C}$ is a non-holonomic Legendrian distribution, i.e.\ $\mathcal{D}\neq
[\mathcal{D},\mathcal{D}]\subset\mathcal{C}$.
 \end{enumerate}
The contact 1-form $\sigma\in\op{Ann}(\mathcal{C})$ is given by:
 \begin{equation}\label{E:sigma}
\sigma = d\theta - s_2\xi_2\xi_3 d\xi_1 - s_3\xi_3\xi_1 d\xi_2 - s_1\xi_1\xi_2 d\xi_3
- (s_2-s_3)\xi_1 dx_{23} - (s_3-s_1)\xi_2 dx_{31} - (s_1-s_2)\xi_3 dx_{12}.
 \end{equation}
 
Let us change the variables
 \begin{gather*}
\psi = \theta - (s_2-s_3)x_{23}\xi_1 - (s_3-s_1)x_{31}\xi_2 - (s_1-s_2)x_{12}\xi_3,\\
\psi_1 = s_2\xi_2\xi_3 - (s_2-s_3)x_{23},\ \psi_2 = s_3\xi_3\xi_1 - (s_3-s_1)x_{31},\
\psi_3 = s_1\xi_1\xi_2 - (s_1-s_2)x_{12}.
 \end{gather*}
Note that $\psi$ is odd, while $\psi_i$ are even, thus $M^{3|4}_\diamond=
\C^{3|4}(\psi_1,\psi_2,\psi_2|\xi_1,\xi_2,\xi_3,\psi)$.  
In these coordinates, contact form \eqref{E:sigma} takes the canonical form
 \begin{equation*}
\sigma= d\psi - (d\xi_1)\psi_1 - (d\xi_2) \psi_2 - (d\xi_3) \psi_3.
 \end{equation*}
In the new coordinates our superdistributions have the following generators (indices vary over $1,2,3$):
 \[
\mathcal{D}= \langle v_i = \p_{\xi_i} + \psi_i\p_\psi + \sum_j\psi_{ij}\p_{\psi_j}\rangle\subset
\mathcal{C}= \langle \p_{\psi_i}| \p_{\xi_j} + \psi_j\p_\psi\rangle,
 \]
where $\psi_{ij} = -\psi_{ji}$.  Explicitly,
 \begin{equation}\label{nPDE}
\boxed{\psi_{12} = -s_3\xi_3, \quad \psi_{23} = -s_1\xi_1, \quad \psi_{31} = -s_2\xi_2}.
 \end{equation}
These equations describe a 2nd order PDE, but this is merely a short-hand notation for $\mathcal{D}$ above.  
(Note: This PDE has no solutions!)  More precisely, we have
 \begin{gather*}
v_1 = \p_{\xi_1} + \psi_1 \p_\psi - s_3 \xi_3 \p_{\psi_2} + s_2 \xi_2 \p_{\psi_3}\\
v_2 = \p_{\xi_2} + \psi_2 \p_\psi - s_1 \xi_1 \p_{\psi_3} + s_3 \xi_3 \p_{\psi_1}\\
v_3 = \p_{\xi_3} + \psi_3 \p_\psi - s_2 \xi_2 \p_{\psi_1} + s_1 \xi_1 \p_{\psi_2}
 \end{gather*}
and compare \eqref{v123a}-\eqref{v123b}:
 \begin{gather*}
v_{12}=[v_1,v_2] = (s_2-s_1) \p_{\psi_3}, \quad
v_{23}=[v_2,v_3] = (s_3-s_2) \p_{\psi_1}, \quad
v_{31}=[v_3,v_1] = (s_1-s_3) \p_{\psi_2};\\
[v_1,v_{23}]=(s_2-s_3)\p_\psi,\quad [v_2,v_{31}]=(s_3-s_1)\p_\psi,\quad [v_3,v_{12}]=(s_1-s_2)\p_\psi.
 \end{gather*}

By naturality, any infinitesimal symmetry $S$ of $\mathcal{D}$ is also a symmetry of $\mathcal{C} = 
[\mathcal{D},\mathcal{D}]$, so $S$ is a contact vector field. This corresponds to a generating function 
$f = \iota_{S}\sigma$ as follows:
 \begin{equation}\label{oddXf}
X_f = -\sum_{i=1}^3(\p_{\psi_i} f) D_{\xi_i} + f \p_{\psi} - (-1)^{|f|}\sum_{i=1}^3 (D_{\xi_i} f) \p_{\psi_i}.
 \end{equation}
The Lagrange--Jacobi bracket, given by $[X_f,X_g]=X_{\{f,g\}}$, has the formula
 \begin{equation}\label{LJBodd}
\{f,g\} = f(\p_\psi g) +(-1)^{|f|} (\p_\psi f) g - \sum_{i=1}^3 (\p_{\psi_i} f)(D_{\xi_i} g) 
- (-1)^{|f|} \sum_{i=1}^3 (D_{\xi_i} f)(\p_{\psi_i} g).
 \end{equation}

Now, contracting the symmetry fields given by Table \ref{F:sym} with $\sigma$ given by \eqref{E:sigma}
and rewriting in new coordinates, we obtain the generating functions in Table \ref{F:gf}. 
 
 \begin{table}[h]
 \[
 \begin{array}{|c|c|cc} \hline
 \mbox{Level} & \mbox{Symmetries}\\ \hline\hline
-3 & 1\\ \hline
-2 & \xi_1, \quad \xi_2, \quad \xi_3\\ \hline
-1 & \psi_1 - s_1 \xi_2 \xi_3, \quad \psi_2 - s_2 \xi_3 \xi_1, \quad \psi_3 - s_3 \xi_1 \xi_2 \\ \hline
0 & \psi - \psi_1 \xi_1, \quad \psi - \psi_2 \xi_2,\quad \psi - \psi_3 \xi_3\\ \hline
1 & \begin{array}{l}
 \psi_2 \psi_3 + (s_2 \psi_2 \xi_2 - s_3 \psi_3 \xi_3 + (s_3-s_2)\psi) \xi_1,\\
 \psi_3 \psi_1 + (s_3 \psi_3 \xi_3 - s_1 \psi_1 \xi_1 + (s_1-s_3)\psi) \xi_2,\\
 \psi_1 \psi_2 + (s_1 \psi_1 \xi_1 - s_2 \psi_2 \xi_2 + (s_2-s_1)\psi) \xi_3
 \end{array} \\ \hline
2 & \begin{array}{l}
\psi(\psi_1 - s_1 \xi_2 \xi_3) + \psi_1(s_1 \xi_1 \xi_2 \xi_3 - \psi_2 \xi_2 - \psi_3 \xi_3),\\
\psi(\psi_2 - s_2 \xi_3 \xi_1) + \psi_2(s_2 \xi_1 \xi_2 \xi_3 - \psi_3 \xi_3 - \psi_1 \xi_1),\\
\psi(\psi_3 - s_3 \xi_1 \xi_2) + \psi_3(s_3 \xi_1 \xi_2 \xi_3 - \psi_1 \xi_1 - \psi_2 \xi_2)
 \end{array}\\ \hline
3 & \begin{array}{l}
 \psi_1 \psi_2 \psi_3 - \psi((s_2 - s_3)\psi_1 \xi_1 + (s_3 - s_1)\psi_2 \xi_2 + (s_1 - s_2) \psi_3 \xi_3)\\
 \quad + 2 s_1 \psi_2 \psi_3 \xi_2 \xi_3 + 2 s_2 \psi_3 \psi_1 \xi_3 \xi_1 + 2 s_3 \psi_1 \psi_2 \xi_1 \xi_2
 \end{array} \\ \hline
 \end{array}
 \]
 \caption{Symmetries of $\mathcal{D}$ expressed as generating functions}\label{F:gf}
 \end{table}

Now we can summarize our computations to reformulate Theorem \ref{Tpiv123} in all cases aside from the standard 
$\mathfrak{osp}(4,2)$ as follows.
 
 \begin{theorem}\label{Tpiv123plus}
If $s_1,s_2,s_3$ are pairwise distinct (and all nonzero with zero sum), 
then the contact symmetry algebra of PDE system \eqref{nPDE} is $D(2,1;a)$ with $a=s_3/s_2$. 
The vector fields are given by \eqref{oddXf} with the generating functions in Table \ref{F:gf} 
and the Lie superalgebra structure is given by \eqref{LJBodd}.
 \end{theorem}
 
Finally, let us give some words of caution, in regards to the odd contact structure, cf.\ \cite[Section D.3.2-3]{BL}.
The Lagrange--Jacobi bracket of \eqref{E:sigma} is parity-reversing, i.e.\ $|\{f,g\}| = |f|+|g|+1$.
This implies $|X_f|=|f|+1$, and furthermore the Lagrange--Jacobi identity becomes
 \[
\{f,\{g,h\}\} =  \{\{f,g\},h\} + (-1)^{(|f|+1)(|g|+1)} \{g,\{f,h\}\}.
 \]

\section{Twistor correspondences and outlook}\label{Concl}\label{S5}

\subsection{Generalized flag supervarieties.}\label{TwSec}

We have constructed six different (flat) supergeometries with symmetry $D(2,1;a)$. 
These are united into the twistor diagram in Figure \ref{TwCorr}. Along any one of these fibrations, 
the pushforward defines an isomorphism of symmetries, and vise versa there is a unique lift of symmetries
from the base to the total space. 

Below, we characterize these correspondences, namely, for each fiber bundle we indicate a construction of
the supermanifold and the superdistribution for the total space from the base and otherwise around, but we skip
discussing relations between higher-order reductions.

 \begin{figure}[h!]\begin{center}\begin{tikzpicture}
\node (T1) {$M^{3|4}=G/P_{123}^\I$};
\node (T2) [right = 2cm of T1] {$M^{3|4}_\diamond=G/P_{123}^\IV$};
\node (T3) [below left = 0.5cm and -0.2cm of T1] {$M^{2|4}=G/P_{12}^\I$};
\node (T4) [below left = 0.5cm and -0.2cm of T2] {$M^{3|3}=G/P_{23}^\I$};
\node (T5) [below left = 0.5cm and -0.2cm of T3]  {$M^{1|4}=G/P_1^\I$};
\node (T6) [below left = 0.5cm and -0.2cm of T4] {$M^{2|2}=G/P_2^\I$};
\draw[->] (T1) -- node[pos=0.25, left] {$\pi_{12}$\,\,} (T3);
\draw[->] (T1) -- node[pos=0.4, right] {\,$\pi_{23}$} (T4);
\draw[->] (T2) -- node[pos=0.25, left] {$\pi_\diamond$\,\,} (T4);
\draw[->] (T3) -- node[pos=0.3, left] {$\pi_1$\,\,} (T5);
\draw[->] (T3) -- node[pos=0.4, right] {\,$\pi_2$} (T6);
\draw[->] (T4) -- node[pos=0.3, left] {$\pi_3$\,\,} (T6);
 \end{tikzpicture}\end{center}
 \caption{Twistor correspondences of $D(2,1;a)$ flag varieties.}\label{TwCorr}
 \end{figure}
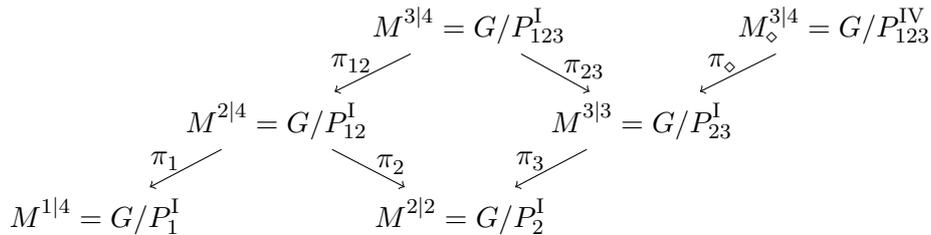

We will use the notation $\cD$ for the distribution on the base of the fibration, corresponding to $\g_{-1}\subset\fm$
(in the grading corresponding to the respective parabolic), and $\hcD$ for the analogous distribution on the total space. 
Let us recall that a Cauchy characteristic of a distribution $\cC$ on a supermanifold $M$ is a vector (or line) in the kernel 
of the bracket $\Lambda^2\cC\to\mathcal{T}M/\cC$.

\smallskip

\underline{$\pi_1$:} 
As noted in Section \ref{Gpi12}, the $\mathbb{P}^1$ bundle $\pi_1:M^{2|4}\to M^{1|4}$
is the twistor fibration, namely the fiber $\pi_1^{-1}(o)$ consists of self-dual null planes in the contact space
$\cC_o$ under a choice of its orientation. Here and beyond, by $o$ we will understand a superpoint of the base, 
given through the functor of points \cite{CCF,Va}. This gives the construction of $M^{2|4}$ through $M^{1|4}$. 

On the other hand, the distribution $\hcD$ of $M^{2|4}$ is such that $\hcD^2=[\hcD,\hcD]$ contains 
one Cauchy characteristic -- in terms of the generators from Section \ref{pi12}, this is the 
even vector $Y_2$. Quotient by it constitutes the projection $\pi_1$, which maps $\hcD^2$ to the distribution 
$\cD=\cC$ on $M^{1|4}$ and pushes forward the reductions (see Remark \ref{Rkp1p12} where the regrading via
$\op{ad}_{\theta_1\theta_2}$ corresponds to a choice of null plane in $\mathcal{C}$ and thus relates to $\pi_1$). 

\smallskip

\underline{$\pi_{12}$:} 
Similarly, as noted in Section \ref{Gpi123}, the $\mathbb{P}^1$ bundle $\pi_{12}:M^{3|4}\to M^{2|4}$ 
is the projectivization of the odd part $\cD_{\bar1}$ of $M^{2|4}$,
singled out by the condition of being the maximal linear subspace of vectors with vanishing self-bracket, 
that is the variety of odd lines in $\cD$. This gives the construction of $M^{2|4}$ through $M^{1|4}$. 

On the other hand, the second derived distribution $\hcD^3=[\hcD,\hcD^2]$ of $M^{3|4}$ contains the whole 
plane of Cauchy characteristics -- $\langle Y_2,Y_3\rangle$ in terms of generators from Section \ref{pi123}.
The two generators are singled out by brackets of $\fm$, and the quotient by $Y_2$ (resp. $Y_3$) 
gives the projection $\pi_{12}$ (resp.\ the projection to the flag supervariety corresponding to the parabolic
$\fp_{13}^\I$ outer equivalent to $\fp_{12}^\I$ modulo a change of parameter $a$), 
which maps $\hcD^3$ to the distribution $\cD^2$ on $M^{2|4}$.
This recovers both the distribution $\cD$ and the reductions.

\smallskip

\underline{$\pi_2$:} The $\mathbb{P}^{0|2}$ bundle $\pi_2:M^{2|4}\to M^{2|2}$ has fiber consisting of
$(1|0)$-dimensional lines $\ell\subset\mathcal{V}_1\subset\mathcal{T}M^{2|2}$. Here
$\mathcal{V}=\mathcal{V}_1\cup\mathcal{V}_2$ and one can equivalently restrict to the second component
of the reduced odd bi-quadric from Section \ref{Gpi2}. Indeed, such lines $\ell$ are parametrized by two odd
parameters, whence the type of the fiber. This gives the construction of $M^{2|4}$ through $M^{2|2}$. 

Conversely, the distribution $\hcD$ of $M^{2|4}$ is such that the space of (not necessary graded) 
vectors $v$ in it with $[v,v]=0\,\op{mod}\hcD$ is the union of the even line $\langle Y_2\rangle$ and 
the odd plane $\langle yxx,yxy\rangle$, in terms of generators from Section \ref{pi12}. 
The quotient by it does not preserve the distribution and gives holonomic $M^{2|2}$.

\smallskip

\underline{$\pi_3$:} The $\mathbb{P}^{1|1}$ bundle $\pi_2:M^{3|3}\to M^{2|2}$ has fiber consisting of
$(1|1)$-dimensional planes $\Pi\subset\mathcal{V}_1\subset\mathcal{T}M^{2|2}$ (again we choose one
component of $\mathcal{V}$, but $\mathcal{V}_2$ can be considered similarly, leading to an equivalent flag
variety $G/P_3^\I$ as the base). Such planes are parameterized by one even and one odd parameters,
whence the type of the fiber; note that for the Lie superalgebra 
$\mathfrak{f}=\mathfrak{sl}(2|1)\simeq\mathfrak{osp}(2|2)$ from Section \ref{Gpi2}
and its distinguished Borel subalgebra (numerated white-grey node) 
the parabolic $\fp_1^\I$ (in matrix realization \eqref{sl13=so22} $b_1=b_2=0$) is the stabilizer of a $(1|0)$ line, 
while the parabolic $\fp_2^\I$ (in matrix realization \eqref{sl13=so22} $b_2=a_5=0$) is the stabilizer of 
a $(1|1)$ plane in its $\C^{2|2}$ representation. This gives the construction of $M^{3|3}$ through $M^{2|2}$. 
 
Conversely, there exist two distinguished subdistributions of dimension three in $\hcD$ of $M^{3|3}$ 
(see the next subsection). One of them is $\langle Y_2,Y_3,yxy\rangle$ in the notations of Section \ref{pi23};
this contains the maximal isotropic sub-distribution $\langle Y_3,yxy\rangle$, the quotient by which gives $M^{2|2}$. 

\smallskip

\underline{$\pi_{23}$:} A description of the $\mathbb{P}^{0|1}$ fiber of the bundle 
$\pi_{23}:M^{3|4}\to M^{3|3}$ in terms of the variety $\mathcal{W}$ from Section \ref{Gpi23} is more involved 
and will be omitted. However the base is internally recovered from the distribution $\hcD$ on the total space
$M^{3|4}$ as follows: the derived distribution $\hcD^2=[\hcD,\hcD]$ has exactly one Cauchy characteristic, 
it is $yxx\in(\g_{-1})_{\bar1}$ in notations of Section \ref{pi123}. The quotient by it is precisely the flag variety $M^{3|3}$. (Note that the distribution $\hcD$ is naturally split into three line fields - two even and one odd;
the quotient by the first one gives the projection $\pi_{12}$,
the quotient by the second one gives an outer related, while the last quotient is $\pi_{23}$.)

\smallskip

\underline{$\pi_\diamond$:} Similarly omitting the construction of the total space of the 
$\mathbb{P}^{0|1}$ bundle $\pi_\diamond:M^{3|4}_\diamond\to M^{3|3}$ in terms of the base,
let us consider the reverse construction. In this case, the distribution $\hcD$ on $M^{3|4}_\diamond$
is split into three odd lines (they play similar roles). 
Quotient by one of them, for instance $xyy$ in notations of Section \ref{piv123},
gives precisely the flag variety $M^{3|3}$.

\subsection{Symmetries of $(M^{3|3},\cD^{2|2})$.}\label{S3322}

This is the remaining case to determine $\pr(\fm_{23}^\I)$ that coincides with 
$\mathfrak{sym}(M^{3|3},\cD^{2|2})$.
The geometric methods presented in Section \ref{TwSec} show useful here.
The distribution $\cD=\langle e_1,e_2,f_1,f_2\rangle$, with structure relations 
 $$
[e_1,f_1]=[e_2,f_2]=f_3,\quad [f_1,f_2]=2e_3,
 $$
can be realized in coordinates $(y_1,y_2,y_3|\xi_1,\xi_2,\xi_3)$ as 
 $$
e_1=\p_{y_1},\ e_2=\p_{y_2},\ e_3=\p_{y_3},\ 
f_1=\p_{\xi_1}+\xi_2\p_{y_3}+y_1\p_{\xi_3},\ f_2=\p_{\xi_2}+\xi_1\p_{y_3}+y_2\p_{\xi_3},\ f_3=\p_{\xi_3}.
 $$

 \begin{theorem}
The symmetry algebra of $(M^{3|3},\cD^{2|2})$ is parametrized by two functions of four arguments
$h=h(y_3,\xi_1,\xi_2,\xi_3)$ and $g=g(y_3,\xi_1,\xi_2,\xi_3)$. More precisely, it is given by 
 \begin{equation}\label{Xgh}
X=-\tfrac12(-1)^{|h|}\bigl(f_2(h)\cdot f_1+f_1(h)\cdot f_2\bigr)+h\cdot e_3+\tilde{g}\cdot f_3+m_1\cdot e_1+m_2\cdot e_2,
 \end{equation}
where $\tilde{g}=g+\tfrac12(-1)^{|h|}(f_2(h)\cdot y_1+f_1(h)\cdot y_2\bigr)$ and $m_i=-(-1)^{|g|}f_i(g)$.
The algebra is graded with weights $w(y_1)=w(y_2)=w(\xi_1)=w(\xi_2)=1$, $w(y_3)=w(\xi_3)=2$.
 \end{theorem}

This implies the structure of $\hat{\g}=\pr(\fm_{23}^\I)$, in particular dimensions of the graded components:
 \begin{table}[h!]\[
\begin{array}{c|c|c|c|c|c|c|c}
k & -2 & -1 & 0 & 1 & 2 & 3 & \dots \\ \hline 
\dim\hat{\g}_k & (1|1) & (2|2) & (3|3) & (4|4) & (4|4) & (4|4) & \dots
\end{array}\]
 \end{table}

 \begin{proof}
Consider the Lie superalgebra structure of $\fm_{23}^\I$.
For a vector $v=a_1e_1+a_2e_2+b_1f_1+b_2f_2\in\g_{-1}$ the condition $[v,v]=0$ writes $b_1b_2=0$, 
so $\g_{-1}$ has two distinguished 3-dimensional subspaces $\langle e_1,e_2|f_1\rangle$
and $\langle e_1,e_2|f_2\rangle$ that intersect by the even 2-plane $\langle e_1,e_2\rangle$.
Also, maximal isotropic (not necessarily graded) 2-planes form two pencils 
$\langle e_1,e_2+rf_2\rangle$, $\langle e_2,e_1+rf_1\rangle$, where $r\in\mathbb{P}^1$. 
Note that those pencils again intersect by $\langle e_1,e_2\rangle$, so it is distinguished.

Let us quotient the distribution $\cD$ by the subdistribution $\Delta=\langle e_1,e_2\rangle$. Since $e_i$ are not 
Cauchy characteristics, a larger distribution $[\cD,\Delta]$ descends to the quotient space $\bar{M}^{1|3}$.
Namely the quotient distribution $\bar{\cD}$ in coordinates ($y_3|\xi_1,\xi_2,\xi_3)$ on $\bar{M}$ is generated by 
 $$
\bar{f}_1=\p_{\xi_1}+\xi_2\p_{y_3},\ \bar{f}_2=\p_{\xi_2}+\xi_1\p_{y_3},\ \bar{f}_3=\p_{\xi_3}
 $$
with the only non-trivial bracket $[\bar{f_1},\bar{f}_2]=2\bar{e}_3$, $\bar{e}_3=\p_{y_3}$. 
Thus the distribution is degenerate contact with one Cauchy characteristic, so its symmetry is
computed by the standard formula
 $$
\bar{X}= -\tfrac12(-1)^{|\bar{h}|}\bigl(\bar{f}_2(\bar{h})\cdot\bar{f}_1+\bar{f}_1(\bar{h})\cdot\bar{f}_2\bigr)
+\bar{h}\cdot\bar{e}_3+\bar{g}\cdot\bar{f}_3.
 $$
Here we use that $\op{Ann}(\cD)$ is generated by $\alpha_1=dy_3+\xi_1d\xi_2+\xi_2d\xi_1$,
$\alpha_2=d\xi_3-y_1d\xi_1-y_2d\xi_2$ and only $\alpha_1$ descends to the quotient, where it defines the
degenerate contact structure $\bar{\cD}$.
 
The algebra of symmetries acts on the space of odd 2-planes transversal to $\bar{f}_3$ in $\bar{\cD}$, 
which can be identified with $M^{3|3}$ via coordinization 
$\bar{\cD}=\langle f_1=\bar{f}_1+y_1\bar{f}_3,f_2=\bar{f}_2+y_2\bar{f}_3\rangle$. 
In formulas, keeping the same coordinates and denoting $\tilde{X}$ the field $\bar{X}$
with $\bar{f}_i$ changed to $f_i$ and $\bar{h},\bar{g}$ to $\tilde{h},\tilde{g}$, the lift is given by 
 $$
X=\tilde{X}+m_1\cdot e_1+m_2\cdot m_2,
 $$
which by the condition that $\langle\alpha_1,\alpha_2\rangle$ is preserved, allows to uniquely determine $m_i$.
 \end{proof}
 
To get generators of $\hat{\g}_k$ assume $h,g$ polynomial (or analytic) and decompose 
$h(y_3,\xi_1,\xi_2,\xi_3)=\sum h_{ijpq}\xi_1^i\xi_2^j\xi_3^py_3^q$ and note that 
$\xi_1^i\xi_2^j\xi_3^py_3^q\p_{y_3}$ has weight $i+j+(p+q-1)$ and parity $i+j+p$.
Similarly, $g(y_3,\xi_1,\xi_2,\xi_3)=\sum g_{ijpq}\xi_1^i\xi_2^j\xi_3^py_3^q$ and note that 
$\xi_1^i\xi_2^j\xi_3^py_3^q\p_{\xi_3}$ has weight $i+j+2(p+q-1)$ and parity $i+j+p-1$.
This implies the above table of dimensions. 

The Lie superalgebra of symmetries is generated by a contact vector field $X_f$ in an odd contact space $\C^{1|2}$ 
with the generating function $f$ depending on variables in the space $\C^{1|3}=\C^{1|2}\oplus\C^{0|1}$, 
and also by an Abelian algebra of shifts along the second factor, the odd line $\C^{0|1}$ with coefficients
being arbitrary functions on $\C^{1|3}$. This finishes the justification of Table \ref{SymmA}.

\subsection{On $D(2,1;a)$ non-flat geometries.}

The geometries considered in this paper have curved analogues. While the general approach to their construction,
following the steps of parabolic geometries \cite{CS}, is complicated and is outside the scope of this work,
let us indicate the deformation approach to their construction.

For instance, for the flat geometry of type $(G,P_{123}^\IV)$, considered in Section \ref{Gpiv123},
the deformation is obtained by perturbing the Legendrian distribution $\cD\subset\cC$ while keeping 
the mixed contact structure $\cC$ untouched. This is equivalent to changing the right-hand-side of 
the odd PDE system \eqref{nPDE} by generic odd functions on the space $\C^{3|4}$. 
In general, this has smaller symmetry.
 
Similarly, perturbing the coefficients of the odd bi-quadric $\mathcal{V}$ in \eqref{Vsupervar}, we obtain 
a geometry of type $(G,P_2^\I)$ that is not necessarily flat, and has smaller symmetry.
It would be interesting to further understand such geometries and to establish
the submaximal symmetry dimension, i.e.\ study the supersymmetric dimension gap phenomenon.

\appendix

\section{Irreducible representations of $\fsl(2|1)$}\label{appendixsl21}

The Lie superalgebra $\fsl(2|1)$ has four even generators $X,H,Y,I$, four odd generators 
$F^\pm$, $\overline{F}^\pm$ and (the only nontrivial) relations:
 \begin{align*}
[H,X]&=2X,  \quad& [H,F^\pm]&=\pm F^\pm\!\!, \quad& [H,\overline{F}^\pm]&=\pm\overline{F}^\pm\!\!,\\
[H,Y]&=-2Y, \quad& [I,F^\pm]&=F^\pm\!\!, \quad& [I,\overline{F}^\pm]&=-\overline{F}^\pm\!\!,\\
[X,Y]&=H,    \quad& [X,F^-]&=-F^+\!\!, \quad& [X,\overline{F}^-]&=\overline{F}^+\!\!,\\
[F^+,\overline{F}^-]&=\tfrac12\left(I-H\right), \quad& [Y,F^+]&=-F^-\!\!, \quad& 
[Y,\overline{F}^+]&=\overline{F}^-\!\!,\\
[F^-,\overline{F}^+]&=\tfrac12\left(I+H\right), \quad& [F^+,\overline{F}^+]&=X, \quad&
[F^-,\overline{F}^-]&=Y.
 \end{align*}

The irreducible representations of $\fsl(2|1)$ are characterized by a pair of labels $(b,m)$ where $b\in\mathbb{C}$ and $m\in\mathbb{N}\cup\{0\}$, cf.\ \cite{Frappat}. 
The representations $\rho(b,m)$ with $b\neq\pm m$ are typical and have dimension $4m$. 
The representations $\rho(\pm m,m)$ are atypical and have dimension $2m+1$. 

The $4m$-dimensional typical representation $\rho(b,m)$ of $\fsl(2|1)$ decomposes under the even part 
$\CC\oplus\fsl_2$ as
 \begin{equation}
\rho(b,m)=\pi_m(b)\oplus\pi_{m-1}(b+1)\oplus\pi_{m-1}(b-1)\oplus\pi_{m-2}(b),
 \end{equation}
where $m$ is the highest $H$-weight, and $b$ is the $I$-weight of the highest $H$-weight vector.

The case $m=1$ is special with the $4$-dimensional representation
 \begin{equation}
\rho(b,1)=\pi_1(b)\oplus\pi_0(b+1)\oplus\pi_0(b-1).
 \end{equation}

 \begin{example}
For the adjoint representation
$\fsl(2|1)=\rho(0,2)=\pi_2(0)\oplus\pi_1(1)\oplus\pi_1(-1)\oplus\pi_0(0)$: 
 \begin{itemize}
\item $\pi_2(0)$ is generated by $X,H,Y$; 
\item $\pi_1(1)$ is generated by $F^+,F^-$;
\item $\pi_1(-1)$ is generated by $\overline{F}^+,\overline{F}^-$;
\item $\pi_0(0)$ is generated by $I$.
 \end{itemize}
 \end{example}

\section{Details on computations of prolongations}\label{B0}

\subsection{Proof of Proposition \ref{propPi23prol} for $\fp_{23}^{\rm I}$}\label{B1}

Let $\{e_1=Y_2,e_2=Y_3,e_3=yxy, e_4=yyx\}$ be a basis of $g_{-1}=\CC^{2|2}$, 
and $\{e_5=Y_1, e_6=yyy\}$ be a basis of $\fg_{-2}$. Let $\omega^i$ be the dual of $e_i$ for $i=1,\dots,6$. 

The non-zero bracket relations in $\fm$ are (see Section \ref{pi23}):
\begin{equation*}
 [e_1,e_3]=e_6,\quad [e_2,e_4]=e_6,\quad [e_3,e_4]=s_1 e_5.
\end{equation*}

\smallskip

First, consider $\pr_1$.
We have $\fg_0\otimes\fg_{-1}^\ast=\CC^{3|2}\otimes\CC^{2|2}=\CC^{10|10}$. If 
$\varphi=(a_1H_1+a_2H_2+a_3H_3)\otimes\omega^1+(b_1H_1+b_2H_2+b_3H_3)\otimes\omega^2+(c_1xyy+c_2yxx)\otimes\omega^3+(d_1xyy+d_2yxx)\otimes\omega^4$ is an even vector in $\pr_1(\fg_{\leq0})\subset\fg_0\otimes\fg_{-1}^\ast$, then:
 \begin{eqnarray*}
0&=&\varphi[e_1,e_2]=[\varphi e_1,e_2]+[e_1,\varphi e_2]=[a_1H_1+a_2H_2+a_3H_3,Y_3]+[Y_2,b_1H_1+b_2H_2+b_3H_3]\\
&=&-2a_3Y_3+2b_2Y_2=0\Rightarrow \boxed{a_3=0},\,\boxed{b_2=0},\\
0&=&\varphi[e_2,e_3]=[\varphi e_2,e_3]+[e_2,\varphi e_3]=[b_1H_1+b_3H_3,yxy]+[Y_3,c_1xyy+c_2yxx]\\
&=& -(b_1+b_3)yxy+c_2yxy=-2b_1yxy\Rightarrow \boxed{b_1=0},\\
0&=&\varphi[e_1,e_4]=[\varphi e_1,e_4]+[e_1,\varphi e_4]=[a_1H_1+a_2H_2,yyx]+[Y_2,d_1xyy+d_2yxx]\\
&=& -(a_1+a_2)yyx+d_2yyx=-2a_1yyx\Rightarrow \boxed{a_1=0},\\
0&=&\varphi[e_3,e_3]=2[\varphi e_3,e_3]=2[c_1xyy+c_2yxx,e_3]=2c_1s_3Y_3\Rightarrow\boxed{c_1=0},\\
0&=&\varphi[e_4,e_4]=2[\varphi e_4,e_4]=2[d_1xyy+d_2yxx,e_4]=2d_1s_2Y_2\Rightarrow\boxed{d_1=0}.
 \end{eqnarray*}
We compute
 \begin{eqnarray*}
\varphi e_6 &=&\varphi[e_1,e_3]=[\varphi e_1,e_3]+[e_1,\varphi e_3]=[a_2H_2,yxy]+[Y_2,c_2yxx]=
a_2yxy+c_2yyx,\\
\varphi e_6&=&\varphi[e_2,e_4]=[\varphi e_2,e_4]+[e_2,\varphi e_4]=[b_3H_3,yyx]+[Y_3,d_2yxx]=
b_3yyx+d_2yxy,\\
\varphi e_5&=&s_1^{-1}\varphi[e_3,e_4]=s_1^{-1}([\varphi e_3,e_4]+[e_3,\varphi e_4])=0,
 \end{eqnarray*}
whence $\boxed{a_2=d_2}$ and $\boxed{b_3=c_2}$. Then we verify
 \[
\varphi[e_1,e_5]=\varphi[e_1,e_6]=\varphi[e_2,e_5]=\varphi[e_2,e_6]=\varphi[e_3,e_5]=\varphi[e_3,e_6]
=\varphi[e_4,e_5]=0.
 \]
We conclude that $(\pr_1(\fg_{\leq0}))_{\bar0}=\langle H_2\otimes\omega^1+ yxx\otimes\omega^4, H_3\otimes\omega^2+yxx\otimes\omega^3\rangle$.
 
For the odd part, consider $\varphi=(a_1H_1+a_2H_2+a_3H_3)\otimes\omega^3 +(b_1H_1+b_2H_2+b_3H_3)\otimes\omega^4+(c_1xyy+c_2yxx)\otimes\omega^1+(d_1xyy+d_2yxx)\otimes\omega^2\in 
\pr_1(\fg_{\leq0})\subset\fg_0\otimes\fg_{-1}^\ast$, then:
  \begin{eqnarray*}
0&=&\varphi[e_1,e_2]=[\varphi e_1,e_2]+[e_1,\varphi e_2]=[c_1xyy+c_2yxx,Y_3]+[Y_2,d_1xyy+d_2yxx]\\
  &=&-c_2yxy+d_2yyx\Rightarrow\boxed{c_2=0},\,\boxed{d_2=0},\\
0&=&\varphi[e_2,e_5]=[\varphi e_2,e_5]+[e_2,\varphi e_5] =[d_1xyy,Y_1]\\
  &&{}+s_1^{-1}[Y_3,(a_3-a_1-a_2)yyx-(b_1-b_2+b_3)yxy]=-s_1^{-1}(a_1+a_2-a_3+s_1d_1)yyy,\\
0&=&\varphi[e_3,e_6]=[\varphi e_3,e_6]-[e_3,\varphi e_6]=[a_1H_1+a_2H_2+a_3H_3,yyy]
  -[yxy,d_1s_2  Y_2+2b_3Y_3]\\
  &=& -(a_1+a_2+a_3-s_2d_1)yyy,\\
0&=&\varphi[e_3,e_3]=2[\varphi e_3,e_3]=2[a_1H_1+a_2H_2+a_3H_3,yxy]=-2(a_1-a_2+a_3),\\
0&=&\varphi[e_1,e_5]=[\varphi e_1,e_5]+[e_1,\varphi e_5]=[c_1xyy,Y_1]\\
  &&+s_1^{-1}[Y_2,(-a_1-a_2+a_3)yyx-(b_1-b_2+b_3)yxy]= -s_1^{-1}(b_1-b_2+b_3+s_1c_1)yyy,\\
0&=&\varphi[e_4,e_6]=[\varphi e_4,e_6]-[e_4,\varphi e_6]=[b_1H_1+b_2H_2+b_3H_3,yyy]
  -[yyx,c_1s_3 Y_3+2a_2Y_2]\\
  &=& -(b_1+b_2+b_3-s_3c_1)yyy,\\
0&=&\varphi[e_4,e_4]=2[\varphi e_4,e_4]=2[b_1H_1+b_2H_2+b_3H_3,yyx]=-2(b_1+b_2-b_3),
\end{eqnarray*}
thus $\boxed{a_1=-\tfrac12s_1d_1},\,\boxed{a_2=\tfrac12s_2d_1},\,\boxed{a_3=-\tfrac12s_3d_1}$ and
$\boxed{b_1=-\tfrac12s_1c_1},\,\boxed{b_2=-\tfrac12s_2c_1},\,\boxed{b_3=\tfrac12s_3c_1}$. Further,
  \begin{eqnarray*}
\varphi e_6 &=&\varphi[e_1,e_3]
 =s_3c_1 Y_3+2a_2 Y_2,\\
\varphi e_6&=&\varphi[e_2,e_4]
 =s_2d_1 Y_2+2b_3Y_3,\\
\varphi e_5&=& s_1^{-1}\varphi[e_3,e_4]
 =d_1yyx+c_1yxy,
 \end{eqnarray*}
and so we get
 \[
\varphi[e_2,e_3]= \varphi[e_1,e_4]= \varphi[e_1,e_6]= \varphi[e_2,e_6]= \varphi[e_5,e_6]=0.
 \]
We conclude that $(\pr_1(\fg_{\leq0})_{\bar1}=\langle(-s_1H_1+s_2H_2-s_3H_3)\otimes\omega^3+2 xyy\otimes\omega^2,(-s_1H_1-s_2H_2+s_3H_3)\otimes\omega^4+2 xyy\otimes\omega^1\rangle$
and consequently $\pr_1(\fg_{\leq0})=\fg_1$.

\medskip

Next, consider $\pr_2$.
We have $\fg_1\otimes\fg_{-1}^\ast=\CC^{2|2}\otimes\CC^{2|2}=\CC^{8|8}$.   
If $\varphi=(a_1f_1+a_2f_2)\otimes\omega^1+(b_1f_1+b_2f_2)\otimes\omega^2+(c_1f_3+c_2f_4)\otimes\omega^3+(d_1f_3+d_2f_4)\otimes\omega^4$ is an even vector in 
$\pr_2(\fg_{\leq0})\subset\fg_1\otimes\fg_{-1}^\ast$, then:
 \begin{eqnarray*}
0&=&\varphi[e_1,e_2]=[\varphi e_1,e_2]+[e_1,\varphi e_2]=[a_1f_1+a_2f_2,e_2]+[e_1,b_1f_1+b_2f_2]\\
  &=& a_2H_3-b_1H_2 \Rightarrow\boxed{a_2=0},\,\boxed{b_1=0},\\
0 &=&\varphi[e_2,e_3]=[\varphi e_2,e_3]+[e_2,\varphi e_3]=[b_2f_2,e_3]+[e_2,c_1f_3+c_2f_4]\\
  &=&b_2yxx+c_1xyy\Rightarrow\boxed{b_2=0},\,\boxed{c_1=0},\\
0&=&\varphi[e_1,e_4]=[\varphi e_1,e_4]+[e_1,\varphi e_4]=[a_1f_1,yyx]+[Y_2,d_1f_3+d_2f_4]\\
  &=&a_1yxx+d_2xyy\Rightarrow\boxed{a_1=0},\,\boxed{d_2=0}.
 \end{eqnarray*}
Then we compute:
 \begin{eqnarray*}
\varphi e_6 &=&\varphi[e_1,e_3]=[\varphi e_1,e_3]+[e_1,\varphi e_3]=[e_1,c_2f_4]=c_2xyy,\\
\varphi e_6 &=&\varphi[e_2,e_4]=[\varphi e_2,e_4]+[e_2,\varphi e_4]=[Y_3,d_1f_3]=d_1xyy    
   \Rightarrow\boxed{d_1=c_2},\\
\varphi e_5 &=&s_1^{-1}\varphi[e_3,e_4]=s_1^{-1}([\varphi e_3,e_4]+[e_3,\varphi e_4])
  =\frac{c_2}{s_1}[xxy,yyx]+\frac{d_1}{s_1}[yxy,xyx]\\
   &=& \frac{c_2}{2s_1}\left(s_1H_1+s_2H_2-s_3H_3\right)
   +\frac{d_1}{2s_1}\left(s_1H_1-s_2H_2+s_3H_3\right)=c_2H_1,\\
0 &=&\varphi[e_1,e_5]=[\varphi e_1,e_5]+[e_1,\varphi e_5]
   =\tfrac12[Y_2,c_2H_1]=0,\\
0 &=&\varphi[e_1,e_6]=[\varphi e_1,e_6]+[e_1,\varphi e_6]=[Y_2,c_2xyy]=0.
 \end{eqnarray*}
We conclude that $(\pr_2(\fg_{\leq0}))_{\bar0}=\langle f_4\otimes\omega^3+f_3\otimes\omega^4\rangle$.

For the odd part, consider $\varphi=(a_1f_3+a_2f_4)\otimes\omega^1+(b_1f_3+b_2f_4)\otimes\omega^2
+(c_1f_1+c_2f_2)\otimes\omega^3+(d_1f_1+d_2f_2)\otimes\omega^4\in\pr_2(\fg_{\leq0})\subset
\fg_1\otimes\fg_{-1}^\ast$, then:
 \begin{eqnarray*}
0 &=&\varphi[e_2,e_3]=[\varphi e_2,e_3]+[e_2,\varphi e_3]=[b_1f_3+b_2f_4,e_3]+[e_2,c_1f_1+c_2f_2]\\
  &=&b_1\left(\frac{1}{s_1}H_1-\frac{1}{s_2}H_2+\frac{1}{s_3}H_3\right)-c_2H_3
  \Rightarrow\boxed{b_1=0},\,\boxed{c_2=0},\\
0&=&\varphi[e_1,e_4]=[\varphi e_1,e_4]+[e_1,\varphi e_4]=[a_1f_3+a_2f_4,yyx]+[Y_2,d_1f_1+d_2f_2]\\
  &=&a_2\left(\frac{1}{s_1}H_1+\frac{1}{s_2}H_2-\frac{1}{s_3}H_3\right)-d_1H_2
 \Rightarrow\boxed{a_2=0},\,\boxed{d_1=0},\\
0&=&\varphi[e_1,e_2]=[\varphi e_1,e_2]+[e_1,\varphi e_2]=[a_1f_3,e_2]+[e_1,b_2f_4]=(b_2-a_1)xyy \Rightarrow\boxed{b_2=a_1}.
  \end{eqnarray*}
Furthermore,
 \begin{eqnarray*}
\varphi e_6 &=&\varphi[e_1,e_3]=[a_1f_3,e_3]+[e_1,c_1f_1]
=a_1\left(\frac{1}{s_1}H_1-\frac{1}{s_2}H_2+\frac{1}{s_3}H_3\right)-c_1H_2,\\
\varphi e_6&=&\varphi[e_2,e_4]=[b_2f_4,yyx]+[Y_3,d_2f_2]
=b_2\left(\frac{1}{s_1}H_1+\frac{1}{s_2}H_2-\frac{1}{s_3}H_3\right)-d_2H_3.
 \end{eqnarray*}
Therefore $\boxed{s_2c_1=-2a_1}$, $\boxed{s_3d_2=-2a_1}$. We conclude that 
$(\pr_2(\fg_{\leq0}))_{\bar1}=\langle f_3\otimes\omega^1+f_4\otimes\omega^2
-\frac{2}{s_2}f_1\otimes\omega^3-\frac{2}{s_3}f_2\otimes\omega^4\rangle$ and consequently
$\pr_2(\fg_{\leq0})=\fg_2$.

\medskip

Finally, consider $\pr_3$.
We have $\fg_2\otimes\fg_{-1}^\ast=\CC^{1|1}\otimes\CC^{2|2}=\CC^{4|4}$.  
If $\varphi=af_5\otimes\omega^1+b f_5\otimes\omega^2+c f_6\otimes\omega^3+d f_6\otimes\omega^4$ 
is an even vector in $\pr_3(\fg_{\leq0})\subset\fg_2\otimes\fg_{-1}^\ast$, then:
 \begin{eqnarray*}
\varphi e_6 &=&\varphi[e_1,e_3]=[\varphi e_1,e_3]+[e_1,\varphi e_3]=[af_5,e_3]+[e_1,cf_6]=axxy+cxyx,\\
\varphi e_6&=&\varphi[e_2,e_4]=[\varphi e_2,e_4]+[e_2,\varphi e_4]=[bf_5,e_4]+[e_2,df_6]=bxyx+dxxy.\\
 \end{eqnarray*}
Hence $\boxed{a=d}$, $\boxed{b=c}$. Furthermore,
 \begin{eqnarray*}
0 &=&\varphi[e_2,e_3]=[\varphi e_2,e_3]+[e_2,\varphi e_3]=[bf_5,e_3]+[e_2,cf_6]=bxxy+cxxy
   \Rightarrow\boxed{b=-c},\\
0&=&\varphi[e_1,e_4]=[\varphi e_1,e_4]+[e_1,\varphi e_4]=[af_5,e_4]+[Y_2,df_6]=axyx+bxxy
   \Rightarrow\boxed{a=-b}.
  \end{eqnarray*}
Therefore $\pr_3(\fg_{\leq0})_{\bar0}=0$. For the odd part, consider 
$\varphi=a f_6\otimes\omega^1+b f_6\otimes\omega^2+c f_5\otimes\omega^3+d f_5\otimes\omega^4
\in\pr_3(\fg_{\leq0})\subset\fg_2\otimes\fg_{-1}^\ast$, then:
 \begin{eqnarray*}
\varphi e_6 &=&\varphi[e_1,e_3]=[\varphi e_1,e_3]+[e_1,\varphi e_3]=[af_6,e_3]+[e_1,cf_5]=as_2f_1,\\
\varphi e_6&=&\varphi[e_2,e_4]=[\varphi e_2,e_4]+[e_2,\varphi e_4]=[bf_6,e_4]+[Y_3,df_5]=bs_3f_2.\\
 \end{eqnarray*}
Hence $\boxed{a=0},\,\boxed{b=0}$. Furthermore,
 \begin{eqnarray*}
  s_1\varphi e_5&=&\varphi[e_3,e_4]=[\varphi e_3,e_4]-[e_3,\varphi e_4]=[cf_5,e_4]-[e_3,df_5]=cf_3-df_4\\
 0 &=&\varphi[e_3,e_5]=[\varphi e_3,e_5]-[e_3,\varphi e_5]=[cf_5,e_5]-[e_3,\frac{1}{s_1}(cf_3-df_4)]=\\
 &=&cH_1-\frac{c}{s_1}(\frac{s_1}{2}H_1-\frac{s_2}{2}H_2+\frac{s_3}{2}H_3)\Rightarrow\boxed{c=0}\\
 0&=&\varphi[e_4,e_5]=[\varphi e_4,e_5]-[e_4,\varphi e_4]=[df_5,e_5]+[e_4,\frac{1}{s_1}(cf_3-df_4)]=\\
 &=&dH_1-\frac{c}{s_1}(\frac{s_1}{2}H_1+\frac{s_2}{2}H_2-\frac{s_3}{2}H_3)\Rightarrow\boxed{d=0}.
  \end{eqnarray*}
Therefore $\pr_3(\fg_{\leq0})_{\bar1}=0$ and we have finished the proof. \hfill\qed

\subsection{Proof of Proposition \ref{propPi123prol} for $\fp_{123}^\I$}\label{B2}

Let $\{e_1=Y_2,e_2=Y_3,e_3=yxx\}$ be a basis of $\fg_{-1}$, $\{e_4=yyx,e_5=yxy\}$ be a basis of $\fg_{-2}$, 
$\{e_6=yyy\}$ be a basis of $\fg_{-3}$ and $\{e_7=Y_1\}$ be a basis of $\fg_{-4}$. 
Let $\omega^i$ be the dual of $e_i$ for $i=1,\dots,7$. 
The non-zero bracket relations in $\fm$ are (see Section \ref{pi123}):
 \begin{equation*}
[e_1,e_3]=e_4,\quad [e_2,e_3]=e_5,\quad [e_1,e_5]=e_6,\quad [e_2,e_4]=e_6,\quad [e_4,e_5]=s_1 e_7,\quad [e_3,e_6]=-s_1 e_7.
 \end{equation*}

\smallskip

First, consider $\pr_2$.
We have $\fg_1\otimes\fg_{-1}^\ast=\CC^{2|1}\otimes\CC^{2|1}=\CC^{5|4}$.
Let $\{f_1=X_2,f_2=X_3,f_3=xyy\}$ be a basis of $\fg_{1}$. 
If $\varphi=(a_1f_1+a_2f_2)\otimes\omega^1+(b_1f_1+b_2f_2)\otimes\omega^2+c f_3\otimes\omega^3$ 
is an even vector in $\pr_2(\fg_{\leq1})\subset\fg_1\otimes\fg_{-1}^\ast$, then:
 \begin{eqnarray*}
0&=&\varphi[e_1,e_2]=[\varphi e_1,e_2]+[e_1,\varphi e_2]=[a_1f_1+a_2f_2,e_2]+[e_1,b_1f_1+b_2f_2]=\\
   &=&a_2H_3+b_1H_2\Rightarrow\boxed{a_2=0},\,\boxed{b_1=0},\\
\varphi e_4&=&\varphi[e_1,e_3]=[\varphi e_1,e_3]+[e_1,\varphi e_3]=[a_1f_1,e_3]+[e_1,cf_3]=0,\\
0&=&\varphi[e_1,e_4]=[\varphi e_1,e_4]+[e_1,\varphi e_4]=[a_1f_1,e_4]=a_1e_3\Rightarrow\boxed{a_1=0},\\
\varphi e_5&=&\varphi[e_2,e_3]=[\varphi e_2,e_3]+[e_2,\varphi e_3]=[b_2f_2,e_3]+[e_2,cf_3]=0,\\
0&=&\varphi[e_2,e_5]=[\varphi e_2,e_5]+[e_2,\varphi e_5]=[b_2f_2,e_5]=b_2e_3\Rightarrow\boxed{b_2=0},\\
0&=&\varphi[e_3,e_4]=[\varphi e_3,e_4]+[e_3,\varphi e_4]=[cf_3,e_4]=cs_2e_1\Rightarrow\boxed{c=0}.
 \end{eqnarray*}
Thus $(\fg_2)_{\bar0}=0$.
For the odd part, consider $\varphi=a f_3\otimes\omega^1+b f_3\otimes\omega^2+ (cf_1+df_2)\otimes\omega^3\in\pr_2(\fg_{\leq1})\subset\fg_1\otimes\fg_{-1}^\ast$, then:
 \begin{eqnarray*}
\varphi e_4&=&\varphi[e_1,e_3]=[af_3,e_3]+[e_1,cf_1+df_2]
   =a\left(-\frac{s_1}{2}H_1+\frac{s_2}{2}H_2+\frac{s_3}{2}H_3\right)-cH_2,\\
0&=&\varphi[e_1,e_4]=[af_3,e_4]+[e_1,a\left(-\frac{s_1}{2}H_1+\frac{s_2}{2}H_2+\frac{s_3}{2}H_3\right)
   -cH_2]=2(as_2-c)e_1,\\
\varphi e_5&=&\varphi[e_2,e_3]=[bf_3,e_3]+[e_2,cf_1+df_2]
   =b\left(-\frac{s_1}{2}H_1+\frac{s_2}{2}H_2+\frac{s_3}{2}H_3\right)-bH_3,\\
0&=&\varphi[e_2,e_5]=[bf_3,e_5]+[e_1,b\left(-\frac{s_1}{2}H_1+\frac{s_2}{2}H_2+\frac{s_3}{2}H_3\right)
   -dH_3]=2(bs_3-d)e_2.
 \end{eqnarray*}
Hence $\boxed{c=s_2a}$, $\boxed{d=s_3b}$. Since $(\fg_2)_{\bar1}\subset\pr_2(\fg_{\leq1})$, we conclude
$\pr_2(\fg_{\leq1})=\CC^{0|2}=\fg_2$.

\medskip

Next, consider $\pr_3$.
We have $\fg_2\otimes\fg_{-1}^\ast=\CC^{0|2}\otimes\CC^{2|1}=\CC^{2|4}$.
Let $\{f_4=xxy,f_5=xyx\}$ be a basis of $g_2$. 
If $\varphi=(af_4+bf_5)\otimes \omega^3$ is an even vector in 
$\pr_3(\fg_{\leq1})\subset\fg_2\otimes\fg_{-1}^\ast$, then:
 \begin{eqnarray*}
\varphi e_4&=&\varphi[e_1,e_3]=[\varphi e_1,e_3]+[e_1,\varphi e_3]=[e_1,af_4+bf_5]=af_3,\\
0&=&\varphi[e_3,e_4]=[\varphi e_3,e_4]+[e_3,\varphi e_4]=[af_4+bf_5,e_4]+[e_3,af_3]=\\
   &=&a\left(\frac{s_1}{2}H_1+\frac{s_2}{2}H_2-\frac{s_3}{2}H_3\right)+a\left(-\frac{s_1}{2}H_1
   +\frac{s_2}{2}H_2+\frac{s_3}{2}H_3\right)=as_2 H_2\Rightarrow\boxed{a=0},\\
\varphi e_5&=&\varphi[e_2,e_3]=[\varphi e_2,e_3]+[e_2,\varphi e_3]=[e_2,af_4+bf_5]=bf_3,\\
0&=&\varphi[e_3,e_5]=[\varphi e_3,e_5]+[e_3,\varphi e_5]=[af_4+bf_5,e_5]+[e_3,bf_3]=\\
   &=&b\left(\frac{s_1}{2}H_1-\frac{s_2}{2}H_2+\frac{s_3}{2}H_3\right)+b\left(-\frac{s_1}{2}H_1
   +\frac{s_2}{2}H_2+\frac{s_3}{2}H_3\right)=bs_3 H_3\Rightarrow\boxed{b=0}.
\end{eqnarray*}
Therefore $\pr_3(\fg_{\leq1})_{\bar0}=0$. For the odd part, consider 
$\varphi=(af_4+bf_5)\otimes\omega^1+ (cf_4+df_5)\otimes\omega^2\in
\pr_3(\fg_{\leq1})\subset\fg_2\otimes\fg_{-1}^\ast$, then:
 \begin{eqnarray*}
0&=&\varphi[e_1,e_2]=[\varphi e_1,e_2]+[e_1,\varphi e_2]=[af_4+bf_5,e_3]+[e_1,cf_4+df_5]=-(b+c)f_3,\\
\varphi e_4&=&\varphi[e_1,e_3]=[\varphi e_1,e_3]+[e_1,\varphi e_3]=[af_4+bf_5,e_3]=-as_2 f_1-bs_3 f_2,\\
0&=&\varphi[e_1,e_4]=[af_4+bf_5,e_4]+[e_1,-as_2 f_1-bs_3 f_2]=\tfrac12a(s_1H_1+3s_2H_2-s_3H_3),\\
\varphi e_5&=&\varphi[e_2,e_3]=[\varphi e_2,e_3]+[e_2,\varphi e_3]=[cf_4+df_5,e_3]=-cs_2 f_1-ds_3 f_2,\\
0&=&\varphi[e_2,e_5]=[cf_4+df_5,e_5]+[e_2,-cs_2 f_1-ds_3 f_2]=\tfrac12d(s_1H_1-s_2H_2+3s_3H_3).
\end{eqnarray*}
Hence $\boxed{a=0}$, $\boxed{d=0}$ and $\boxed{b=-c}$. We conclude 
$\pr_3(\fg_{\leq1})=\CC^{0|1}=\fg_3$.

\medskip

Next, consider $\pr_4$. We have $\fg_3\otimes\fg_{-1}^\ast=\CC^{0|1}\otimes\CC^{2|1}\CC^{1|2}$.
Let $\{f_6=xxx\}$ be a basis of $\fg_3$. 
If $\varphi=af_6\otimes\omega^1+b f_6\otimes\omega^2$ is an even vector in 
$\pr_4(\fg_{\leq1})\subset\fg_3\otimes\fg_{-1}^\ast$, then:
 \begin{eqnarray*}
0&=&\varphi[e_1,e_2]=[\varphi e_1,e_2]+[e_1,\varphi e_2]=[af_6,e_2]+[e_1,bf_6]=-af_4+bf_5,
 \end{eqnarray*}
therefore $\boxed{a=0}$, $\boxed{b=0}$. 
For the odd part, only $\varphi=cf_6\otimes\omega^3$ is the candidate. However, since 
$\fg_4\subset\pr_4(\fg_{\leq1})$, $c$ is un-restricted, and we conclude $\pr_4(\fg_{\leq1})=\CC^{1|0}=\fg_4$.

\medskip

Finally, consider $\pr_5$. Let $\{f_7=X_1\}$ be a basis of $\fg_4$. 
We have $\fg_4\otimes\fg_{-1}^\ast=\CC^{1|0}\otimes\CC^{2|1}=\CC^{2|1}$.
If $\varphi=af_7\otimes\omega^1+b f_7\otimes\omega^2$ is an even vector in
$\pr_5(\fg_{\leq1})\subset\fg_4\otimes\fg_{-1}^\ast$, then:
 \begin{eqnarray*}
\varphi e_4&=&\varphi[e_1,e_3]=[\varphi e_1,e_3]+[e_1,\varphi e_3]=[af_7,e_3]=af_6,\\
0&=&\varphi[e_1,e_4]=[\varphi e_1,e_4]+[e_1,\varphi e_4]=[af_7,e_4]+[e_1,af_6]=2af_5,\\
\varphi e_5&=&\varphi[e_2,e_3]=[\varphi e_2,e_3]+[e_2,\varphi e_3]=[bf_7,e_3]=bf_6,\\
 0&=&\varphi[e_2,e_5]=[\varphi e_2,e_5]+[e_2,\varphi e_5]=[bf_7,e_5]+[e_2,bf_6]=2bf_4.
 \end{eqnarray*}
Hence $\boxed{a=0}$, $\boxed{b=0}$, and $\pr_5(\fg_{\leq1})_{\bar0}=0$. For the odd part, 
consider $\varphi= c f_7\otimes\omega^3\in\pr_5(\fg_{\leq1})\subset\fg_4\otimes\fg_{-1}^\ast$, then:
 \begin{eqnarray*}
\varphi e_4&=&\varphi[e_1,e_3]=[\varphi e_1,e_3]+[e_1,\varphi e_3]=[e_1,cf_7]=0,\\
0&=&\varphi[e_3,e_4]=[\varphi e_3,e_4]+[e_3,\varphi e_4]=[cf_7,e_4]=ce_5\neq0.
 \end{eqnarray*}
Hence $\boxed{c=0}$, and we conclude $\pr_5(\fg_{\leq1})_{\bar1}=0$. \hfill\qed

\subsection{Proof of Proposition \ref{propPiv123prol} for $\fp_{123}^\IV$}\label{B3}

Let $\{e_1=xyy,e_2=yxy,e_3=yyx\}$ be a basis of $\fg_{-1}$, $\{e_4=Y_1,e_5=Y_2,e_6=Y_3\}$ be a basis of 
$\fg_{-2}$, and $\{e_7=yyy\}$ be a basis of $\fg_{-3}$. Let $\omega^i$ be the dual of $e_i$ for $i=1,\dots,7$.
Note that the gradation is parity consistent, so the prolongations will be parity consistent as well. 

The non-zero bracket relations in $\fm$ are the following (see Section \ref{piv123}):
 \begin{equation*}
[e_1,e_2]=s_3 e_6,\enspace [e_1,e_3]=s_2 e_5,\enspace [e_2,e_3]=s_1 e_4,\enspace
[e_1,e_4]=-e_7,\enspace [e_2,e_5]=- e_7, \enspace [e_3,e_6]=-e_7.
 \end{equation*}

\smallskip

First, consider $\pr_0$. We have $\fg_{-1}\otimes\fg_{-1}^\ast=\CC^{0|3}\otimes\CC^{0|3}=\CC^{9|0}$.
If $\varphi=(a_1e_1+a_2e_2+a_3e_3)\otimes\omega^1+(b_1e_1+b_2e_2+b_3e_3)\otimes\omega^2+(c_1e_1+c_2e_2+c_3e_3)\otimes\omega^3\in\pr_0(\fm)\subset\fg_{-1}\otimes\fg_{-1}^\ast$, then:
 \begin{eqnarray*}
s_3\varphi e_6 &=&\varphi[e_1,e_2]=[\varphi e_1,e_2]+[e_1,\varphi e_2]=
   a_3s_1e_4+b_3s_2e_5+s_3(a_1+b_2)e_6,\\
s_2\varphi e_5 &=&\varphi[e_1,e_3]=[\varphi e_1,e_3]+[e_1,\varphi e_3]=
   a_2s_1e_4+s_2(a_1+c_3)e_5+c_2s_3e_6,\\
s_1\varphi e_4 &=&\varphi[e_2,e_3]=[\varphi e_2,e_3]+[e_2,\varphi e_3]=
   s_1(b_2+c_3)e_4+b_1s_2e_5+c_1s_3e_6.
 \end{eqnarray*}
This yields 
 \begin{eqnarray*}
0&=&\varphi[e_1,e_5]=[\varphi e_1,e_5]+[e_1,\varphi e_5]=[a_1e_1+a_2e_2+a_3e_3,e_5]\\
   && {}+[e_1,s_2^{-1}(a_2s_1e_4+c_2s_3e_6)+(a_1+c_3)e_5]=a_2\frac{s_3}{s_2}e_7 
   \Rightarrow\boxed{a_2=0},\\
0&=&\varphi[e_1,e_6]=[\varphi e_1,e_6]+[e_1,\varphi e_6]=[a_1e_1+a_2e_2+a_3e_3,e_6]\\ 
   && {}+[e_1,s_3^{-1}(a_3s_1e_4+b_3s_2e_5)+(a_1+b_2)e_6]=a_3\frac{s_2}{s_3}e_7
   \Rightarrow\boxed{a_3=0},\\
0&=&\varphi[e_2,e_4]=[\varphi e_2,e_4]+[e_2,\varphi e_4]=[b_1e_1+b_2e_2+b_3e_3,e_4]\\
   && {}+[e_2,s_1^{-1}(b_1s_2e_5+c_1s_3e_6)+(b_2+c_3)e_4]=b_1\frac{s_3}{s_1}e_7
   \Rightarrow\boxed{b_1=0},\\
0&=&\varphi[e_2,e_6]=[\varphi e_2,e_6]-[e_2,\varphi e_6]=[b_1e_1+b_2e_2+b_3e_3,e_6]\\
   && {}+[e_2,s_3^{-1}(a_3s_1e_4+b_3s_2e_5)+(a_1+b_2)e_6]=b_3\frac{s_1}{s_3}e_7
   \Rightarrow\boxed{b_3=0},\\
0&=&\varphi[e_3,e_4]=[\varphi e_3,e_4]+[e_3,\varphi e_4]=[c_1e_1+c_2e_2+c_3e_3,e_4]\\
   && {}+[e_3,s_1^{-1}(b_1s_2e_5+c_1s_3e_6)+(b_2+c_3)e_4]=c_1\frac{s_2}{s_1}e_7
   \Rightarrow\boxed{c_1=0},\\
0&=&\varphi[e_3,e_5]=[\varphi e_3,e_5]+[e_3,\varphi e_5]=[c_1e_1+c_2e_2+c_3e_3,e_5]\\
   && {}+[e_3,s_2^{-1}(a_2s_1e_4+c_2s_3e_6)+(a_1+c_3)e_5]=c_2\frac{s_1}{s_2}e_7
   \Rightarrow\boxed{c_2=0}.
 \end{eqnarray*}
Furthermore,
 \begin{eqnarray*}
-\varphi e_7&=&\varphi[e_1,e_4]=[\varphi e_1,e_4]+[e_1,\varphi e_4]=[a_1e_1,e_4]+[e_1,(b_2+c_3)e_4]=-(a_1+b_2+c_3)e_7,\\
-\varphi e_7&=&\varphi[e_2,e_5]=[\varphi e_2,e_5]+[e_2,\varphi e_5]=[b_2e_2,e_5]+[e_2,(a_1+c_3)e_5]=-(a_1+b_2+c_3)e_7,\\
-\varphi e_7&=&\varphi[e_3,e_6]=[\varphi e_3,e_6]+[e_3,\varphi e_6]=[c_3e_3,e_6]+[e_3,(a_1+b_2)e_6]=-(a_1+b_2+c_3)e_7,
 \end{eqnarray*}
which implies
 \[
\varphi[e_1,e_7]=\varphi[e_1,e_1]=\varphi[e_2,e_2]=\varphi[e_3,e_3]=\varphi[e_7,e_7]=0.
 \]
We conclude that $\pr_0(\fm)=\langle  e_1\otimes\omega^1, e_2\otimes\omega^2, e_3\otimes\omega^3\rangle
=\fg_0$.

\medskip

Next, consider $\pr_1$. We have $\fg_0\otimes\fg_{-1}^\ast=\CC^{3|0}\otimes\CC^{0|3}=\CC^{0|9}$.
If $\varphi=(a_1H_1+a_2H_2+a_3H_3)\otimes\omega^1+(b_1H_1+b_2H_2+b_3H_3)\otimes\omega^2+(c_1H_1+c_2H_2+c_3H_3)\otimes\omega^3\in\pr_1(\fm)\subset\fg_0\otimes\fg_{-1}^\ast$, then:
 \begin{eqnarray*}
s_3\varphi e_6 &=&\varphi[e_1,e_2]=[\varphi e_1,e_2]-[e_1,\varphi e_2]=(a_2-a_1-a_3)e_2+(b_1-b_2-b_3)e_1,\\
s_2\varphi e_5 &=&\varphi[e_1,e_3]=[\varphi e_1,e_3]-[e_1,\varphi e_3]=(a_3-a_1-a_2)e_3+(c_1-c_2-c_3)e_1,\\
s_1\varphi e_4 &=&\varphi[e_2,e_3]=[\varphi e_2,e_3]-[e_2,\varphi e_3]=(b_3-b_1-b_2)e_3+(c_2-c_1-c_3)e_2.
 \end{eqnarray*}
Furthermore, 
 \begin{eqnarray*}
0&=&\varphi[e_1,e_5]=[\varphi e_1,e_5]-[e_1,\varphi e_5]=-2a_2e_5-[e_1,s_2^{-1}
   \bigl((c_1-c_2-c_3)e_1+(a_3-a_1-a_2)e_3\bigr)]\\
   &=& (a_1-a_2-a_3)e_5\Rightarrow\boxed{a_1-a_2-a_3=0},\\
0&=&\varphi[e_2,e_6]=[\varphi e_2,e_6]-[e_2,\varphi e_6]=-2b_3e_6-[e_2,s_3^{-1}
   \bigl((b_1-b_2-b_3)e_1+(a_2-a_1-a_3)e_2\bigr)]\\
  &=& (b_2-b_1-b_3)e_6\Rightarrow\boxed{b_2-b_1-b_3=0},\\
0&=&\varphi[e_3,e_4]=[\varphi e_3,e_4]-[e_3,\varphi e_4]=-2c_1e_4-[e_3,s_1^{-1}
   \bigl((c_2-c_1-c_3)e_2+(b_3-b_1-b_2)e_3\bigr)]\\
   &=& (c_3-c_1-c_2)e_4\Rightarrow\boxed{c_3-c_1-c_2=0}
 \end{eqnarray*}
and this implies $\varphi[e_1,e_6]=\varphi[e_2,e_4]=\varphi[e_3,e_5]=0$. In addition, 
 \begin{eqnarray*} 
-\varphi e_7&=&\varphi[e_1,e_4]=-2a_1e_4-[e_1,s_1^{-1}
   \bigl((c_2-c_1-c_3)e_2+(b_3-b_1-b_2)e_3\bigr)]=\\
  &=& -2a_1e_4-\frac{s_2}{s_1}(b_3-b_1-b_2)e_5-\frac{s_3}{s_1}(c_2-c_1-c_3)e_6,\\
-\varphi e_7&=&\varphi[e_2,e_5]=-2b_2e_5-[e_2,s_2^{-1}
   \bigl((c_1-c_2-c_3)e_1+(a_3-a_1-a_2)e_3\bigr)]=\\
  &=& -2b_2e_5-\frac{s_1}{s_2}(a_3-a_1-a_2)e_4-\frac{s_3}{s_2}(c_1-c_2-c_3)e_6,\\
-\varphi e_7&=&\varphi[e_3,e_6]=-2c_3e_6-[e_3,s_3^{-1}
   \bigl((b_1-b_2-b_3)e_1+(a_2-a_1-a_3)e_2\bigr)]=\\
   &=& -2c_3e_6-\frac{s_1}{s_3}(a_2-a_1-a_3)e_4-\frac{s_2}{s_3}(b_1-b_2-b_3)e_5.
   \end{eqnarray*}
Consequently, 
$\boxed{\frac{a_2}{s_2}=\frac{a_3}{s_3}=-\frac{a_1}{s_1}}$, 
$\boxed{\frac{b_1}{s_1}=\frac{b_3}{s_3}=-\frac{b_2}{s_2}}$, 
$\boxed{\frac{c_1}{s_1}=\frac{c_2}{s_2}=-\frac{c_3}{s_3}}$, and we conclude that
$\pr_1(\fm)=\langle(-s_1H_1+s_2H_2+s_3H_3)\otimes\omega^1,(s_1H_1-s_2H_2+s_3H_3)\otimes\omega^2,(s_1H_1+s_2H_2-s_3H_3)\otimes\omega^3\rangle=\fg_1$.

\medskip

Next,  consider $\pr_2$. We have $\fg_1\otimes\fg_{-1}^\ast=\CC^{0|3}\otimes\CC^{0|3}=\CC^{9|0}$.
If $\varphi=(a_1f_1+a_2f_2+a_3f_3)\otimes\omega^1+(b_1f_1+b_2f_2+b_3f_3)\otimes\omega^2+(c_1f_1+c_2f_2+c_3f_3)\otimes\omega^3\in\pr_2(\fm)\subset\fg_1\otimes\fg_{-1}^\ast$, then:
 \begin{eqnarray*}
 s_3\varphi e_6 &=&\varphi[e_1,e_2]=[\varphi e_1,e_2]+[e_1,\varphi e_2]=
 \frac{a_2-b_1}{2}s_1H_1+\frac{b_1-a_2}{2}s_2H_2+\frac{a_2+b_1}{2}s_3H_3,\\
 s_2\varphi e_5 &=&\varphi[e_1,e_3]=[\varphi e_1,e_3]+[e_1,\varphi e_3]=
 \frac{a_3-c_1}{2}s_1H_1+\frac{a_3+c_1}{2}s_2H_2+\frac{c_1-a_3}{2}s_3H_3,\\
 s_1\varphi e_4 &=&\varphi[e_2,e_3]=[\varphi e_2,e_3]+[e_2,\varphi e_3]=
 \frac{b_3+c_2}{2}s_1H_1+\frac{b_3-c_2}{2}s_2H_2+\frac{c_2-b_3}{2}s_3H_3,
 \end{eqnarray*}
whence
 \begin{eqnarray*}
-\varphi e_7&=&\varphi[e_1,e_4]=[a_1f_1+a_2f_2+a_3f_3,e_4]+\left[e_1,
  \frac{b_3+c_2}{2}H_1+\frac{b_3-c_2}{2}\frac{s_2}{s_1}H_2+\frac{c_2-b_3}{2}\frac{s_3}{s_1}H_3\right]\\
  &=& \left(-\frac{b_3+c_2}{2}+\frac{b_3-c_2}{2}\frac{s_2}{s_1}+\frac{c_2-b_3}{2}\frac{s_3}{s_1}\right)e_1
  -a_3e_2-a_2e_3,\\
-\varphi e_7&=&\varphi[e_2,e_5]=[b_1f_1+b_2f_2+b_3f_3,e_5]+\left[e_2,
   \frac{a_3-c_1}{2}\frac{s_1}{s_2}H_1+\frac{a_3+c_1}{2}H_2+\frac{c_1-a_3}{2}\frac{s_3}{s_2}H_3\right]\\
  &=& -b_3e_1
  +\left(\frac{a_3-c_1}{2}\frac{s_1}{s_2}-\frac{a_3+c_1}{2}+\frac{c_1-a_3}{2}\frac{s_3}{s_2}\right)e_2-b_1e_3,\\
-\varphi e_7&=&\varphi[e_3,e_6]=[c_1f_1+c_2f_2+c_3f_3,e_6]+\left[e_3,
   \frac{a_2-b_1}{2}\frac{s_1}{s_3}H_1+\frac{b_1-a_2}{2}\frac{s_2}{s_3}H_2+\frac{a_2+b_1}{2}H_3\right]\\
  &=& -c_2e_1-c_1e_2
  +\left(\frac{a_2-b_1}{2}\frac{s_1}{s_3}+\frac{b_1-a_2}{2}\frac{s_2}{s_3}-\frac{a_2+b_1}{2}\right)e_3.
 \end{eqnarray*}
Thus $\boxed{b_1=a_2}$, $\boxed{c_2=b_3}$, $\boxed{a_3=c_1}$, and
$\varphi e_4=b_3H_1$, $\varphi e_5=c_1H_2$, $\varphi e_6=a_2H_3$. Furthermore,
 \begin{eqnarray*}
0&=&\varphi[e_1,e_5]=[\varphi e_1,e_5]+[e_1,\varphi e_5]=-a_1e_3\Rightarrow\boxed{a_1=0},\\
0&=&\varphi[e_2,e_6]=[\varphi e_2,e_6]+[e_2,\varphi e_6]=-b_2e_1\Rightarrow\boxed{b_2=0},\\
0&=&\varphi[e_3,e_4]=[\varphi e_3,e_4]+[e_3,\varphi e_4]=-c_3e_2\Rightarrow\boxed{c_3=0},
 \end{eqnarray*}
and we also compute $\varphi[e_1,e_6]=\varphi[e_2,e_4]=\varphi[e_3,e_5]=0$. 
We conclude $\pr_2(\fm)=\langle f_2\otimes\omega^1+ f_1\otimes\omega^2, f_3\otimes\omega^1+ f_1\otimes\omega^3,f_3\otimes\omega^2+ f_2\otimes\omega^3\rangle=\fg_2$.

\medskip

Next, consider $\pr_3$.
We have $\fg_2\otimes\fg_{-1}^\ast=\CC^{3|0}\otimes\CC^{0|3}=\CC^{0|9}$.
If $\varphi=(a_1f_4+a_2f_5+a_3f_6)\otimes\omega^1+(b_1f_4+b_2f_5+b_3f_6)\otimes\omega^2+(c_1f_4+c_2f_5+c_3f_6)\otimes\omega^3\in\pr_3(\fm)\subset\fg_2\otimes\fg_{-1}^\ast$, then:
 \begin{eqnarray*}
s_3\varphi e_6 &=&\varphi[e_1,e_2]=[\varphi e_1,e_2]-[e_1,\varphi e_2]=a_3f_1+b_3f_2+(a_1+b_2)f_3,\\
s_2\varphi e_5 &=&\varphi[e_1,e_3]=[\varphi e_1,e_3]-[e_1,\varphi e_3]=a_2f_1+(a_1+c_3)f_2+c_2f_3,\\
s_1\varphi e_4 &=&\varphi[e_2,e_3]=[\varphi e_2,e_3]-[e_2,\varphi e_3]=(b_2+c_3)f_1+b_1f_2+c_1f_3,
 \end{eqnarray*}
and so
 \begin{eqnarray*}
-\varphi e_7&=&\varphi[e_1,e_4]=[a_1f_4+a_2f_5+a_3f_6,e_4]
   -\Bigr[e_1,s_1^{-1}\bigl((b_2+c_3)f_1+b_1f_2+c_1f_3\bigr)\Bigl]\\
  &=& \left(a_1+\frac{b_2+c_3}2\right)H_1-\frac{b_2+c_3}2\frac{s_2}{s_1}H_2
  -\frac{b_2+c_3}2\frac{s_3}{s_1}H_3,\\
-\varphi e_7&=&\varphi[e_2,e_5]=[b_1f_4+b_2f_5+b_3f_6,e_5]
   -\Bigr[e_2,s_2^{-1}\bigl(a_2f_1+(a_1+c_3)f_2+c_2f_3\bigr)\Bigr]\\
  &=& -\frac{a_1+c_3}2\frac{s_1}{s_2}H_1+\left(b_2+\frac{a_1+c_3}2\right)H_2
  -\frac{a_1+c_3}2\frac{s_3}{s_2}H_3,\\
-\varphi e_7&=&\varphi[e_3,e_6]=[c_1f_4+c_2f_5+c_3f_6,e_6]
   -\Bigl[e_3,s_3^{-1}\bigl(a_3f_1+b_3f_2+(a_1+b_2)f_3\bigr)\Bigr]\\
  &=& -\frac{a_1+b_2}2\frac{s_1}{s_3}H_1-\frac{a_1+b_2}2\frac{s_2}{s_3}H_2
  +\left(c_3+\frac{a_1+b_2}2\right)H_3.
 \end{eqnarray*}
This implies $\boxed{\frac{a_1}{s_1}=\frac{b_2}{s_2}=\frac{c_3}{s_3}}$ and we continue:
 \begin{eqnarray*}
0&=&\varphi[e_1,e_5]=[\varphi e_1,e_5]-[e_1,\varphi e_5]=
   \frac{a_2}{2s_2}\left(s_1H_1+s_2H_2-s_3H_3\right)\Rightarrow\boxed{a_2=0},\\
0&=&\varphi[e_1,e_6]=[\varphi e_1,e_6]-[e_1,\varphi e_6]=
   \frac{a_3}{2s_3}\left(s_1H_1-s_2H_2+s_3H_3\right)\Rightarrow\boxed{a_3=0},\\
0&=&\varphi[e_2,e_4]=[\varphi e_2,e_4]-[e_2,\varphi e_4]=
   \frac{b_1}{2s_1}\left(s_1H_1+s_2H_2-s_3H_3\right)\Rightarrow\boxed{b_1=0},\\
0&=&\varphi[e_2,e_6]=[\varphi e_2,e_6]-[e_2,\varphi e_6]=
   \frac{b_3}{2s_3}\left(-s_1H_1+s_2H_2+s_3H_3\right)\Rightarrow\boxed{b_3=0},\\
0&=&\varphi[e_3,e_4]=[\varphi e_3,e_4]-[e_3,\varphi e_4]=
   \frac{c_1}{2s_1}\left(s_1H_1-s_2H_2+s_3H_3\right)\Rightarrow\boxed{c_1=0},\\
0&=&\varphi[e_3,e_5]=[\varphi e_3,e_5]-[e_3,\varphi e_5]=
   \frac{c_2}{2s_2}\left(-s_1H_1+s_2H_2+s_3H_3\right)\Rightarrow\boxed{c_2=0}.
\end{eqnarray*}
We conclude that $\pr_3(\fm)=\langle s_1 f_4\otimes\omega^1+s_2 f_5\otimes \omega^2+s_3 f_6\otimes\omega^3\rangle=\fg_3$.

\medskip

Finally, consider $\pr_4$.
We have $\fg_3\otimes\fg_{-1}^\ast=\CC^{0|1}\otimes\CC^{0|3}=\CC^{3|0}$.
If $\varphi=f_7\otimes(a\omega^1+b\omega^2+c\omega^3)\in\pr_4(\fm)\subset\fg_3\otimes\fg_{-1}^\ast$, then:
 \begin{eqnarray*}
s_3\varphi e_6 &=&\varphi[e_1,e_2]=[\varphi e_1,e_2]+[e_1,\varphi e_2]=[af_7,e_2]+[e_1,bf_7]=as_2f_5+bs_1f_4,\\
s_2\varphi e_5 &=&\varphi[e_1,e_3]=[\varphi e_1,e_3]+[e_1,\varphi e_3]=[af_7,e_3]+[e_1,cf_7]=as_3f_6+cs_1f_4, \\
s_1\varphi e_4 &=&\varphi[e_2,e_3]=[\varphi e_2,e_3]+[e_2,\varphi e_3]=[bf_7,e_3]+[e_2,cf_7]=bs_3f_6+cs_2f_5,
 \end{eqnarray*}
whence
 \begin{eqnarray*}
0&=&\varphi[e_1,e_5]=[af_7,e_5]+[e_1,s_2^{-1}(as_3f_6+cs_1f_4)]=a\frac{s_1}{s_2}f_2\Rightarrow\boxed{a=0}, \\
0&=&\varphi[e_2,e_4]=[bf_7,e_4]+[e_2,s_1^{-1}(bs_3f_6+cs_2f_5)]=b\frac{s_2}{s_1}f_1\Rightarrow\boxed{b=0}, \\
0&=&\varphi[e_3,e_4]=[cf_7,e_4]+[e_3,s_1^{-1}(bs_3f_6+cs_2f_5)]=c\frac{s_3}{s_1}f_1\Rightarrow\boxed{c=0}.
 \end{eqnarray*}
We conclude that $\pr_4(\fm)=\fg_4=0$, which finishes the proof. \hfill\qed

\section{Hochshild--Serre spectral sequence}\label{ApHS}

For a Lie superalgebra $\fm$ and its module $\fg$, the generalized Spencer cohomology $H^n(\fm,\fg)$ 
is the cohomology of the Chevalley-Eilenberg complex $C^n(\fm,\fg)=\fg\otimes\Lambda^n\fm^*$ with 
the differential $\pd$ given by the standard formula with the sign rule \cite{Kac,L}. 
We assume $\fg$ is a simple Lie superalgebra, namely $D(2,1;a)$, and $\fm$ the nilradical $\fg_{-}$ 
of the opposite parabolic $\fp^\text{op}=\fp^\perp\subset\fg$, where orthogonality is computed
with respect to the analog of the Killing form \eqref{KFB}.

To compute this cohomology we use the spectral sequence \cite{Fuks}. Consider the filtration\linebreak 
$0=F^{n+1}C^n\subset F^nC^n\subset\cdots \subset F^1C^n\subset F^0C^n=C^n$, where for 
$q,p\in\NN\cup0$:
 \begin{equation*}
F^pC^n(\fm,\fg)=\{\varphi\in C^n(\fm,\fg):\,
\iota_{v_1}\cdots\iota_{v_m}\varphi=0\ \forall\,v_j\in\fm_{\bar0},\,\forall\,m>n-p\}.
 \end{equation*}
 
The Hochshild--Serre spectral sequence is associated to the filtration, $(E_s^{\bullet,\bullet},\partial_s)_{s\geq0}$, 
with differentials $\partial_s:E_s^{p,q}\longrightarrow E_s^{p+s,q+1-s}$ induced by the Chevalley--Eilenberg 
differentials, and it converges to the generalized Spencer cohomology
$E^{\bullet,\bullet}_\infty\Rightarrow H^\bullet(\fm,\fg)$. The initial page is
 \begin{equation*}
E_0^{p,q}=\frac{F^pC^{p+q}(\fm,\fg)}{F^{p+1}C^{p+q}(\fm,\fg)}=
\fg\otimes\Lambda^p(\fm_{\bar1})^\ast\otimes\Lambda^q(\fm_{\bar0})^\ast  
 \end{equation*}
and $E_{s+1}=H^\bullet(E_s,\partial_s)$. This implies generally, for all $p,q\geq0$:
 \begin{equation}\label{E1pq}
E_1^{p,q}=H^q\left(\fm_{\bar0},\fg\otimes\Lambda^p(\fm_{\bar1}^\ast)\right).
 \end{equation}
By diagram reasons, for the cohomology we are interested in, the stabilization occurs as follows:
 \begin{proposition}\label{H0H1H2}
$H^0(\fm,\fg)=E_2^{0,0}$,  $H^1(\fm,\fg)=E_2^{1,0}\oplus E_3^{0,1}$,  $H^2(\fm,\fg)=E_3^{2,0}\oplus E_3^{1,1}\oplus E_4^{0,2}$. 
 \end{proposition}

Here we follow closely the development in \cite{G3super} in a similar situation. In particular, because
in the pure parity contact grading $\fm_{\bar1}$ is the trivial $\fm_{\bar0}$ module, from \eqref{E1pq} we conclude:
 \begin{lemma}\label{H1H2}
For the $\fp_1^\I$ consistent gradation of $\fg=D(2,1;a)$ we have: 
$E_1^{p,q}=H^q(\fm_{\bar0},\fg)\otimes\Lambda^p(\fm_{\bar1})^*$.
 \end{lemma}
Now and until the end of this section, we restrict to the contact grading associated to $\fp_1^\I$.

\smallskip

In order to compute $E^{p,q}_1$, we need the structure of the $\fg_0$-modules $\Lambda^p(\fm_{\bar1})^*$ 
and $H^q(\fm_{\bar0},\fg)$ for $p+q\leq3$. We will obtain the latter using Kostant's theorem.
Recall that the Weyl group of $\fg$ is that of the even part $\fg_{\bar0}$ and for a parabolic $\fp$
the group $W_\fp$ is defined similarly. 

The Hasse diagram $W^\fp$ is a subset of $W$ isomorphic to the coset $W/W_\fp$ split into disjoint subsets 
$W^\fp(i)$ according the length of representative words. In particular, for a set $\chi\subset\mathcal{A}$ of crosses 
on a Dynkin diagram $\Xi$ marking the parabolic $\fp$, 
$W^\fp(0)=\emptyset$, $W^\fp(1)=\{(j)\,|\,j\in\chi\}$ and 
$W^\fp(2)=\{(jk)\,:\,\chi\ni j\neq k\in\chi\cup\mathcal{N}(j),\}$, where 
$\mathcal{N}(j)=\{i\in\mathcal{A}\,|\,C_{ij}\neq0\}$ is the set of neighbors of $j$ 
(connected to $j$-th node by an edge in the Dynkin diagram).

 \begin{theorem}[\cite{Ko}]
Let $\fg$ be a $\ZZ$-graded semisimple Lie algebra with highest weight $\lambda$. 
Define $\rho=\sum_i\lambda_i$ where $\lambda_i$ are the fundamental weights. 
The affine action of $W$ is $w\cdot\lambda=w(\lambda+\rho)-\rho$. Then
 \begin{equation*}
H^k(\fm,\fg)=\bigoplus_{w\in W^{\fp}(k)}\mathbb{M}_{-w\cdot\lambda}
 \end{equation*}
where $\mathbb{M}_\mu$ is the simple $\fg_0$-module (irrep) with lowest weight $\mu$.
 \end{theorem}

Below, $\fg=\fg_{-2}\oplus\fg_{-1}\oplus\fg_0\oplus\fg_1\oplus\fg_2$ is the contact grading associated to 
$\fp_1^\I$, $\fg_0=\CC\oplus\sl_2\oplus\sl_2$, and by $\VV_{k,l}[d]$ resp.\ $\UU_{k,l}[d]$ we denote the 
even resp.\ odd irreducible representation of $\fg_0$ with $\sl_2\oplus\sl_2$-weight $(k,l)$ and $\CC$-degree $d$.
We will use the notations of Section \ref{pi1}.

 \begin{figure}[h!]\centering
 \begin{minipage}{0.5\textwidth}\centering\begin{small}
 \[
\begin{array}{|c|c|c|}\hline
k &  \Lambda^p(\mathfrak{m}_{\bar1})^\ast \vphantom{\frac{A^a}{A}} & \text{lowest weight vector} \\
\hline
3 & \begin{array}{l}
     \UU_{3,3}[3] \\ 
     \UU_{1,1}[3]
    \end{array}& \begin{array}{c}
     \omega^1\wedge \omega^1\wedge \omega^1 \vphantom{\frac{A^A}{A}} \\ 
     \omega^1\wedge \omega^1\wedge \omega^4-\omega^1\wedge \omega^2\wedge \omega^3
    \end{array}\\
\hline
2 & \begin{array}{l}
     \VV_{2,2}[2] \\ 
     \VV_{0,0}[2]
    \end{array} & \begin{array}{c}
     \omega^1\wedge \omega^1 \vphantom{\frac{A^A}{A}} \\ 
     \omega^1\wedge \omega^4- \omega^2\wedge \omega^3
    \end{array}\\
\hline
1 &  \UU_{1,1}[1] & \omega^1 \vphantom{\frac{A^A}{A}} \\
    \hline
    0 & \mathbb{C} & \\
    \hline
\end{array}
 \]
\end{small}
 \caption{The modules $\Lambda^p\left(\fm_{\bar1}\right)^*$.\!\!\!}
    \end{minipage}\hfill
    \begin{minipage}{0.5\textwidth}\centering\begin{small}
 \[
\begin{array}{|c|c|c|}\hline
q &  H^q(\mathfrak{m}_{\bar0},\fg) \vphantom{\frac{A^a}{A}} & \text{lowest weight vector}\\
\hline
3 & 0 &\\
\hline
2 & 0 &\\
\hline
1 & \begin{array}{l}
 \VV_{0,0}[4] \\ \UU_{1,1}[3] \\ \VV_{2,0}[2] \\ \VV_{0,2}[2] 
    \end{array}
 & 
 \begin{array}{r}
 [X_1\otimes\omega^5] \vphantom{\frac{A^A}{A}} \\{}
 [xyy\otimes\omega^5] \\{} [Y_2\otimes\omega^5] \\{} [Y_3\otimes\omega^5]
    \end{array}\\
    \hline
    0 & \begin{array}{l}
 \VV_{2,0}[0] \\ \VV_{0,2}[0] \\ \UU_{1,1}[-1] \\ \VV_{0,0}[-2]     
    \end{array} & 
    \begin{array}{l}
 [Y_2] \vphantom{\frac{A^A}{A}} \\{}
 [Y_3] \\{} [e_4] \\{} [e_5]
    \end{array}\\
    \hline
\end{array}
 \]
 \end{small}
    \caption{The modules $H^q(\fm_{\bar0},\fg)$; \ $\omega^5$ dual to $s_1X_1$.}
    \end{minipage}
\end{figure}

For $p+q=k\leq2$, $E_1^{p,q}$ only has degrees $d\leq 5$, and we are only interested in cohomology
$H^{d,k}(\fm,\fg)=H^k(\fm,\fg)_d$ for $d\geq0$. We only display the terms of degrees $0\leq d\leq5$ 
in Table \ref{table:DE1}:

\begin{tiny}
\begin{table}[h!]
\[
\hspace{-10pt}\begin{array}{|c|c|c|c|}
\hline
0 & 0 & 0 & 0 \vphantom{\frac{A}{A}}\\ \hline 
\begin{array}{c}
     \VV_{0,0}[4] \\ \oplus\\
     \UU_{1,1}[3]\\ \oplus\\
     \VV_{2,0}[2]\\ \oplus\\
     \VV_{0,2}[2]
\end{array} & 
\begin{array}{c}
     \UU_{1,1}[5]\\ \oplus\\
     \VV_{2,2}[4] \oplus\VV_{2,0}[4] \oplus\VV_{0,2}[4] \oplus\VV_{0,0}[4] \\ \oplus\\
     \UU_{3,1}[3]\oplus\UU_{1,1}[3] \\ \oplus\\ 
     \UU_{1,3}[3]\oplus\UU_{1,1}[3]
\end{array} & 
\begin{array}{c}
     \UU_{3,3}[5]\oplus\UU_{3,1}[5] \oplus\UU_{1,3}[5] \oplus\UU_{1,1}[5] \vphantom{\frac{A^A}{A}}\\ \oplus\\
     \UU_{1,1}[5] \\ \oplus\\
     \VV_{4,2}[4] \oplus \VV_{2,2}[4] \oplus \VV_{0,2}[4] \\ \oplus\\ 
     \VV_{2,4}[4] \oplus \VV_{2,2}[4] \oplus \VV_{2,0}[4] \\ \oplus\\
     \VV_{2,0}[4] \\ \oplus\\
     \VV_{0,2}[4]
\end{array} & *\\ 
\hline
\begin{array}{c}
    \VV_{2,0}[0] \\ \oplus\\
    \VV_{0,2}[0]
\end{array} & 
\begin{array}{c}
    \UU_{3,1}[1] \oplus\UU_{1,1}[1] \\\oplus\\
    \UU_{1,3}[1] \oplus\UU_{1,1}[1]\\ \oplus\\
    \VV_{2,2}[0] \oplus\VV_{2,0}[0] \oplus\VV_{0,2}[0] \oplus\VV_{0,0}[0]
\end{array} & 
\begin{array}{c}
\VV_{4,2}[2] \oplus\VV_{2,2}[2] \oplus\VV_{0,2}[2] \vphantom{\frac{A^A}{A}}\\ \oplus\\
\VV_{2,4}[2] \oplus\VV_{2,2}[2] \oplus\VV_{2,0}[2] \\ \oplus\\
\VV_{2,0}[2] \\ \oplus\\
\VV_{0,2}[2] \\ \oplus\\
\UU_{3,3}[1] \oplus\UU_{3,1}[1] \oplus\UU_{1,3}[1] \oplus\UU_{1,1}[1] \\ \oplus\\
\UU_{1,1}[1] \\ \oplus\\
\VV_{2,2}[0] \\ \oplus\\
\VV_{0,0}[0]
\end{array}
& 
\begin{array}{c}
     \UU_{5,3}[3] \oplus\UU_{3,3}[3] \oplus\UU_{1,3}[3] \\ \oplus\\
     \UU_{3,5}[3] \oplus\UU_{3,3}[3] \oplus\UU_{3,1}[3] \\ \oplus\\
     \UU_{3,1}[3] \oplus\UU_{1,1}[3] \\ \oplus\\ 
     \UU_{1,3}[3] \oplus\UU_{1,1}[3] \\ \oplus\\
     \VV_{4,4}[2] \oplus\VV_{4,2}[2] \oplus\VV_{2,4}[2] \oplus\VV_{2,2}[2] \\ \oplus\\
     \VV_{2,2}[2] \oplus\VV_{0,2}[2] \oplus\VV_{2,0}[2] \oplus\VV_{0,0}[2] \\ \oplus\\
     \UU_{3,3}[1] \\ \oplus\\
     \UU_{1,1}[1] 
    \end{array}
\\ \hline
\end{array}
\]
\caption{The $E_1$-page: modules $E_1^{p,q}=\Lambda^p(\fm_{\bar1})^*\otimes H^q(\fm_{\bar0},\fg)$ for $p+q\leq3$, $d\leq5$.}
\label{table:DE1}
\end{table}
\end{tiny}

The differential $\partial: E_{1}^{p,q}\rightarrow E_{1}^{p+1,q}$ is $(\mathfrak{g}_0)_{\bar0}$-equivariant. 
Since $\VV_{2,0}[2]\not\subset E_{1}^{1,1}$, $\VV_{2,0}[2]\subset E_1^{0,1}$ is in $\ker(\partial)$, and since 
$\VV_{2,0}[2]\not\subset E_{1}^{-1,1}=0$, $\VV_{2,0}[2]\subset E_{1}^{0,1}$ is not in $\im(\partial)$. 
Therefore  $\VV_{2,0}[2]\subset E_{2}^{0,1}$. Analogous arguments show that  
$\VV_{0,2}[2]\subset E_{2}^{0,1}$, $\UU_{3,1}[3],\UU_{1,3}[3]\subset E_{2}^{1,1}$.

 \begin{lemma}\label{Arep}
Let $\mathbb{M}\subset E_s^{p,q}$ be a simple $(\mathfrak{g}_0)_{\bar0}$-module. Assume 
$\mathbb{M}\not\subset\ker({\partial:E_s^{p,q}\rightarrow E_s^{p+s,q-s+1}})$. 
Then $\mathbb{M}\not\subset E_{s+1}^{p,q}$, and 
$\p\mathbb{M}\subset\im(\partial:E_s^{p,q}\rightarrow E_s^{p+s,q-s+1})$. 
If $\p\mathbb{M}\simeq\mathbb{M}$ appears in $E_s^{p+s,q-s+1}$ with multiplicity one, 
we can also conclude that $\mathbb{M}\not\subset E_{s+1}^{p+s,q-s+1}$.
 \end{lemma}

 \begin{example}
Consider the module $\UU_{3,1}[1]\subset E_1^{1,0}=H^0(\fm_{\bar0},\fg)\otimes(\fm_{\bar1})^*$, 
which has lowest weight vector $\phi=[Y_2]\otimes\omega^1$. 
Note that $(\partial\phi)(e_1,e_3)=-e_3\cdot[Y_2]=-[e_4]\neq0$, and 
$(\partial\phi)(e_1,e_3)\in  H^0(\fm_{\bar0},\fg)$. Then $\UU_{3,1}[1]\not\subset\ker(\partial)$, 
$\UU_{3,1}[1]\not\subset E_2^{1,0}$, and $\UU_{3,1}[1]\subset\im(\partial)$. Since $\UU_{3,1}[1]$ appears 
in $E_1^{2,0}$ with multiplicity one, we conclude $\UU_{3,1}[1]\not\subset E_2^{2,0}$. 
 \end{example}

We can use Lemma \ref{Arep} with $s=1$ for the representations in Table \ref{D:E1lwv}.

 \begin{table}[h!]\begin{small}
 \[
 \begin{array}{|c|c|l|l|} \hline
 (p,q) 
& \mbox{$(\fg_{0})_{\bar 0}$-module} & \vphantom{\frac{A^a}{A}}
 \mbox{Lowest weight vector $\phi \in E_1^{p,q}$} & \mbox{$\partial\phi\neq0$}\\ \hline
 (0,0) 
& \VV_{2,0}[0] & [Y_2] & (\partial \phi)(e_1) \neq 0\\ \cline{2-4}
& \VV_{0,2}[0] & [Y_3] & (\partial \phi)(e_1) \neq 0\\ \cline{2-4} \hline
 (1,0)
& \VV_{2,2}[0] & [e_4]\otimes\omega^1 & (\partial \phi)(e_1,e_1) \neq 0\\ \cline{2-4}
& \VV_{0,0}[0] &  [e_1]\otimes\omega^1+[e_2]\otimes\omega^2+[e_3]\otimes\omega^3+[e_4]\otimes\omega^4
& (\partial \phi)(e_1,e_4) \neq 0\\ \cline{2-4}
& \UU_{3,1}[1] &  [Y_2]\otimes\omega^1
& (\partial \phi)(e_1,e_1) \neq 0\\ \cline{2-4} 
& \UU_{1,3}[1] &  [Y_3]\otimes\omega^1
& (\partial \phi)(e_1,e_1) \neq 0\\ \cline{2-4} \hline
 (2,0) 
& \UU_{3,3}[1] & [e_4]\otimes\omega^1\wedge \omega^1 & (\partial \phi)(e_1,e_1,e_1) \neq 0\\ 
 \cline{2-4}
& \VV_{2,4}[2] & [Y_3]\otimes\omega^1\wedge \omega^1 & (\partial \phi)(e_1,e_1,e_1) \neq 0\\ \cline{2-4}
& \VV_{4,2}[2] & [Y_2]\otimes\omega^1\wedge \omega^1 & (\partial \phi)(e_1,e_1,e_1) \neq 0\\ \cline{2-4}\hline
 (0,1)
& \VV_{0,0}[4] & [X_1\otimes\omega^5] & (\partial \phi)(e_1) \neq 0\\ \cline{2-4}
& \UU_{1,1}[3] & [xyy\otimes\omega^5] & (\partial \phi)(e_2) \neq 0\\ \hline
 (1,1) 
& \UU_{1,1}[5] & [X_1\otimes\omega^5]\otimes\omega^1 & (\partial \phi)(e_1,e_1) \neq 0\\ \cline{2-4}
\cline{2-4}
& \VV_{2,2}[4] & [xyy\otimes\omega^5]\otimes\omega^1 & (\partial \phi)(e_1,e_2) \neq 0\\ \cline{2-4}
& \VV_{2,0}[4] & [xyx\otimes\omega^5]\otimes\omega^1+[xyy\otimes\omega^5]\otimes\omega^3 & (\partial \phi)(e_1,e_4) \neq 0 \\ \cline{2-4}
& \VV_{0,2}[4] & [xxy\otimes\omega^5]\otimes\omega^1+[xyy\otimes\omega^5]\otimes\omega^2 & (\partial \phi)(e_1,e_4) \neq 0\\ \hline
 \end{array}
 \]
 \end{small}
 \caption{Modules in $E_{1}^{p,q}= H^q(\fm_{\bar{0}},\fg)\otimes\Lambda^p (\fm_{\bar1})^\ast$, $p+q\leq3$, with multiplicity one, $0\leq d\leq5$.} \label{D:E1lwv}
 \end{table} 
 
Of the two copies of $\UU_{1,1}[3]$ in $E_1^{1,1}$, only one survives to $E_2$, since there is no $\UU_{1,1}[3]$ in $E_1^{2,1}$, therefore $2\UU_{1,1}[3]\subset\ker(\partial_1)$, and there is one $\UU_{1,1}[3]$ in $\im(\partial)$.

 \begin{example}
A lowest weight vector in $2\UU_{1,1}[1]\subset E_1^{1,0}$ is of the form
 \begin{equation*}
\phi=a\left([H_2]\otimes\omega^1+2[Y_2]\otimes\omega^2\right)+b\left([H_3]\otimes\omega^1+2[Y_3]\otimes\omega^3\right),
 \end{equation*}
and $\partial\phi=0$ if and only if $a=-b$. For generic $a,b$ we get $\phi\not\in\ker(\partial)$, 
but the condition $a=-b$ yields a vector in $\ker(\partial)$, therefore $\UU_{1,1}[1]\subset E_2^{1,0}$, 
since there is no $\UU_{1,1}[1]$ in $E_1^{0,0}$. 

A lowest weight vector in $2\UU_{1,1}[1]\subset E_1^{2,0}$ is of the form
 \begin{eqnarray*}
&\psi=c\left(2[e_1]\otimes\omega^1\wedge\omega^1+2[e_2]\otimes\omega^1\wedge\omega^2+2[e_3]\otimes\omega^1\wedge\omega^3+[e_4]\otimes(\omega^1\wedge\omega^4+\omega^2\wedge \omega^3)\right)\\
&+d\left([e_4]\otimes(\omega^1\wedge \omega^4-\omega^2\wedge\omega^3)\right)
 \end{eqnarray*}
and $\partial\psi=0$ if and only if $c=-3d$. For generic $c,d$ we get $\psi\not\in\ker(\partial)$, 
but the condition $c=-3d$ yields a vector in $\ker(\partial)$ that is also in $\im(\partial)$, 
since in that case $\partial\phi\sim\psi$. 
 \end{example}
 
Similar arguments determine weather the modules in Table \ref{D:E1lwvdouble} survive to the second page. 
 
 \begin{center}\begin{table}[h!]\begin{tiny}
 \[
 \hspace{-25pt} \begin{array}{|c|c|l|l|} \hline
 (p,q) & \mbox{$(\fg_{0})_{\bar 0}$-module} & \vphantom{\frac{A^a}{A}}
\mbox{Form of l.w.v. in $E_1^{p,q}$} & \mbox{Remarks}\\ \hline
 (1,0) 
 & 2\UU_{1,1}[1] & 
\varphi=a\left([H_2]\otimes\omega^1+2[Y_2]\otimes\omega^2\right)+b\left([H_3]\otimes\omega^1+2[Y_3]\otimes\omega^3\right) &
 \begin{array}{c}
 \partial \varphi= 0\Leftrightarrow b=-a,\\
 \partial \varphi\sim\phi\Leftrightarrow b\neq-a,
 \end{array}
 \\ \hline
 (2,0)  
  & 2\UU_{1,1}[1] & 
  \begin{array}{c}
\phi=a\left(2[e_1]\otimes\omega^1\wedge\omega^1+2[e_2]\otimes\omega^1\wedge\omega^2+2[e_3]\otimes
\omega^1\wedge\omega^3+[e_4]\otimes(\omega^1\wedge\omega^4+\omega^2\wedge \omega^3)\right)\\
+b\left([e_4]\otimes(\omega^1\wedge \omega^4-\omega^2\wedge\omega^3)\right)\\
  \end{array}  &
 \partial\phi=0\Leftrightarrow b=-3a \\ \cline{2-4}\\ \cline{2-4} & 2\VV_{2,2}[2] &
 \xi=a\left([H_2]\otimes\omega^1\wedge  \omega^1+2[Y_2]\otimes\omega^1\wedge \omega^2\right)
    +b\left([H_3]\otimes\omega^1\wedge \omega^1+2[Y_3]\otimes\omega^1\wedge \omega^3\right)  &
  \partial \xi= 0\Leftrightarrow b=-a \\ \cline{2-4}
  &  2\VV_{2,0}[2]& \xi=a[Y_2]\otimes\left(\omega^1\wedge \omega^4-\omega^2\wedge\omega^3\right) +b[X_3]\otimes\left(\omega^1\wedge \omega^1-[H_3]\otimes\omega^1\wedge\omega^3-[Y_3]\otimes\omega^3\wedge\omega^3\right)
  & \partial \xi=0\Leftrightarrow b=-a \\ \cline{2-4}
  &  2\VV_{0,2}[2]& \xi=a[Y_3]\otimes\left(\omega^1\wedge \omega^4-\omega^2\wedge\omega^3\right) +b\left([X_2]\otimes\omega^1\wedge \omega^1-[H_2]\otimes\omega^1\wedge\omega^2-[Y_2]\otimes\omega^2\wedge\omega^2\right)  & \partial \xi=0\Leftrightarrow b=-a \\ \hline
 \end{array}
 \]
 \end{tiny}
 \caption{Modules in $E_{1}^{p,q}=H^q(\fm_{\bar{0}},\fg)\otimes\Lambda^p (\fm_{\bar1})^\ast$, $p+q\leq2$, 
 with multiplicity more than one, $0\leq d\leq5$.} \label{D:E1lwvdouble}
 \end{table} \end{center} 

We get the $E_2$-page of the spectral sequence in Table \ref{table:DE2}.
 
\begin{table}[h!]
 \[
\begin{array}{|c|c|c|c|}\hline
0 & 0 &  &\\ \hline         
\VV_{2,0}[2]\oplus\VV_{0,2}[2] & \UU_{3,1}[3]\oplus\UU_{1,3}[3]\oplus\UU_{1,1}[3]  & * & \\ \hline
0 & \UU_{1,1}[1] & \VV_{2,2}[2]\oplus\VV_{2,0}[2]\oplus\VV_{0,2}[2] & * \\  \hline
\end{array}
 \]
 \caption{The $E_2$-page: the modules $E_2^{p,q}$ for $p+q\leq3$ and $0\leq d\leq5$.}\label{table:DE2}
\end{table}

 \begin{proposition}
For the contact grading of $\fg=D(2,1;a)$ associated with $\fp_1^\I$:
 \begin{equation*}
H^{d,1}(\mathfrak{m},\mathfrak{g}) = 
\Big\{\begin{array}{lr} 0 & d\neq 1,\\
\UU_{1,1}[1] & d=1. \end{array}
 \end{equation*}
\end{proposition}

 \begin{proof}
By Proposition \ref{H0H1H2}, $H^1(\fm,\fg)=E_2^{1,0}\oplus E_3^{0,1}$. From the previous computations we get
$E_2^{1,0}=\UU_{1,1}[1]$, $E_3^{0,1}\subset\VV_{2,0}[2]\oplus\VV_{0,2}[2]$. Thus, we need to compute 
the differential $\partial:E_2^{0,1}\longrightarrow E_2^{2,0}$ on $\VV_{2,0}[2]$ and $\VV_{0,2}[2]$.

A lowest weight vector of $\VV_{2,0}[2]$ is $\varphi=[Y_2\otimes\omega^5]$. 
Since $\partial\varphi= \partial Y_2\wedge\omega^5+[Y_2]\otimes\partial\omega^5$ and 
$\partial\omega^5=s_1\omega^1\wedge\omega^4-s_1\omega^2\wedge\omega^3$, we have
 \[
\partial\varphi(e_1,e_4)=s_1[Y_2]\neq0,
 \]
and therefore $\VV_{2,0}[2]\not\subset E_3^{0,1}$ and $\VV_{2,0}[2]\not\subset E_3^{2,0}$ by 
Lemma \ref{Arep} with $s=2$.

Similarly, a lowest weight vector of $\VV_{0,2}[0]$ is $\psi=[Y_3\otimes\omega^5]$, so we have
 \[
\partial\psi(e_1,e_4)=s_1[Y_3]\neq0\\
 \]
and therefore $\VV_{0,2}[2]\not\subset E_3^{0,1}$ and $\VV_{0,2}[2]\not\subset E_3^{0,2}$.

Lastly, by Proposition \ref{H0H1H2}, we get $H^1(\fm,\fg)=E_2^{1,0}\oplus E_3^{0,1}=\UU_{1,1}[1]$.
 \end{proof}

 \begin{proposition}
For the contact grading of $\fg=D(2,1;a)$ associated with $\fp_1^\I$:
 \begin{equation*}
H^{d,2}(\mathfrak{m},\mathfrak{g}) = 
\Big\{ \begin{array}{lr} 0 & d\neq 2,\\ \VV_{2,2}[2] & d=2. \end{array}
 \end{equation*}
 \end{proposition}

 \begin{proof}
By lemma \ref{H1H2}, $H^2(\fm,\fg)=E_3^{2,0}\oplus E_3^{1,1}\oplus E_4^{0,2}$. 
The previous computations tell us $E_2^{0,2}=0$, $E_3^{1,1}\subset\UU_{3,1}[3]\oplus\UU_{1,3}[3]\oplus\UU_{1,1}[3]$, $E_3^{0,2}\subset\VV_{2,2}[2]$.
 \begin{itemize}
\item The lowest weight vector of $\UU_{3,1}[3]$ is $\varphi=[Y_2\otimes\omega^5]\otimes\omega^1$, so
 \[
\partial \varphi(e_1,e_4,e_1)= s_1[Y_2]\neq0
 \]
and therefore $\UU_{3,1}[3]\not\subset E_3^{1,1}$.
\item Similarly, $\UU_{1,3}[3]\not\subset E_3^{1,1}$
\item The highest weight vectors of $\UU_{1,1}[3]$ are of the form
 \[
\phi=a\left([H_2\otimes\omega^5]\otimes\omega^1+2[ Y_2\otimes\omega^5]\otimes\omega^2\right)+b\left([H_3\otimes\omega^5]\otimes+2[ Y_3\otimes\omega^5]\otimes\omega^3\right),
 \]
so we compute $\partial\phi(e_1,e_4,e_1)= a[H_2]+b[H_3]$.
 \end{itemize}

We can obtain the $E_3$-page in Table \ref{table:DE3}:
 \begin{table}[h!]\[
\begin{array}{|c|c|c|c|}\hline
0 & 0 &  &\\ \hline 
0 & 0 & * & \\ \hline
0 & \UU_{1,1}[1] & \VV_{2,2}[2] & * \\ \hline
 \end{array}\]
\caption{The $E_3$-page: the modules $E_3^{p,q}$ for $p+q\leq3$ and $0\leq d\leq5$.}\label{table:DE3}
 \end{table}

Lastly, since $H^2(\fm,\fg)=E_3^{2,0}\oplus E_3^{1,1}\oplus E_4^{0,2}$ by \eqref{H1H2}, 
so $H^2(\fm,\fg)=\VV_{2,2}[2]$.
 \end{proof}

This finishes the computation of generalized Spencer cohomology $H^d(\fm,\fg)$ in degrees $d\leq2$ for the 
$\fp_1^\I$ grading of $\fg=D(2,1;a)$. The other five gradings, displayed in Table \ref{SpH} of the Introduction, 
are obtained by similar but longer computations (though $H^0=\fg_{-\nu}$ is straightforward). 
We made an independent computer verification in each case.

\bigskip\bigskip

\textsc{Acknowledgment.}
We would like to thank Andrea Santi for helpful discussions and Andreu Llabr\'{e}s for his Maple package.
The research leading to these results has received funding from 
the Tromsø Research Foundation (project “Pure Mathematics in Norway”) and the UiT Aurora project MASCOT;
this publication is based upon work from COST Action CaLISTA CA21109 supported by COST 
(European Cooperation in Science and Technology). 

\end{document}